\documentclass[10pt]{amsart}

\expandafter\let\csname ver@amsthm.sty\endcsname\relax
\let\theoremstyle\relax

\title{Tangent spaces to the Teichm\"uller space 
from the energy-conscious perspective}

\author{Divya Sharma}
\address{\scriptsize{WWU M\"unster\\
         FB 10 - Mathematik und Informatik\\
         Einsteinstr.~62,
         D-48149 M\"unster, Germany}}
\email{dsharma@uni-muenster.de}

\date{May 2021}

\usepackage{amssymb} 
\usepackage{amsmath} 
\usepackage{mathrsfs}
\usepackage{amsthm}
\usepackage[shortlabels]{enumitem}
\usepackage{nicefrac}
\usepackage{tikz}
\usetikzlibrary{matrix}
\usepackage{longtable}
\usepackage{import}
\usepackage{calc}
\usepackage{xcolor}
\usepackage{graphicx}
\usepackage{tgpagella} 
\usepackage{csquotes}
\usepackage[all,color]{xy}
\usepackage{tikz}
\usepackage{tikz-cd}
\usetikzlibrary{cd}
\usetikzlibrary{quotes}
\usepackage{float}
\usepackage{nomencl}
\usepackage{collcell}
\usepackage{mathtools}

\usepackage{geometry}
\newgeometry{vmargin={35mm}, hmargin={32mm, 32mm}}

\usepackage[hypertexnames=true, bookmarksnumbered=true, 
pageanchor=false, pdfencoding=auto]{hyperref}
\hypersetup{
colorlinks,
linkcolor={magenta!100},
citecolor={blue!100}
}

\usepackage{amsthm}
\usepackage[noabbrev,nameinlink,capitalize]{cleveref}
\usepackage[figure]{hypcap}
\hypersetup{
 pdfauthor 	= Divya Sharma 
}

\crefformat{section}{\S#2#1#3} 
\crefformat{subsection}{\S#2#1#3}
\crefformat{subsubsection}{\S#2#1#3}

\newtheorem{theorem}{Theorem}[section]
\newtheorem{prop}[theorem]{Proposition}
\newtheorem{lemma}[theorem]{Lemma}
\newtheorem{coro}[theorem]{Corollary}

\theoremstyle{definition}
\newtheorem{defn}[theorem]{Definition}
\newtheorem{example}[theorem]{Example}
\newtheorem{problem}[theorem]{Open Problem}

\theoremstyle{remark}

\newtheorem{remark}[theorem]{Remark}

\newtheorem{qus}{Question}

\theoremstyle{plain}
\newtheorem{mainthmintro}{Theorem}

\newtheorem{mainprop}{Proposition}

\newtheorem{maindefn}{Definition}

\newtheorem{mainlemma}{Lemma}

\newtheorem{maincoro}{Corollary}

\newtheorem{mainopen}{Open Problem}

\newtheorem{mainremark}{Remark}

\numberwithin{equation}{section}

\def\rr{\mathbb{R}}
\def\hh{\mathbb{H}^{2}}
\def\cc{\mathbb{C}}
\def\zz{\mathbb{Z}}
\def\dd{\mathbb{D}}
\def\ss{\mathbb{S}}

\begin{document}

\begin{abstract}
Usually, the description of tangent spaces to 
the Teichmueller space $\mathscr{T}(\Sigma_{g})$ of a 
compact Riemann surface $\Sigma_{g}$ of genus $g \geq 2$ \big(which we can 
identify with the quotient space $\hh / \Gamma_{g}$ of the upper half plane $\hh$ by 
a discrete cocompact subgroup $\Gamma_{g}$ of $\mathrm{PSL}(2, \rr)$\big)
comes 
in two different flavours: the space of 
holomorphic quadratic differentials on $\Sigma_{g}$ 
which are holomorphic sections of the tensor square of the 
canonical line bundle of $\Sigma_{g}$
and the first cohomology group $H^{1}(\Gamma_{g}; \mathfrak{g})$ of the 
fundamental group $\Gamma_{g}$ of $\Sigma_{g}$ with coefficients in 
the vector space $\mathfrak{g}$
of Killing vector fields on $\hh$ (or on $\dd$), a.k.a the Lie algebra of 
$\mathrm{PSL}(2, \rr)$. In this article, we are 
concerned with connecting the above-mentioned 
descriptions using the notion of a \textit{harmonic 
vector field} on the upper half plane $\hh$ (equivalently, on $\dd$) 
that takes inspiration 
from the theory of harmonic maps 
between compact hyperbolic Riemann surfaces. As an application, we also show that 
how a harmonic vector field on $\hh$ (or on $\dd$) describes a connection on 
the universal Teichm\"uller curve. 
\end{abstract}

\maketitle

\setcounter{tocdepth}{2}
\tableofcontents

\section*{Introduction} 
\label{introduction}
In Riemann surface theory, Teichmueller theory, and the theory of moduli spaces, 
   on the one hand, we benefit a lot 
   from cross-pollination of techniques coming from sometimes disparate fields like
   topology, complex analysis, algebraic geometry, and arithmetic 
   geometry. However, on the other hand, passing from one structure/definition to 
   another is 
   quite often an arduous task, making the use of 
   different techniques simultaneously rather tricky. 
   
   Incase of a closed Riemann surface $\Sigma_{g}$ of genus $g \geq 2$, 
   to make a smooth transition from 
   \begin{displayquote}
    \textit{complex structures}. A \textit{complex structure} on $\Sigma_{g}$ is
an equivalence class of complex atlases, where two atlases, say, $\{U_{i}, 
f_{i}\}$ and $\{ V_{i}, g_{i}\}$ are equivalent iff their union forms a new complex 
atlas. We denote the set of almost 
complex structures on $\Sigma_{g}$ by $\mathscr{C}(\Sigma_{g})$ 
   \end{displayquote}
to 
    \begin{displayquote}
      \textit{hyperbolic structures}. A closed oriented surface $\Sigma_{g}$ of 
      genus $g \geq 2$
 endowed with a fixed hyperbolic metric, i.e., a Riemannian 
 metric of constant sectional curvature -1, is known as a hyperbolic surface or
 $\Sigma_{g}$ equipped with 
 a hyperbolic structure (see \cite{goldman1}, \cite{goldman2}, \cite[Chapter 5]{goldman3} for 
 equivalent definitions of hyperbolic structures on $\Sigma_{g}$),
    \end{displayquote}
we need the 
   Uniformization
   theorem (\cite{abik}, \cite{klein}). 
   From the lens of the Korn-Lichtenstein theorem\footnote{The Korn-Lichtenstein theorem is 
   same as the Newlander-Nirenberg theorem for surfaces.} (\cite{Ch55}, \cite{NN57}) 
   we watch 
   metamorphosis of 
   \begin{displayquote}
     \textit{almost complex structures}. \textit{An almost complex structure} on $\Sigma_{g}$
 is a smooth bundle endomorphism $J: T \Sigma_{g} \longrightarrow T \Sigma_{g}$ such that
  for all $x \in \Sigma_{g}$, $J_{x}^{2}=-I_{x}$ and for all non-zero
   $v \in T_{x} \Sigma_{g}$, $(v, J_{x}(v))$
   is an oriented basis for $T_{x}\Sigma_{g}.$
 Equivalently, an almost complex structure $J$ is a smooth section of the fiber bundle 
   $$\mathrm{GL}(\Sigma_{g}) \times_{\mathrm{GL}^{+}(2, \rr)} \mathrm{GL}^{+}(2, \rr) / 
   \mathrm{GL}(1, \cc) \longrightarrow \Sigma_{g},$$ where  
   $\mathrm{GL}(1, \cc)$ is the multiplicative group of 
 non-zero complex numbers embedded in the group $\mathrm{GL}^{+}(2, \rr)$ of the
real $2 \times 2$ matrices with positive determinant. We denote the set of almost 
complex structures on $\Sigma_{g}$ by $\mathscr{A}(\Sigma_{g})$. Note that  
$\mathscr{A}(\Sigma_{g})$ is endowed with the $C^{\infty}$-topology and is 
clearly contractible because the homogeneous 
space $\mathrm{GL}^{+}(2, \rr) / \mathrm{GL}(1, \cc)$ 
is contractible 
   \end{displayquote}
into 
   \begin{displayquote}
    complex structures. 
   \end{displayquote} 
   
   Usually, the problems even get worse when 
   passing from a single Riemann surface to either the parametrization space 
   $\mathscr{T}(\Sigma_{g})$ - famously known as the 
   Teichmueller space of $\Sigma_{g}$ - 
   parameterizing hyperbolic structures/complex structures/almost complex structures 
   on $\Sigma_{g}$ upto \textit{isotopy} or the
   bundles of Riemann surfaces. 
   As already mentioned, this problem is not only confined to structures but it is also 
   valid when it comes to connecting 
   different definitions and different descriptions of a mathematical object 
   in Teichmueller theory. 
   For instance, the description of the 
   Teichmueller 
   space $\mathscr{T}(\Sigma_{g})$ 
   of a closed oriented surface $\Sigma_{g}$ of genus $g \geq 2$ enjoys a 
   multifaceted viewpoint, i.e., we can view $\mathscr{T}(\Sigma_{g})$ as 
   \begin{itemize}
 \item the quotient space of the space $\mathscr{A}(\Sigma_{g})$ of 
   almost complex structures on $\Sigma_{g}$ by the action of the group 
   $\mathrm{Diff}^{+}_{0}(\Sigma_{g})$
   of orientation preserving diffeomorphisms on $\Sigma_{g}$ that are 
   isotopic to the identity. The group
   $\mathrm{Diff}_{0}^{+}(\Sigma_{g})$ acts 
 on $\mathscr{A}(\Sigma_{g})$ in the following manner
 $$ (f^{\ast}J)_{x}:= (df_{x})^{-1}J_{f(x)}df_{x}; \quad f 
 \in \mathrm{Diff}_{0}^{+}(\Sigma_{g});$$ 
   \item the quotient space of the space $\mathscr{H}(\Sigma_{g})$ of 
   Riemannian metrics of constant sectional curvature $-1$ on $\Sigma_{g}$ 
   by the action of the group $\mathrm{Diff}^{+}_{0}(\Sigma_{g})$ (\cite{fischer}, \cite{Tr92}). The 
   group $\mathrm{Diff}_{0}^{+}(\Sigma_{g})$ acts on $\mathscr{H}(\Sigma_{g})$ by pullback 
   of metrics. Note that $\mathscr{H}(\Sigma_{g}) \subset 
   \mathscr{M}(\Sigma_{g})$, where $\mathscr{M}(\Sigma_{g})$ denotes the space of 
   all Riemannian metrics on $\Sigma_{g}$;
   \item the quotient space of $\mathrm{Hom}_{0}(\Gamma_{g}, \mathrm{PSL}(2, \rr))$ 
   by the action of the Lie group $\mathrm{PSL}(2, \rr)$, where $\Gamma_{g}$ is the 
   fundamental group of $\Sigma_{g}$ and
   $\mathrm{Hom}_{0}(\Gamma_{g}, \mathrm{PSL}(2, \rr))$ is the space of 
 homomorphisms $\Gamma_{g} \longrightarrow \mathrm{PSL}(2, \mathbb{R})$
 which describe a discrete and cocompact action of $\Gamma_{g}$ on $\hh$.
   \end{itemize}
   
  The above-mentioned viewpoints are brought together in one-to-one correspondence by the 
  following commutative diagram:

  \begin{figure}[H]
  \label{differentteich}
    \begin{center}
    \[
    \xymatrixcolsep{1.5pc}
   \xymatrix{
   \mathscr{M}(\Sigma_{g}) \ar@/_2pc/[ddrr]_\varsigma & &
   \mathscr{H}(\Sigma_{g}) \ar@{_{(}->}[ll]^-{i} \ar@{->>}[r]^-{p_{1}} \ar@[red][d]^\Theta & 
   \mathscr{H}(\Sigma_{g}) / \mathrm{Diff}^{+}_{0}(\Sigma_{g})  \ar@/^1pc/[dr]^\Psi 
   \ar@{=}[d] \\
  & & \mathscr{C}(\Sigma_{g})  \ar@[red][d]^\Xi & 
   \mathscr{T}(\Sigma_{g}) \ar@{=}[d] \ar@{=}[r] & 
   \mathrm{Hom}_{0}(\Gamma_{g}, \mathrm{PSL}(2, \rr)) / \mathrm{PSL}(2, \rr) \\
  & &  \mathscr{A}(\Sigma_{g}) \ar@{->>}[r]^-{p_{2}} \ar@/^2pc/@[red][uu]^\Phi & 
    \mathscr{A}(\Sigma_{g}) / \mathrm {Diff}^{+}_{0}(\Sigma_{g})  \ar@/_1pc/[ur]_\varLambda
   }
   \]
   \end{center}
   \caption{Different perspectives to look at the Teichmueller space $\mathscr{T}(\Sigma_{g})$: 
  ``red'' arrows denote $\mathrm{Diff}^{+}_{0}(\Sigma_{g})$-equivariant bijective maps}
   \end{figure}
   In Figure \ref{differentteich}, $p_{1}$ and $p_{2}$ are the projection maps. Infact, 
   $p_{1}$ and $p_{2}$ are principal $\mathrm{Diff}^{+}_{0}(\Sigma_{g})$-bundles 
   (\cite{EE69}). 
   Clearly, $i$ is the inclusion map. 
   The map $ \varsigma$ is also clear because given a Riemannian metric $h$ on $\Sigma_{g}$, 
   we automatically have an almost complex structure on $\Sigma_{g}$ because with the 
   metric $h$, the notion of angles is clear. 
   The map $\Xi$ is an obvious (forgetful) map given by 
\begin{equation*}
     c \ni (U \subset \Sigma_{g}, \phi) \longmapsto \bigg(J_{\phi}(x):= 
     d\phi^{-1}_{x} \hat{J} d\phi_{x}, x \in U, \hat{J}:= \begin{bmatrix}
                  0& -1 \\
                  1 & 0
                 \end{bmatrix} \bigg).
\end{equation*}
The forgetful map $\Theta$ 
is a consequence of the Uniformization theorem. The continous 
map $\Xi \circ \Theta$ is a bijection. The map $\Phi$ - the inverse of $\Xi \circ \Theta$ - 
is also continous (\cite{EE69}, 
\cite{eellsscha}). One of the main ingredient of the map $\Psi$ is the \textit{holonomy} 
representation (see \cite[Section 2]{burger}). For the description of the map $\varLambda$, 
see \cite{EE69} and \cite{robbin}.

  In the literature,
   $\mathscr{T}(\Sigma_{g})$ 
  is also defined as a
   connected component of the representation variety 
   $\mathrm{Hom}(\Gamma_{g}, \mathrm{PSL}(2, \rr)) / \mathrm{PSL}(2, \rr)$ 
   (\cite{GM80}, \cite{GM88}, \cite{matsumoto}) and the universal orbifold cover 
   of the moduli space of algebraic curves $\mathfrak{M}_{g}$. 
   In their own right, these descriptions are great motivations to study 
 the Teichmueller space $\mathscr{T}(\Sigma_{g})$ in detail, 
 this article will not discuss them further. In this article, 
 $\mathrm{Hom}_{0}(\Gamma_{g}, \mathrm{PSL}(2, \rr)) / \mathrm{PSL}(2, \rr)$ will 
 be our main definition of $\mathscr{T}(\Sigma_{g})$.

 Other than the Teichmueller 
  space $\mathscr{T}(\Sigma_{g})$, there are many examples of spaces 
  in Teichmueller theory
  that enjoy a kaleidoscopic 
  picture. One famous example is tangent spaces to the Teichmueller space 
   $\mathscr{T}(\Sigma_{g})$. 
  Tangent spaces to the Teichm\"{u}ller space $\mathscr{T}(\Sigma_{g})$ 
  are best described 
using the theory of \textit{infinitesimal deformations}. 
The main slogan of the theory is to \textit{deform} a point 
in the Teichmueller space $\mathscr{T}(\Sigma_{g})$ be it 
\begin{itemize}
 \item a homomorphism $\rho$ representing $[\rho]
 \in \mathrm{Hom}_{0}(\Gamma_{g}, \mathrm{PSL}(2, \rr)) /  \mathrm{PSL}(2, \rr)$; 
 \item or a complex 
structure on $\Sigma_{g}$ 
\end{itemize}
with respect to a (real) parameter $t$
and then analyze the local structure of the corresponding spaces. 
Recall the 
Taylor expansion of a 
smooth function $f$ (on a smooth manifold $M$) around a point $x \in M$. The first order 
derivative at $x$ provides good information of $f$. In the same way, 
certain cohomology groups
provide basic and satisfactory information on deformations of a homomorphism 
$\rho \in \mathrm{Hom}_{0}(\Gamma_{g}, \mathrm{PSL}(2, \rr))$. 
Formally speaking, deformation of a homomorphism 
$\rho \in \mathrm{Hom}_{0}(\Gamma_{g}, \mathrm{PSL}(2, \rr))$ 
has the following meaning: 
we take a curve of maps $\rho_{t}$ where $\rho_{0}=\rho$ is a homomorphism, and ask
for (infinitesimal) conditions which ensure that 
this curve $\rho_{t}$ satisfies the homomorphism condition 
\begin{equation*}
 \rho_{t}(\gamma_{1}\gamma_{2})=\rho_{t}(\gamma_{1})\rho_{t}(\gamma_{2}), \quad 
\forall \gamma_{1}, \gamma_{2} \in \Gamma_{g}.
\end{equation*}
Solving $\frac{d \rho_{t}}{dt}\big|_{t=0}$ 
up to the first order 
determines a $1$-cocycle with values in the 
vector space of Killing vector fields on $\hh$, a.k.a the 
Lie algebra $\mathfrak{g}$ of $\mathrm{PSL}(2, \rr)$. As a result, 
$T_{\rho}\mathrm{Hom}_{0}(\Gamma_{g}, \mathrm{PSL}(2, \rr))$ is nothing but the space of 
$\mathfrak{g}$-valued $1$-cocycles $Z^{1}(\Gamma_{g}; \mathfrak{g}_{\mathrm{Ad}_{\rho}})$. Next, 
by considering ``trivial'' deformations $\rho_{t}$ of $\rho$ given by conjugation via elements 
of $\mathrm{PSL}(2, \rr)$ and solving the above-mentioned homomorphism condition up to the first-order 
determines a $1$-coboundary 
$c \in B^{1}(\Gamma_{g}; \mathfrak{g}_{\mathrm{Ad}_{\rho}})$. Hence, 
$$T_{[\rho]}\mathrm{Hom}_{0}(\Gamma_{g}, \mathrm{PSL}(2, \rr))/\mathrm{PSL}(2, \rr) 
\cong
H^{1}(\Gamma_{g}; \mathfrak{g}_{\mathrm{Ad}_{\rho}}).$$
Therefore, 
$H^{1}(\Gamma_{g}; \mathfrak{g}_{\mathrm{Ad}_{\rho}})$ serves as the cohomological description 
of tangent spaces to the Teichmueller space $\mathscr{T}(\Sigma_{g})$. 
The space of 
infinitesimal deformations of a complex structure on $\Sigma_{g}$ is parametrized 
by the space $\mathrm{HQD}(\Sigma_{g})$ 
of \textit{holomorphic quadratic differentials} on $\Sigma_{g}$ (\cite{Kodaira}, 
\cite{Morrow}, \cite[Chapter 1]{wolpert2}), 
where a holomorphic quadratic differential is a holomorphic 
 section of $Q_{\Sigma_{g}}$, 
 the tensor square of the canonical line bundle $K_{\Sigma_{g}}$ of $\Sigma_{g}$. 
 Hence, the analytic description of tangent spaces to the Teichmueller space
 $\mathscr{T}(\Sigma_{g})$ is given by $\mathrm{HQD}(\Sigma_{g})$. 
 
 So, the main aim of this article is to construct  
explicit maps from $\mathrm{HQD}(\Sigma_{g})$ to 
$H^{1}(\Gamma_{g}; \mathfrak{g}_{\mathrm{Ad}\rho})$ and vice-versa, i.e., 
 \begin{equation}
\label{mainthing}
\xymatrix{ 
\mathrm{HQD}(\Sigma_{g}) \ar[r]^-{?} &
H^{1}(\Gamma_{g}; \mathfrak{g}_{\mathrm{Ad}\rho}) \ar[l]^-{?} 
}
\end{equation}

Now, we can ask ourselves the following question: what recipes are we going to use 
in the 
construction of maps from $\mathrm{HQD}(\Sigma_{g})$ to 
$H^{1}(\Gamma_{g}; \mathfrak{g}_{\mathrm{Ad}\rho})$ and vice-versa?

Since the inception of Teichmueller's theorems, the use of 
   \textit{quasiconformal maps}
   in classical Teichmueller theory is prevalent. 
   However, in this thesis, 
   we don't focus much on quasiconformal maps. We take an 
   unconventional road that \textit{minimizes energy} to connect the above-mentioned  
   descriptions
   of tangent spaces to the Teichmueller 
   space $\mathscr{T}(\Sigma_{g})$. Our essential recipe will be the notion of 
   a \textit{harmonic vector field} on the upper half plane $\hh$ or the 
   Poincar\'{e} disk $\dd$ in constructing maps from $\mathrm{HQD}(\Sigma_{g})$ to 
$H^{1}(\Gamma_{g}; \mathfrak{g}_{\mathrm{Ad}\rho})$ and vice-versa. 

The notion of a \textit{harmonic vector field} on $\hh$ (or on $\dd$) takes inspiration 
from the definition 
(see Definition \ref{harmonicdefn}) of a 
\textit{harmonic map} $\phi: \Sigma_{1} 
\longrightarrow \Sigma_{2}$ 
between Riemann surfaces equipped with conformal metrics. 
Harmonic maps are critical 
points of the energy functional 
$$E(\phi) = \int_{\Sigma_{1}} \| d \phi \|^{2}  d\mu,$$
where $ \| \cdot \|$ is the Hilbert-Schmidt norm  
and $d\mu$ is the measure on $\Sigma_{1}$ determined by the Riemannian metric 
on $\Sigma_{1}$. The integrand is also known as the 
\textit{energy density} (see (\ref{energydensity})). Equivalently, 
harmonic maps 
satisfy the \textit{Euler-Lagrange partial differential equations} associated with 
the energy functional 
(see (\ref{euler})). These PDEs are non-linear 
and elliptic. Harmonic maps exist in the homotopy class of any diffeomorphism when the 
target surface is equipped with a strictly negatively curved metric, and are 
unique (\cite{ells}, \cite{hart}). Harmonic maps are related to 
holomorphic quadratic 
differentials intimately, 
hence play 
an important role in Teichmueller theory. This relation arises from the
fact that 
\begin{displayquote}
 a diffeomorphism $\phi: (\Sigma_{1}, \sigma) 
\longrightarrow (\Sigma_{2}, \rho)$ between two Riemann surfaces equipped with conformal metrics 
is harmonic 
iff the quadratic differential $(\phi^{\ast}\rho)^{(2, 0)}$ 
on the source surface $\Sigma_{1}$
is holomorphic (see Example \ref{harmimplieshol} and 
 \cite[Lemma 1.1]{jost1}). \hfill $(\dagger)$
\end{displayquote}

The use of harmonic maps in Teichmueller theory goes all the way back to 
Gerstenhaber and Rauch's program (\cite{Gerstenhaber}, \cite{Gerstenhaber1}, \cite{Reich}) to prove 
Teichmueller's Theorems (\cite{teichmu}) using harmonic maps. 
In order to state our main results, 
we need to define a harmonic vector field on the upper half plane 
$\hh$ or the Poincar\'{e} disk $\dd$: let $U$ be an open subset of $M$, where 
$M$ is either the upper half plane $\hh$ or the Poincar\'{e} disk $\dd$. 
Let $\{\phi_{t}\}_{t \in [0, \epsilon)}$ be a 
smooth family of smooth  maps
$$\phi_{t} : U \longrightarrow M$$
where $\phi_{0}$ is the inclusion. Then $\xi= \frac{d \phi_{t}}{dt} \vert_{t=0}$ is 
a vector field on $U$.
\begin{maindefn}[Definition \ref{defn12}]\label{maindefn1}
 \normalfont The vector field $\xi$ on $U$ is harmonic if there exists a 
 smooth family of smooth maps
 $\{\phi_{t}: U \longrightarrow M\}_{t \in [0, \epsilon)}$ which satisfies 
 the following:
 \begin{enumerate}
\item $\phi_{0}$ is the inclusion map,
\item $\displaystyle\frac{d \phi_{t}}{dt}\Big\vert_{t=0} = \xi$,
\item $\forall x \in U:~\displaystyle\frac{d}{dt}\Big\vert_{t=0}~\tau(\phi_{t})(x)=0\,$, where 
$\tau$ is the \textit{tension field} (see Definition \ref{tensionfieldharmonic}).
\end{enumerate}
\end{maindefn}
An infinitesimal version of $(\dagger)$ is given by the following: 
\begin{mainprop}[Proposition \ref{thm2}]
\label{firstprop}
 A smooth vector field $\xi$ on $\hh$ or on $\dd$
 is harmonic iff $\big(\mathcal{L}_{\xi}\textbf{g}_{\hh}\big)^{(2, 0)}$ or 
 $\big(\mathcal{L}_{\xi}\textbf{g}_{\dd}\big)^{(2, 0)}$
 is holomorphic. 
\end{mainprop}
Our first main theorem is based on the above 
Proposition and the fact that a holomorphic 
vector field on $U \subset \hh$ is a harmonic vector field on $U \subset \hh$. 
\begin{mainthmintro}[Theorem \ref{sess}]\label{maintheorem1}
Let $\mathcal{HOL}$ denote the sheaf of holomorphic vector fields on $\hh$, $\mathcal{HARM}$ 
denote the sheaf of harmonic vector fields on $\hh$ and 
$\mathcal{HQD}$ denote the sheaf of holomorphic quadratic differentials on $\hh$. 
Then the following sequence of sheaves
\begin{equation}
\label{maintheoremequation}
 \xymatrix{ 
\mathcal{HOL} \ar[r]^-{\alpha} &
\mathcal{HARM} \ar[r]^-{\beta} &
\mathcal{HQD} 
}
\end{equation}
is a short exact sequence of sheaves on $\hh$. In (\ref{maintheoremequation}), 
$\alpha$ is the inclusion map and $\beta$ is 
 given by the formula in Proposition \ref{firstprop}. 
\end{mainthmintro}
\begin{mainremark}
 \normalfont
 Theorem \ref{maintheorem1} is also valid if we replace $\hh$ with $\dd$. 
\end{mainremark}

\begin{mainremark}
 \normalfont
 (\ref{maintheoremequation}) is related to the following short exact sequence of sheaves 
 in classical Teichmueller theory 
 \begin{equation}
 \label{maintheoremequation3}
 \xymatrix{
  0 \ar[r] & \mathcal{S}_{\mathrm{Hol}}\big(T \mathbb{H}^{2}\big) 
  \ar[r]^-{i} 
  & \mathcal{S} \big(T \mathbb{H}^{2}\big)
  \ar[r]^-{\frac{\partial}{\partial \bar{z}}} & \mathcal{BEL} \ar[r] & 0
  }
 \end{equation}
 where $\mathcal{S}_{\mathrm{Hol}}\big(T \mathbb{H}^{2}\big)$ is the sheaf 
 of holomorphic vector fields on $\hh$, $\mathcal{S} \big(T \mathbb{H}^{2}\big)$ is 
 the sheaf of smooth vector field on $\hh$, and $\mathcal{BEL}$ is the 
 sheaf of \textit{Beltrami differentials} on $\hh$. (\ref{maintheoremequation3}) is 
 a special case of a 
more general construction called the  
\textit{Dolbeault resolution} of the sheaf 
$\mathcal{S}_{\mathrm{Hol}}\big(T \mathbb{H}^{2}\big)$. See  
\cref{genesis} for more details. 
\end{mainremark}

Our next main Theorem is about proving the global surjectivity of the map $\beta$ in 
(\ref{maintheoremequation}) in Theorem \ref{maintheorem1}. 

\begin{mainthmintro} [Theorem \ref{thmglobharmvf} + Theorem \ref{boundary}] \label{maintheorem2}
Let $q=f(z)dz^{2}$ be a 
holomorphic quadratic differential on $\hh$. Suppose that 
$q$ satisfies the following boundedness
conditions 
\begin{enumerate}
 \item $q$ is bounded in the hyperbolic metric $\textbf{g}_{\hh}$, i.e.,
\begin{equation*}
 \Arrowvert q \Arrowvert_{\textbf{g}_{\hh}} = \arrowvert f(z) \arrowvert  
 \Arrowvert dz^{2} \Arrowvert_{\textbf{g}_{\hh}} \leq D,
\end{equation*}
where $\Arrowvert dz^{2} \Arrowvert_{\textbf{g}_{\hh}}= \Im(z)^{2}$ and $D$ is a 
positive real number.
\item The first and second covariant derivative of $q$ w.r.t $\nabla$, the
linear connection on $T^{\ast} \hh \otimes_{\cc} T^{\ast} \hh$,
are bounded in the hyperbolic metric $\textbf{g}_{\hh}$. 
\end{enumerate}
Then there exists a harmonic vector field $\xi^{\mathrm{reg}}$ on $\hh$ 
such that $\beta(\xi^{\mathrm{reg}})=q$, where $\beta$ is introduced in 
Theorem \ref{maintheorem1}. An explicit formula is 
\begin{equation*}
  \xi^{\mathrm{reg}}(z) = \lim_{c \to \infty} \Bigg( \xi_{c}(z) - 
  \bigg( \xi_{c}(\iota) + \frac{\partial \xi_{c}}{\partial z}\bigg|_{z=\iota}\cdot 
  (z-\iota)\bigg) \Bigg),
  \end{equation*}
where 
$$\xi_{c}(z)= \bigg(  \int_{y_{\ast}(z)}^{c} \iota \zeta^{2} 
\overline{f(\bar{z}+2 \iota \zeta)} d \zeta \bigg) \eta(z)$$ and 
$c$ is a positive real number. The harmonic vector field $\xi^{\mathrm{reg}}$
transformed from $\hh$ to the open unit disc $\dd$ by 
the Cayley transform $C$
extends to a continuous vector field, say $\chi$, 
on $\overline{\dd}$ defined as follows: \begin{equation*}
 \label{maintheoremequation1}
 \chi(C(z)) = \begin{cases}
            C_{\ast}(\xi^{\mathrm{reg}}(z)) & \quad z \in \hh \\
            C_{\ast}(\xi^{\mathrm{reg}}(z)) & \quad z \in \partial \hh \setminus \{\infty\} \\
            0                     & \quad z = \{ \infty \}
          \end{cases}
          \end{equation*}
          where $C_{\ast}(\xi^{\mathrm{reg}}(z))$ is 
the pushforward of $\xi^{\mathrm{reg}}(z)$ by the Cayley transform $C$. 
\end{mainthmintro} 
\begin{mainremark}
 \normalfont
 We have introduced a simple terminology $\mathrm{reg}$ short 
for ``regularisation'' to characterise our required harmonic vector field.
\end{mainremark}

\begin{mainremark}
 \normalfont
 The global surjectivity of the map $\beta$ in Theorem \ref{maintheorem1}
 is proven independently by 
S. Wolpert in \cite[Section 2]{wolpert}. See the beginning
of \textbf{\cref{harmonicexplicit}} in \textbf{\cref{chapter3}}. 
\end{mainremark}
\smallskip

Theorem \ref{maintheorem2} implies that the coboundary $\delta \chi$ 
$$\chi \longmapsto \big(\gamma \longmapsto
\chi(\gamma) \gamma^{-1} - \chi \big), \quad \forall \gamma \in \Gamma$$
where $\Gamma$ is a discrete cocompact subgroup of $\mathrm{Isom}^{+}(\dd)$, defines  
a $1$-cocycle with values in the vector space 
$\mathrm{HOL}$ of holomorphic vector fields on $\dd$. Note that we view 
$\chi$ as a $0$-cocycle with values in the vector space of harmonic vector 
fields on $\dd$. The following results ensure that $\delta \chi$ 
is a $1$-cocycle with values in the vector space of Killing vector fields on $\dd$:

\begin{mainthmintro}
[Theorem \ref{infinitedimproblem} + Theorem \ref{summarytheo}]
 \label{chap3theorem1}
Given a holomorphic quadratic differential $q=fdz^{2}$ on 
the Poincar\'{e} disk $\dd$ which satisfies the following boundedness conditions:
\begin{enumerate}
 \item $q$ is bounded in the hyperbolic metric on $\dd$, i.e., 
 $$\lVert q \rVert_{\textbf{g}_{\dd}} \leq D, $$ 
 where $D$ is a positive real number. 
 \item The first and the second covariant derivative of $q$ w.r.t 
 the linear connection on $T^{\ast}\dd \otimes_{\cc} T^{\ast} \dd$ are bounded in 
 $\textbf{g}_{\dd}$. 
\end{enumerate}
Then there exists a harmonic vector field $\chi$ on $\dd$ 
which admits an $L^{2}$-extension to the closed unit disk $\overline{\dd}$ 
such that 
$(\mathcal{L}_{\chi}\textbf{g}_{\dd})^{(2, 0)}= q$. 
Moreover, the restriction of that extension to the boundary circle $\ss^{1}$ 
is tangential and $\chi$ is unique upto the addition 
of holomorphic vector fields on $\dd$ which extend tangentially to the boundary 
circle $\ss^{1}$. Also, $\chi$ is 
unique upto the addition of the vector space 
$\mathfrak{g}$ of Killing vector fields on $\dd$. 
\end{mainthmintro}
\begin{maincoro}[Corollary \ref{summarytheocor}]
\label{chap3coro1}
Let $\Gamma$ denote a subgroup of $\mathrm{Isom}^{+}(\dd)$, where 
$\mathrm{Isom}^{+}(\dd)$ is 
the group of orientation preserving 
isometries of $\dd$. 
 If $q=fdz^{2}$ and $\chi$ are related as in Theorem \ref{chap3theorem1} and if in 
 addition to (1) and (2) in Theorem \ref{chap3theorem1}, 
 $q$ is $\Gamma$-invariant, i.e., 
$$f(\gamma(z))\gamma'(z)^{2}=f(z), \quad \forall \gamma \in \Gamma, z \in \dd,$$
then $\delta \chi$ defined by 
$$ \gamma \longmapsto \chi(\gamma) \gamma'^{-1}-\chi, \quad \forall \gamma \in \Gamma$$
is a $1$-cocycle $c$ for the group $\Gamma$ with coefficients in 
the Lie algebra $\mathfrak{g}$
of $\mathrm{Isom}^{+}(\dd)$ and its cohomology class $[c]$ depends only on $q$. 
\end{maincoro}

\begin{maincoro}[Corollary \ref{onewaymap}]
\label{corosecondlast}
 Let $\Gamma$ be a discrete cocompact subgroup 
 of $\mathrm{Isom}^{+}(\dd)$. Then 
 we have an injective mapping 
$$\varPhi: \mathrm{HQD}(\dd, \Gamma) \longrightarrow H^{1}(\Gamma; \mathfrak{g}) $$
$$q \longmapsto [c],$$
where $\mathrm{HQD}(\dd, \Gamma)$ denotes the 
vector space of $\Gamma$-invariant holomorphic quadratic 
differentials on $\dd$ and $c= \delta \chi$. 
\end{maincoro}
\smallskip
 
To construct an inverse of $\varPhi$ in Corollary \ref{corosecondlast}, 
we first 
construct a smooth vector field $\psi$ on $\dd$ such that 
$\delta \psi =c$, where $c$ is a $1$-cocycle $c$ representing 
$[c] \in H^{1}(\Gamma; \mathfrak{g})$. And then, we show that 
$\psi$ admits an $L^{2}$-extension 
to the closed unit disk $\overline{\dd}$ whose restriction 
to the boundary circle $\ss^{1}$ is tangential. 
This construction relies on 
the existence of a $\Gamma$-invariant
partition of unity on $\dd$. See \textbf{\cref{cohomtoanalsection1}}
in \textbf{\cref{cohomtoanal}}.
\begin{mainlemma}[Lemma \ref{partition}]
\label{lemmaintro}
 There exists a smooth function $\varphi$ on $\dd$ such that
 \begin{enumerate}
  \item $0 \leq \varphi \leq 1$.
  \item For each $z \in \dd$, there is a neighborhood $U$ of $z$ and a finite subset 
  $S$ of $\Gamma$ such that $\varphi=0$ on $\gamma(U)$ for every $\gamma \in 
 \Gamma-S$.
 \item $\sum_{\gamma \in \Gamma} \varphi(\gamma(z))=1$ on $\dd$.
 \end{enumerate}
\end{mainlemma}
\begin{mainremark}
  \normalfont
 We suspect that Lemma \ref{lemmaintro} is a simpler version 
 of results on \textit{Kleinian groups} (see \cite{Kra}).
 \end{mainremark}

\begin{mainlemma}[Lemma \ref{continuousvector}]
\label{lemmaintro1}
Given any $[c] \in H^{1}(\Gamma; \mathfrak{g})$ 
we set 
$$\psi(z)= - \sum_{\gamma \in \Gamma} \varphi(\gamma(z)) 
c_{\gamma}(z), \quad z \in \dd,$$
where $\varphi$ is introduced in Lemma \ref{lemmaintro}. 
$\psi$ is a $C^{\infty}$-vector field on $\dd$ such that $\delta \psi = c$. 
\end{mainlemma}

\begin{maincoro}[Corollary \ref{shortcoro}]
\label{maincoro2}
 $\psi$ in Lemma \ref{lemmaintro1} admits a unique $L^{2}$-extension 
to the closed unit disk $\overline{\dd}$ whose restriction 
$\psi^{\sharp}$ to the boundary 
circle $\ss^{1}$ is tangential.
\end{maincoro}

\begin{mainremark}
 \normalfont
 The above-mentioned construction of a vector field on the boundary circle $\ss^{1}$ 
 from a cocycle $c$ representing $[c] \in H^{1}(\Gamma; \mathfrak{g})$ is in 
 the spirit of \textit{universal Teichmueller theory}. See \cite{fletcher}, 
 \cite{gardiner}, 
 \cite{lehto1}, 
 \cite{lehto2}, \cite{mar} for more details. 
\end{mainremark}
For the construction of $\psi$ in Lemma \ref{lemmaintro1}, 
we can either use the $\Gamma$-invariant partition 
 of unity method or the difficult theory of \textbf{\cref{chapter3}} and 
 \textbf{\cref{analtocohom}} which produces a harmonic solution. 
 Lemma \ref{lemmaintro1} 
 is valid for all of these but the construction of an $L^{2}$-extension 
 of $\psi$ to $\overline{\dd}$ 
 relies on the existence of harmonic vector fields. Therefore, 
 it is worth 
 asking the following: 
\begin{mainopen}[Open Problem \ref{openprob1}]
  Is there 
 a more direct way of proving Corollary \ref{maincoro2} which does not 
 take harmonicity into account? 
\end{mainopen}
The final results of this article are based on 
the reincarnation (see \textbf{\cref{invariancepoisson}}) and 
adaptation of the 
\textit{Poisson integral formula}
in the case of continuous tangential vector fields on $\ss^{1}$. 
First, we construct 
a harmonic vector field on the open unit disk $\dd$ from a continuous 
tangential vector field $X$ on $\ss^{1}$. Note that 
a continuous tangential vector field $X$ on $\ss^{1}$ can be written as 
$X=fY$ where $f$ is a
real-valued continuous function on $\ss^{1}$ and $Y$ is the norm $1$ tangential vector 
field on $\ss^{1}$ given by $z \longmapsto \iota z$. 
\begin{mainthmintro}[Theorem \ref{kernelvectorfieldisharmonic}]
\label{maintheorem4}
Let $\mathcal{S}_{C^{0}}(T\ss^{1})$ be the Banach space 
of (tangential) continuous vector fields on 
$\ss^{1}$ and $\mathcal{S}_{C^{0}}(T\dd)$ be the 
space of continuous vector fields on the 
open disk $\dd$. A linear map 
 $$\mathcal{F}: \mathcal{S}_{C^{0}}(T\ss^{1}) \longrightarrow 
\mathcal{S}_{C^{0}}(T\dd)$$
is given by the normalized convolution 
$$\mathcal{F}(X) = f \ast \textbf{K}, $$
where $\textbf{K}$ is the Poisson Kernel vector field given by 
$$\textbf{K}(z) =  
\frac{\iota (1-|z|^{2})^{3}}{|1-\bar{z}|^{2} \cdot (1-\bar{z})^{2}}.$$
 Moreover, $\mathcal{F}(X)$ is a harmonic vector field on the open unit disk $\dd$. 
\end{mainthmintro}
\begin{mainlemma}[Lemma \ref{finaltheorem}]
 \label{mainlemma3}
$\mathcal{F}(X)$ and $X$ make up a continuous
 vector field on the the closed unit disk $\overline{\dd}$.  
\end{mainlemma}

We adapt Lemma \ref{mainlemma3} in the case of tangential $L^{2}$-vector 
fields on $\ss^{1}$ as follows: 
\begin{maincoro}[Corollary \ref{lastcoro}]
\label{lastcorointro}
 For an $L^{2}$-tangential vector field $X$ on $\ss^{1}$, $X$ is an $L^{2}$-boundary 
 extension of the smooth vector field $\mathcal{F}(X)$ on the open unit disk $\dd$. 
\end{maincoro} 
\begin{mainremark}
 \normalfont
 We suspect that Corollary \ref{lastcorointro} is an infinitesimal version of 
 the problem of finding harmonic extensions of quasiconformal maps (from $\ss^{1}$ to 
 itself) to the open unit disk $\dd$ or the upper half plane $\hh$. See \cite{hardt} 
 for more details. 
\end{mainremark}
We have not shown that there exists a \textit{unique} harmonic
extension of a tangential $L^{2}$-vector field $X$ on $\ss^{1}$ to the closed unit
disk $\overline{\dd}$. 
And this brings us to our second open problem: 
\begin{mainopen}[Open Problem \ref{openprob2}]
\label{openprobintro2}
 Given a tangential $L^{2}$-vector field $X$ on the boundary circle $\ss^{1}$, 
 does there exist a unique harmonic extension to the closed unit disk
$\overline{\dd}$? 
\end{mainopen}
From Theorem \ref{maintheorem4} and Corollary \ref{lastcorointro}, we get 
the following result:
\begin{mainthmintro}[Theorem \ref{lastsectiontheo}]
\label{maintheorem5}
Let $\Gamma$ be a discrete cocompact subgroup of $\mathrm{PSU}(1, 1)$. 
 For every cocycle $c$ representing a cohomology class $[c] \in  
 H^{1}(\Gamma; \mathfrak{g})$, there exists a smooth vector field 
 $\psi$ on the open unit disk $\dd$ 
 such that $c= \delta \psi$. 
Moreover, any such $\psi$ admits an 
$L^{2}$-extension to $\overline{\dd}$ whose restriction $\psi^{\sharp}$ to the 
boundary circle $\ss^{1}$ is tangential. 
There exists a homomorphism
\begin{equation*}
\label{mainmap2}
 \begin{split}
  \varPsi: H^{1}(\Gamma; \mathfrak{g}) & \longrightarrow 
  \mathrm{HQD}(\dd, \Gamma) \\
  [c] & \longmapsto 
  \big(\mathcal{L}_{\mathcal{F}(\psi^{\sharp})}\textbf{g}_{\dd} \big)^{(2, 0)},
 \end{split}
\end{equation*}
  where the map $\mathcal{F}$ is introduced in Theorem \ref{maintheorem4}
  and 
  $\mathcal{F}(\psi^{\sharp})$ is a harmonic vector field on the open disk $\dd$.
\end{mainthmintro}

\begin{maincoro}[Corollary \ref{lastsectioncorol}]
\label{maincoro5}
$$ \varPhi \circ \varPsi  = \mathrm{Id},$$
where $\varPhi$ is defined in   
Corollary \ref{corosecondlast} and $\varPsi$ is defined in Theorem 
\ref{maintheorem5}. 
\end{maincoro}

\subsection*{Organisation of the article} 
The main goals of \textbf{\cref{beginning}} are to 
  gather some necessary results, prove that 
   the Teichmueller space $\mathscr{T}(\Sigma_{g})$ is a $6g-6$ dimensional 
   manifold using techniques from 
  differential topology, and discuss briefly about tangent spaces to 
  the Teichmueller space $\mathscr{T}(\Sigma_{g})$. 
  We have attempted to follow a coherent narrative. The reader who is familiar with these 
  notions can skip \cref{beginning}. \textbf{\cref{chapter3}} is dedicated 
  to establishing the notion of a harmonic vector field on $\hh$ (or on $\dd$) and
  proving Proposition \ref{firstprop}, 
Theorem \ref{maintheorem1}, and Theorem \ref{maintheorem2}. It also discusses 
the main advantages of the method which is used in \textbf{\cref{chapter3}} in proving 
Theorem \ref{maintheorem2} over Scott Wolpert's method. 
In \textbf{\cref{analtocohom}}, we ensure that we get
an explicit map 
from the vector space of $\Gamma$-invariant holomorphic quadratic differentials 
$\mathrm{HQD}(\dd, \Gamma)$ on $\dd$ to $H^{1}(\Gamma; \mathfrak{g})$
by using the theory of 
$L^{2}$-vector fields on $\ss^{1}$, where $\Gamma$ denotes a discrete cocompact subgroup of $\mathrm{PSU}(1, 1)$ and 
$\mathfrak{g}$ denotes the Lie algebra of $\mathrm{PSU}(1, 1)$.
One of the main actors in \textbf{\cref{analtocohom}} is the notion of a 
\textit{tangential $L^{2}$-vector field} on $\ss^{1}$ (see 
Definition \ref{defnoftan} and Example \ref{examoftan}). 
In \textbf{\cref{analtocohom}}, we prove Theorem \ref{chap3theorem1}, Corollary 
\ref{chap3coro1}, and Corollary \ref{corosecondlast}. 
\textbf{\cref{cohomtoanal}} is dedicated to constructing a map in 
the other direction in 
(\ref{mainthing}), i.e., from the cohomological description of tangent spaces 
to the analytic description of tangent spaces to the Teichmueller space 
$\mathscr{T}(\Sigma_{g})$. 
In \textbf{\cref{cohomtoanal}}, we prove Lemma \ref{lemmaintro}, Lemma \ref{lemmaintro1}, 
Corollary \ref{maincoro2}, Theorem \ref{maintheorem4}, Lemma \ref{mainlemma3}, Corollary 
\ref{lastcorointro}, Theorem \ref{maintheorem5}, and Corollary \ref{maincoro5}. 
And, in \textbf{\cref{connectionuniversal}}, we show 
that how we can describe a connection on the \textit{universal Teichmueller curve} using the notion 
of a harmonic vector field on $\dd$ developed in \textbf{\cref{chapter3}}, 
\textbf{\cref{analtocohom}}, and \textbf{\cref{cohomtoanal}}. \cref{genesis} sheds light on 
how (\ref{maintheoremequation}) and (\ref{maintheoremequation3}) relate to each other. 
 
\subsection*{Acknowledgements}
This article is based on the author Ph.D. thesis work. The author
is immensely indebted to Michael Weiss for his constant support and encouragement. 
The author would also like to thank Scott Wolpert for many enlightening correspondences 
and for discussing some of his work that has been used in this article. 
The author thesis work was supported 
 by the Alexander von Humboldt Professorship of Michael Weiss (2012-2017) and 
the Ada Lovelace Research Fellowship provided by the Deutsche Forschungsgemeinschaft 
under Germany's Excellence Strategy EXC 2044-390685587,
Mathematics M\"unster: Dynamics-Geometry-Structure. 
\section{Preliminaries}
\label{beginning} 
   \subsection{Some facts from hyperbolic geometry}
   \label{factshyperbolic}
    The upper half plane $\hh$ with the metric 
   $\textbf{g}_{\hh} = \frac{dx^{2}+dy^{2}}{y^{2}}$ and the Poincar\'{e} disk $\dd$ with the 
   metric $\textbf{g}_{\dd}= \frac{4dx^{2}+dy^{2}}{(1-(x^{2}+y^{2}))^{2}}$ are the common 
   models for the hyperbolic plane. Semicircles and half lines orthogonal to $\rr$ are the geodesics in the upper half plane model
$\hh$. In the Poincar\'{e} disk model $\dd$, if two points $z_{1}$ and $z_{2}$
are on the same diameter 
then the geodesic from $z_{1}$ to $z_{2}$ is the Euclidean line segment 
joining them, otherwise the geodesic is 
the arc of circle, orthogonal to 
$\mathbb{S}^{1}$. Both $\hh$ and $\dd$ have curvature $-1$ 
w.r.t $\textbf{g}_{\hh}$ and $\textbf{g}_{\dd}$. 
Both $\textbf{g}_{\hh}$ and $\textbf{g}_{\dd}$ are 
invariant under 
$$\mathrm{Aut}(\hh) =\{f \in \mathrm{Aut}(\overline{\cc})| f(\hh)=\hh \},$$ 
where $\mathrm{Aut}(\overline{\cc})$ is the automorphism group of the Riemann sphere 
$\overline{\cc}$, 
and 
$$\mathrm{Aut}(\dd) = \{f \in \mathrm{Aut}(\overline{\cc})| f(\dd)=\dd \}.$$ 
Note that
$\mathrm{Aut}(\hh) \cong \mathrm{PSL}(2, \mathbb{R}) \cong
\mathrm{Isom}^{+}(\hh)$, 
where $\mathrm{Isom}^{+}(\hh)$ is the group of 
orientation preserving isometries of $\hh$. Every element of
$\mathrm{Isom}^{+}(\hh)$ has a form
 $\gamma(z)=\frac{az+b}{cz+d},$
 where $a, b, c, d \in \mathbb{R}$ with $ad-bc=1$. We classify 
 elements of $\mathrm{PSL}(2, \rr)$ based 
 on an extremal problem on hyperbolic translation length as follows: 
 for every $\gamma \in \mathrm{PSL}(2, \rr)$ except the identity element, set 
 $$\alpha(\gamma) = \inf_{z \in \hh} d_{\hh}(z, \gamma(z)),$$ 
where $d_{\hh}(-, -)$ denotes 
 the hyperbolic distance, then 
$\gamma$ is 
 \textit{elliptic} if $\alpha(\gamma)=0$ and 
 there exists a point $z \in \hh$ with $\alpha(\gamma)= d_{\hh}(z, \gamma(z))$. 
 In other words, $z$ is a 
 fixed point of $\gamma$; 
$\gamma$ is \textit{parabolic} if $\alpha(\gamma)=0$ 
 but there exists no point $z \in \hh$ with $\alpha(\gamma)= d_{\hh}(z, \gamma(z))$;  
$\gamma$ is
 \textit{hyperbolic} if $\alpha(\gamma) > 0$ and there 
 exists a point $z \in \hh$ with $\alpha(\gamma)= d_{\hh}(z, \gamma(z))$. 
 Since $\hh$ is isometric to $\dd$, normal forms of above elements 
 are given as follows: 
 any elliptic element is conjugate to a rotation 
 $z \longmapsto \lambda z$ in $\mathrm{Aut}(\dd)$, 
 for some $\lambda$ with $|\lambda|=1$; 
 any parabolic element is 
 conjugate to either $z \longmapsto z+1$ or to $z \longmapsto z-1$ in $\mathrm{Aut}(\hh)$, 
 and these maps are not conjugate to each other; 
 any hyperbolic element is conjugate to $z \longmapsto \lambda z$ in $\mathrm{Aut}(\hh)$, 
 where $\lambda > 1$. 
Since elements of $\mathrm{PSL}(2, \rr)$ have matrix representations, 
 they are also classified by $\mathtt{trace}$, i.e., for a non-identity $\gamma \in 
 \mathrm{PSL}(2, \rr)$ the following holds: 
$\gamma$ is parabolic iff 
 $\mathtt{trace}^{2}(\gamma)=4$; $\gamma$ is elliptic iff 
 $0 \leqq \mathtt{trace}^{2}(\gamma) < 4$; $\gamma$ is hyperbolic iff 
 $\mathtt{trace}^{2}(\gamma)>4$. 
\cite{Beardon} and \cite{Iverson} are great references to absorb 
different flavours of hyperbolic geometry.
\subsection{The Teichm\"{u}ller space, a kaleidoscopic view}
\subsubsection{Classical definition}
\label{classical}
We choose a basepoint $x_{0} \in \Sigma_{g}$. 
The fundamental group 
$\pi_{1}(\Sigma_{g}, x_{0})$ is generated by the homotopy classes
$[a_{1}], [b_{1}], \ldots, [a_{g}], [b_{g}]$ induced from 
simple closed curves $a_{1}, b_{1}$, $\ldots, a_{g}, b_{g}$ with base point $x_{0}$ satisfying 
the following relation: 
$$[[a_{1}], [b_{1}]] \cdots [[a_{g}], [b_{g}]] = 1,$$ where $1$ is the 
unit element. 
We denote the fundamental group $\pi_{1}(\Sigma_{g}, x_{0})$ by 
$\Gamma_{g}$. By abuse of notation, we denote the generators of $\Gamma_{g}$ by 
$a_{1}, b_{1}, \ldots, a_{g}, b_{g}$ satisfying the fundamental relation 
$[a_{1}, b_{1}] \cdots [a_{g}, b_{g}] = 1$. 
From the Uniformization theorem, 
   $\Gamma_{g}$ is isomorphic to a discrete cocompact subgroup of 
    $\mathrm{PSL}(2, \rr)$. Before giving the classical definition of the Teichmueller space, 
    we describe elements of $\Gamma_{g}$. 
    \begin{prop}[\cite{katok}]
\label{hyperbolicelements}
 Every non-identity element of 
 $\Gamma_{g}$ is hyperbolic. 
\end{prop}
\begin{proof}
 We prove the proposition by contradiction. Assume that $\gamma \in 
 \Gamma_{g}-\{1\}$ 
 is either parabolic or elliptic. Note that $\Gamma_{g}$ acts freely on $\hh$ and hence 
 cannot have elliptic elements. Now, assume
 that $\gamma \in \Gamma_{g}-\{\ 1 \}$ is a
parabolic element. Since every parabolic element of $\mathrm{PSL}(2, \rr)$ is conjugate in 
$\mathrm{PSL}(2, \rr)$ to either $z \longmapsto z+1$ or $z \longmapsto z-1$ 
(see \textbf{\cref{factshyperbolic}}), 
we work with $\gamma(z)=z+1$ for the rest of the proof. 
Let $a$ be a positive real number. 
Let us
denote the image of the segment joining $\iota a$ to $\gamma(\iota a)$ by 
the projection map $p: \hh \longrightarrow \hh/ \Gamma_{g}$ 
by $C_{a}$. We note that $C_{a}$ is a closed curve.  
See Figure \ref{pehlachitra} below. 
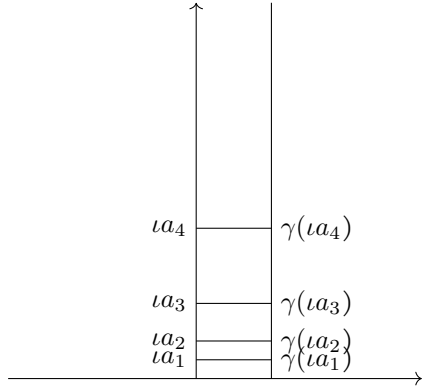
\begin{figure}[H]
\begin{center}
  \begin{tikzpicture}
\draw  [->](-2.5,0) -- (3,0);
 \draw [->](0,0) -- (0, 5);
\draw (1,0) -- (1,5);
 \draw (0,0.5)node[left]{$\iota a_{2}$} -- (1,0.5) node[right]{$\gamma(\iota a_{2}$)};
 \draw (0,1)node[left]{$\iota a_{3}$} -- (1,1)node[right]{$\gamma(\iota a_{3}$)};
 \draw (0,2)node[left]{$\iota a_{4}$} -- (1,2)node[right]{$\gamma(\iota a_{4}$)} ;
 \draw (0,0.25) node[left]{$\iota a_{1}$} -- (1, 0.25)node[right]{$\gamma(\iota a_{1}$)};
\end{tikzpicture}
\end{center}
\caption{Line segments joining $\iota a_{i}$ to $\gamma(\iota a_{i})$}
\label{pehlachitra}
\end{figure}
Recall the Poincar\'{e} metric on $\hh$ induces a 
hyperbolic metric on the compact surface $\hh / \Gamma_{g}$. 
Let $l(C_{a})$ be the hyperbolic length of $C_{a}$ w.r.t to a  
hyperbolic metric on $\hh / \Gamma_{g}$. We have one-to-one correspondence between the 
free homotopy classes of closed curves on the compact surface $\hh / \Gamma_{g}$ and 
the set of conjugacy classes in the fundamental group $\pi_{1}(\hh / \Gamma_{g})$. So 
we could view $C_{a}$ as an element of the 
fundamental group $\pi_{1}(\hh / \Gamma_{g})$. $C_{a}$ is null-homotopic because 
for a sequence of positive numbers 
$\{ a_{i}\}_{i=1}^{\infty}$, $l(C_{a_{i}}) \longrightarrow 0$ as 
$i \longrightarrow \infty$. In order to get a contradiction we have to show that 
$C_{a}$ is not null-homotopic as an element of the fundamental group 
$\pi_{1}(\hh / \Gamma_{g}) \simeq  \Gamma_{g}$. It's obvious because we started with a non-identity element 
$\gamma \in \Gamma_{g}$. \hfill \qedsymbol
\end{proof}
\begin{lemma}[\cite{hub}, \cite{imayoshi}]
\label{fuchsianlemma}
Let $\gamma_{1}, \gamma_{2} \in \mathrm{PSL}(2, \rr) -\{ 1 \}$ be hyperbolic, where 
$1$ denotes the identity element of $\mathrm{PSL}(2, \rr)$.  
Let $\mathrm{Fix}(\gamma_{1})$ and 
$\mathrm{Fix}(\gamma_{2})$ be the set of fixed points of $\gamma_{1}$ and $\gamma_{2}$, 
where the set of 
fixed points of an element $\gamma \in \mathrm{PSL}(2, \rr)-\{ 1 \}$ is the 
set of all $z \in \rr \cup \{\infty \}$ satisfying $\gamma(z)=z$. Then
$\gamma_{1}$ and $\gamma_{2}$ commute 
iff they have atleast one common fixed point, i.e., 
$$z \in \mathrm{Fix}(\gamma_{1}) \cap \mathrm{Fix}(\gamma_{2}) \neq \emptyset.$$
 
\end{lemma} 

\begin{remark}
\normalfont
 \label{cyclic}
We denote the centralizer of $\Gamma_{g}$ in $\mathrm{PSL}(2, \rr)$ by 
$C_{\Gamma_{g}} \mathrm{PSL}(2, \rr)$. From Lemma \ref{fuchsianlemma}, it is 
easy to see that 
$C_{\Gamma_{g}} \mathrm{PSL}(2, \rr)$ is trivial. Here is an argument: 
 from Lemma \ref{fuchsianlemma}, $\gamma_{1}$ and $\gamma_{2}$ are noncommuting iff 
 $\mathrm{Fix}(\gamma_{1}) \cap \mathrm{Fix}(\gamma_{2}) = \emptyset$. Now, let's assume that 
 $\gamma \in \mathrm{PSL}(2, \rr)$ commutes with $\gamma_{1}$ and $\gamma_{2}$. Then 
 $\gamma$ fixes the axis of $\gamma_{1}$ and $\gamma_{2}$, since 
 $\gamma(\mathrm{Ax}\gamma_{i}) = \mathrm{Ax}\gamma \gamma_{i} \gamma^{-1} = 
 \mathrm{Ax} \gamma_{i}$, for $i=1, 2$. Thus $\gamma$ maps $\mathrm{Fix}(\gamma_{i}), i=1, 2$ 
 to itself. However we cannot conclude that 
 $\gamma(z)= z$, $z \in \mathrm{Fix}(\gamma_{i}), i=1, 2$. We have two possibilities: 
 \begin{enumerate}
  \item $\gamma$ is hyperbolic with the same axis as of $\gamma_{1}$ 
  and $\gamma_{2}$. 
  \item $\gamma$ is elliptic of order $2$, i.e., $\gamma$ interchanges the fixed points of 
  $\gamma_{1}$ and $\gamma_{2}$. And $\gamma \gamma_{i} \gamma^{-1} = \gamma_{i}^{-1}$.
 \end{enumerate}
We can exclude both the possibilities because according to (1), $\gamma$ has $4$ 
fixed points, hence a contradiction. And from (2), 
$\gamma \notin C_{\Gamma_{g}} \mathrm{PSL}(2, \rr)$.
Hence, 
 $C_{\Gamma_{g}} \mathrm{PSL}(2, \rr)$ is trivial. 
 
\end{remark}

\begin{defn}
\label{classicaldefn}
 \normalfont The Teichm\"{u}ller space of $\Sigma_{g}$ is defined as
 the space of equivalence classes of \textit{marked hyperbolic surfaces}. 
 By a \textit{marked hyperbolic surface} we mean a pair $(S, \phi)$ where 
 $S$ is a hyperbolic surface and
 $\phi: \Sigma_{g} \longrightarrow S$ is an orientation preserving diffeomorphism. 
 Equivalence relation is defined as follows:
 $$ (S, \phi) \sim (S', \psi),$$
 if there exists an isometry $h: S \longrightarrow S'$ such that $\psi$ is 
 isotopic to $h \circ \phi$. We denote the Teichm\"{u}ller space of $\Sigma_{g}$ 
 by $\mathscr{T}(\Sigma_{g})$.
\end{defn}
\begin{remark}
\label{glitchtopology}
 \normalfont
 Note that there is a glitch in the above definition as we have not introduced 
 a topology on the Teichm\"{u}ller space $\mathscr{T}(\Sigma_{g})$. There 
 is a notion of the \textit{Teichmueller metric} which gives a topology on 
 $\mathscr{T}(\Sigma_{g})$. 
 See \cite{farbmar} and \cite{imayoshi} 
 for a complete understanding. 
\end{remark}

\subsubsection{$\mathscr{T}(\Sigma_{g})$ as a representation variety}
Let $\Gamma$ be a finitely generated group and $G$ be a connected Lie group. 
The most interesting case for us is when
$\Gamma= \Gamma_{g}$
and $G= \mathrm{Isom}^{+}(\mathbb{H}^{2}) \cong \mathrm{PSL}(2, \mathbb{R})$. 
Let $\mathrm{Hom}(\Gamma, G)$ 
denote the space of all homomorphisms $\Gamma \longrightarrow G$ 
with the compact-open topology. 
Note that $G$ can be described 
as a closed subgroup of $\mathrm{GL}(k, \rr)$ for some large $k$. 
Therefore, we can think of $G$
as a real algebraic subgroup of $\mathrm{GL}(k, \rr)$.
The space $\mathrm{Hom}(\Gamma, G)$ has 
the structure of an algebraic variety. 
 The representation
variety $\mathrm{Hom}(\Gamma, G)$
is isomorphic to an algebraic subvariety (in $G^{n}$). 
The isomorphism type of the variety $\mathrm{Hom}(\Gamma, G)$ 
does not depend on the 
choice of the presentation of $\Gamma$ (see \cite{KM99}, \cite{LM85}). 
Note that the spaces 
$\mathrm{Hom}(\Gamma, G)$ are not generally manifolds. 
The natural symmetries of the space
$\mathrm{Hom}(\Gamma, G)$ come from the action of 
$\mathrm{Aut}(\Gamma) \times \mathrm{Aut}(G)$ 
where the action is described as:
if $\gamma \in \mathrm{Aut}(\Gamma)$ and $\alpha \in \mathrm{Aut}(G)$, 
then $\rho^{(\gamma, \alpha)} \in \mathrm{Hom}(\Gamma, G)$ is defined as:
$$\rho^{(\gamma, \alpha)}(x)= (\alpha \circ \rho \circ \gamma^{-1})(x). $$
We will be mainly concerned with the quotient space of $\mathrm{Hom}(\Gamma, G)$ 
by $\mathrm{Inn}(G)$ 
which will be denoted by $\mathrm{Hom}(\Gamma, G)/G$. Note
that $\mathrm{Inn}(G)$ does not act freely on $\mathrm{Hom}(\Gamma, G)$ in some cases. 
The isotropy group 
of a point $\rho \in \mathrm{Hom}(\Gamma, G)$ is the centralizer $C_{G}(\rho)$ in 
$\mathrm{Inn}(G)$ and $\mathrm{Inn}(G)$ acts freely on $\mathrm{Hom}(\Gamma, G)$ if 
$C_{G}(\rho)$ is trivial for all $\rho \in  \mathrm{Hom}(\Gamma, G)$. 
In the case of our interest, i.e., when $\Gamma = \Gamma_{g}$ and
 $G= \mathrm{PSL}(2, \rr)$, we overcome this pathology (see Remark \ref{cyclic}).
The quotient space $\mathrm{Hom}(\Gamma, G)/G$ is 
 not generally a Hausdorff space unless $G$ is a 
 compact Lie group.  
 \begin{defn}
  \label{subspacesofrep}
  \normalfont 
  \begin{equation*}
 \begin{split}
  \mathrm{Hom}_{\mathrm{DF}}(\Gamma, G) & := \{ \rho \in \mathrm{Hom}(\Gamma, G) | 
  \rho \hspace{2pt} \text{is injective with discrete image}\}, \\
  \mathrm{Hom}_{0}(\Gamma, G) & := \{ \rho \in \mathrm{Hom}_{\mathrm{DF}}(\Gamma, G) | 
  G/ \rho(\Gamma)\hspace{2pt}\text{is compact}\}.
 \end{split}
\end{equation*}
 \end{defn}
 \begin{remark}
 \label{weilremark}
  \normalfont
  It is clear that $\mathrm{Hom}_{0}(\Gamma, G) \subset \mathrm{Hom}_{\mathrm{DF}}(\Gamma, G) 
 \subset \mathrm{Hom}(\Gamma, G)$.
 $\mathrm{Hom}_{0}(\Gamma, G)$ is an open subset of 
 $\mathrm{Hom}(\Gamma, G)$ \cite{Weil60}, \cite{Weil62}. 
 \end{remark}
\begin{defn}
 \label{teichrep}
  \normalfont
  The Teichmueller space $\mathscr{T}(\Sigma_{g})$ of $\Sigma_{g}$ is (also) 
  defined as the quotient space 
  $$\mathrm{Hom}_{0}(\Gamma_{g}, \mathrm{PSL}(2, \rr))/ \mathrm{PSL}(2, \rr),$$ where 
  $\mathrm{Hom}_{0}(\Gamma_{g}, \mathrm{PSL}(2, \rr))$ is defined in Definition 
  \ref{subspacesofrep}. 
 \end{defn}
 The above definition will be the main definition of the Teichmueller 
 space in this thesis. Now, we prove the following general fact using techniques 
from differential topology:
\begin{prop}
\label{teichmanifold}
$\mathrm{Hom}_{0}(\Gamma_{g}, \mathrm{PSL}(2, \rr))/ \mathrm{PSL}(2, \rr)$ has a preferred 
structure of smooth manifold 
 of dimension $6g-6$.
\end{prop}
\begin{proof}
 We prove the statement in the following steps:
 \smallskip
\paragraph{\textbf{Step I:}}
Here we prove that $\mathrm{Hom}_{0}(\Gamma_{g}, \mathrm{PSL}(2, \rr))$ is a smooth manifold 
of dimension $6g-3$. 
Since a homomorphism $\rho: \Gamma_{g} \longrightarrow \mathrm{PSL}(2, \rr)$ is 
determined by choosing the $2g$ images $\rho(a_{i}), \rho(b_{i}), 1 \leq i \leq 
g$, there is a natural inclusion of 
$\mathrm{Hom}(\Gamma_{g},  \mathrm{PSL}(2, \rr))$ 
into the direct product $\mathrm{PSL}(2, \rr)^{2g}$ of $2g$ copies of 
$\mathrm{PSL}(2, \rr)$. Consider the following map
$$R: \mathrm{PSL}(2, \rr)^{2g} \longrightarrow \mathrm{PSL}(2, \rr) $$
given by 
\begin{equation}
\label{commutatormap}
R(A_{1}, B_{1}, \ldots, A_{g}, B_{g}) = A_{1}B_{1}A_{1}^{-1}B_{1}^{-1} \cdots 
A_{g}B_{g}A_{g}^{-1}B_{g}^{-1}.
\end{equation}
\textit{Claim:} We assume that $A_{1}$ and $B_{1}$ 
are noncommuting hyperbolic elements. Then the differential of $R$ at 
$(A_{1}, B_{1}, \ldots, A_{g}, B_{g}) 
\in \mathrm{PSL}(2, \rr)^{2g}$ is surjective. 
\smallskip
\newline
\textit{Proof of the Claim:}
By precomposing the map $R$ given in (\ref{commutatormap}) with the map 
$$ \mathrm{PSL}(2, \rr) \times \mathrm{PSL}(2, \rr) \longrightarrow \mathrm{PSL}(2, \rr)^{2g}$$
$(A, B) \longmapsto (A, B, 1, \ldots, 1) $, 
we get a map 
$$ \mathrm{PSL}(2, \rr) \times \mathrm{PSL}(2, \rr)  \longrightarrow
\mathrm{PSL}(2, \rr)$$ given by 
\begin{equation}
 \label{reduced}
 (A, B) \longmapsto ABA^{-1}B^{-1}.
\end{equation}
We denote this composite map by $R$ as well. 
Therefore, 
proving the above-mentioned claim amounts to proving the following statement: 
Let $\mathfrak{g}$ denote the 
Lie algebra of 
$\mathrm{PSL}(2, \rr)$. 
If $A$ and $B$ are noncommuting hyperbolic elements, then the differential of 
the map $R$ given in (\ref{reduced})
\begin{equation}
 \label{differentialata}
 dR(A, B): T_{(A, B)} (\mathrm{PSL}(2, \rr) \times \mathrm{PSL}(2, \rr))
 \longrightarrow T_{R(A, B)} \mathrm{PSL}(2, \rr)
\end{equation}
 is surjective. For the calculation of the differential $dR(A, B)$ we can replace 
$\mathrm{PSL}(2, \rr)$ with 
$\mathrm{SL}(2, \rr)$. 
A simple calculation shows that 
$$T_{A}\mathrm{SL}(2, \rr) =  A \cdot \mathfrak{gl}(2, \rr),$$
where $\mathfrak{gl}(2, \rr)$ is the Lie algebra of $\mathrm{SL}(2, \rr)$, equivalently, 
the tangent space at the identity. From this discussion on tangent spaces, we can write
(\ref{differentialata}) as
\begin{equation}
 \label{final}
 dR(A, B): A \mathfrak{gl}(2, \rr) \times B \mathfrak{gl}(2, \rr) 
\longrightarrow R(A, B) \mathfrak{gl}(2, \rr).
\end{equation}
 Now, we prove the surjectivity of the map given by (\ref{final}). 
 First, we calculate the differential of $R$ 
 at $(A, B)$. Let $u, v \in \mathfrak{gl}(2, \rr)$. For $t \rightarrow 0$, we have
\begin{equation*}
 \label{differentialapprox}
 \resizebox{0.95\hsize}{!}{$
 \begin{split}
  R(A\exp tu, B\exp tv) - R(A, B) 
                   & \approx A(I+tu)B(I+tv)(I-tu)A^{-1}(I-tv)B^{-1}
                   -ABA^{-1}B^{-1} \\
                   & \approx (A+Atu)(B+Btv)(A^{-1}-tuA^{-1})
                   (B^{-1}-tvB^{-1}) - ABA^{-1}B^{-1} \\ 
                   & \approx (AB+ABtv+AtuB) 
                   (A^{-1}B^{-1}-A^{-1}tvB^{-1}-tuA^{-1}B^{-1}) \\
                   & \quad - ABA^{-1}B^{-1} \\ 
                   & \approx ABA^{-1}B^{-1} - ABA^{-1} 
                   tv B^{-1}
                   - ABtuA^{-1}B^{-1} + ABtvA^{-1}B^{-1} \\
                   & \quad + 
                   AtuBA^{-1}B^{-1} - ABA^{-1}B^{-1} \\ 
                   & \approx - ABA^{-1} tv B^{-1}
                   - ABtuA^{-1}B^{-1} + ABtvA^{-1}B^{-1} + 
                   AtuBA^{-1}B^{-1} \\ 
                   & \approx AB \big(-A^{-1}tvA-tu+tv+B^{-1}tuB\big) 
                   A^{-1}B^{-1}. \\
 \end{split}
$}
\end{equation*}
Recall that the adjoint representation $\mathrm{Ad}$  
of $\mathrm{SL}(2, \rr)$ on $\mathfrak{gl}(2, \rr)$ is defined by 
\begin{equation*}
 \label{adjointrepresen}
(\mathrm{Ad}A)w:= A^{-1}wA, \quad w \in \mathfrak{gl}(2, \rr).  
\end{equation*}
Therefore, 
 the differential $dR(A, B): A \mathfrak{gl}(2, \rr) \times B \mathfrak{gl}(2, \rr)
\longrightarrow R(A, B) \mathfrak{gl}(2, \rr)$ is given by 
 the following: 
 \begin{equation*}
  \label{differentialcom}
  (Au, Bv) \longmapsto AB\big((\mathrm{Ad}B)u-u + v- (\mathrm{Ad}A)v 
  \big) A^{-1}B^{-1}, \quad u, v \in \mathfrak{gl}(2, \rr). 
 \end{equation*}
It is enough to show that the map $\mathfrak{gl}(2, \rr) \times \mathfrak{gl}(2, \rr) \longrightarrow 
\mathfrak{gl}(2, \rr)$ given by
 \begin{equation}
  \label{differentialcom2}
  (u, v) \longmapsto (\mathrm{Ad}B)u-u + v- (\mathrm{Ad}A)v, \quad u, v \in 
  \mathfrak{gl}(2, \rr)
 \end{equation}
 is surjective. 
 \smallskip
\newline
 \textit{Proof of surjectivity of the map given in (\ref{differentialcom2}):} Note 
 that $\mathrm{SL}(2, \rr)$ preserves a nondegenerate bilinear form on its Lie
 algebra $\mathfrak{gl}(2, \rr)$. Moreover, $\mathrm{PSL}(2, \rr)$ embeds into the isometry 
 group of the Killing form on $\mathfrak{gl}(2, \rr)$. So, we think of $B$ as 
 an element of 
 one parameter subgroup generated by $b \in \mathfrak{gl}(2, \rr)$ of the 
 isometry group of the Killing form on $\mathfrak{gl}(2, \rr)$. The image of the linear map 
 $u \longmapsto (\mathrm{Ad}B)u-u$ from $\mathfrak{gl}(2, \rr)$ to itself is precisely 
 the $2$-dimensional subspace of $\mathfrak{gl}(2, \rr)$ which is perpendicular (in the sense of the 
 Killing form) 
 to $b$. Similarly, the image of the linear map 
 $v \longmapsto v - (\mathrm{Ad}A)v$ from $\mathfrak{gl}(2, \rr)$ to itself is precisely 
 the $2$-dimensional subspace of $\mathfrak{gl}(2, \rr)$ which is 
 perpendicular (in the sense of Killing form) 
 to $a \in \mathfrak{gl}(2, \rr)$.  
 Since we have chosen $A$ and $B$ such that they are noncommuting hyperbolic elements,
 $a$ and $b$ are 
 linearly independent in $\mathfrak{gl}(2, \rr)$. The reader can also verify these 
 two statements in 
 coordinates, i.e., by making choices for $B$ (and $A$ respectively), $u$ 
 (and $v$ respectively) and plugging these into $u \longmapsto (\mathrm{Ad}B)u-u$ and 
   $v \longmapsto v - (\mathrm{Ad}A)v$.
 Therefore, 
 the map $\mathfrak{gl}(2, \rr) \times \mathfrak{gl}(2, \rr) \longrightarrow \mathfrak{gl}(2, \rr)$ given in 
 (\ref{differentialcom2}) 
 is surjective. 
 \smallskip
\newline
 We denote the subset of $\mathrm{PSL}(2, \rr)^{2g}$ consisting of 
 elements $A_{1}, B_{1}, \ldots, A_{g}, B_{g}$ such that 
 $A_{1}, B_{1}$ are noncommuting hyperbolic elements by $W$. 
 Since $W$ is open in $\mathrm{PSL}(2, \rr)^{2g}$, hence $W$ is a manifold of dimension $6g$. 
 From the above-mentioned claim, 
 $1$ is a regular value of the restriction map 
 $R|_{W}: W \longrightarrow \mathrm{PSL}(2, \rr)$. In fact, 
 every value of the map $R|_{W}$ is a regular value. 
 Hence, 
 $R|_{W}^{-1}(1)$ is a submanifold of $W$ of dimension $6g-3$. Note that 
 $R|_{W}^{-1}(1)$ is nothing but $\mathrm{Hom}(\Gamma_{g}, \mathrm{PSL}(2, \rr)) \cap W$.  
 From Remark \ref{weilremark}, we know that 
 $\mathrm{Hom}_{0}(\Gamma_{g}, \mathrm{PSL}(2, \rr))$ is an open subset of 
 $\mathrm{Hom}(\Gamma_{g}, \mathrm{PSL}(2, \rr))$, therefore, 
 $\mathrm{Hom}_{0}(\Gamma_{g}, \mathrm{PSL}(2, \rr))$ is a $6g-3$ dimensional 
 smooth manifold. 
 \smallskip
 \paragraph{\textbf{Step II:}} In this step, we study the action of $\mathrm{PSL}(2, \rr)$ on 
$\mathrm{Hom}_{0}(\Gamma_{g}, \mathrm{PSL}(2, \rr))$. Given $g \in 
\mathrm{PSL}(2, \rr)$ and $\rho \in \mathrm{Hom}_{0}(\Gamma_{g}, \mathrm{PSL}(2, \rr))$, we define 
$\rho^{g}: \Gamma_{g}\longrightarrow \mathrm{PSL}(2, \rr)$ by setting 
\begin{equation}
 \label{freeaction}
 \rho^{g}(\gamma) = 
g \rho(\gamma) g^{-1}, \quad \forall \gamma \in \Gamma_{g}. 
\end{equation}
The map $(g, \rho) \longmapsto \rho^{g}$ is a continuous action of 
$\mathrm{PSL}(2, \rr)$ on
$\mathrm{Hom}_{0}(\Gamma_{g}, \mathrm{PSL}(2, \rr))$. 
We want to show that the 
action is free and that the orbit space of this action is again a smooth manifold. 
Consider the following map
$$\psi_{1}: \mathrm{Hom}_{0}(\Gamma_{g}, \mathrm{PSL}(2, \rr)) \longrightarrow 
 \mathrm{Conf}_{3}(\partial \hh),$$
 $$ \rho \longmapsto (z_{1}, z_{2}, z_{3})$$
 where $\mathrm{Conf}_{3}(\partial \hh)$ is the space of ordered configurations of distinct
 $3$ points in the boundary $\partial \hh$. In the above map, $z_{1}$, $z_{2}$ 
 are \textit{attractive} and \textit{repelling} 
  fixed points of $A_{1}$, i.e., 
 $$\lim_{n \rightarrow \infty} A_{1}^{n}(z)=z_{1}, \forall z \in \hh, \quad 
 \lim_{n \rightarrow -\infty} A_{1}^{n}(z)=z_{2}, \forall z \in \hh, $$
 and $z_{3}$ is the attractive fixed point of 
 $B_{1}$. Moreover, the group $\mathrm{PSL}(2, \rr)$ acts 
sharply transitively on ordered 
triples in $\partial \hh$, 
we can also think of $\psi_{1}$ 
 as a map 
  $$
  \psi_{1}: \mathrm{Hom}_{0}(\Gamma_{g}, \mathrm{PSL}(2, \rr)) 
  \longrightarrow \mathrm{PSL}(2, \rr).$$
  Note that we have identified
  $\mathrm{PSL}(2, \rr)$ with $\mathrm{Conf}_{3}(\partial \hh)$ by the map 
  $g \longmapsto g \cdot (0, 1, \infty)$. Observe that $\psi_{1}$ is a $\mathrm{PSL}(2, \rr)$-equivariant map, i.e.,
  $$\psi_{1}(g \cdot \rho) = g \cdot \psi_{1}(\rho), \quad \forall g \in \mathrm{PSL}(2, \rr), $$
  where the action on the L.H.S is by 
  conjugation and the action on the R.H.S is by left-multiplication. In other words, 
  if we change $\rho$ 
  by conjugating it by an element $g \in \mathrm{PSL}(2, \rr)$, the three distinct 
  points $z_{1}, z_{2}, z_{3}$ in $\partial \hh$ are also transformed by the same element $g \in 
  \mathrm{PSL}(2, \rr)$. The only thing we have to show now is 
  that $\psi_{1}$ is differentiable. Here is an 
 argument: 
 $\mathrm{Hom}_{0}(\Gamma_{g}, \mathrm{PSL}(2, \rr))$ is also a closed subset of 
  $\mathrm{PSL}(2, \rr)^{2g}$. Now, $\psi_{1}$ extends to a small open neighborhood 
 $$\mathscr{U} \subset \mathrm{Hom}_{0}(\Gamma_{g}, \mathrm{PSL}(2, \rr))$$
 in $\mathrm{PSL}(2, \rr)^{2g}$. We know that an element $\rho \in 
 \mathrm{Hom}_{0}(\Gamma_{g}, \mathrm{PSL}(2, \rr))$ is determined by hyperbolic elements
 $(A_{1}, B_{1}, \ldots, A_{g}, B_{g}) 
 \in \mathrm{PSL}(2, \rr)^{2g}$ satisfying the relation
 $$[A_{1}, B_{1}] \cdots [A_{g}, B_{g}]=1.$$ Since 
 the set of hyperbolic elements form  an 
 open subset
 of $\mathrm{PSL}(2, \rr)$ (see \textbf{\cref{factshyperbolic}}), 
 then an open neighborhood
 $\mathscr{U} \subseteq \mathrm{PSL}(2, \rr)^{2g}$ of 
 $ \mathrm{Hom}_{0}(\Gamma_{g}, \mathrm{PSL}(2, \rr))$ also
 contains hyperbolic elements $A'_{1}, B'_{1}$, $\ldots, A'_{g}, B'_{g}$ which may not 
 satisfy $[A'_{1}, B'_{1}] \cdots [A'_{g}, B'_{g}]=1$. 
 The upshot is $\psi_{1}$ is smooth because it is the 
 restriction of a map defined on an open neighborhood $\mathscr{U}$ of 
 $\mathrm{Hom}_{0}(\Gamma_{g}, \mathrm{PSL}(2, \rr))$ which is obviously smooth. 
 $\mathrm{PSL}(2, \rr)$-equivariance of $\psi_{1}$ makes immediately clear that 
  $\psi_{1}$ is everywhere regular. 
  Therefore, $\psi_{1}^{-1}(1)$ is a submanifold of codimension 3 of
  $\mathrm{Hom}_{0}(\Gamma_{g}, \mathrm{PSL}(2, \rr))$. We denote $\psi_{1}^{-1}(1)$ 
  by $Z$. 
 Tying it all together, the action of $\mathrm{PSL}(2, \rr)$ on 
 $\mathrm{Hom}_{0}(\Gamma_{g}, \mathrm{PSL}(2, \rr))$ admits 
 a transversal, i.e., there exists a submanifold $Z$ of 
 $\mathrm{Hom}_{0}(\Gamma_{g}, \mathrm{PSL}(2, \rr))$ of 
 codimension $3$ such that the action of $\mathrm{PSL}(2, \rr)$ 
 gives us a diffeomorphism 
 $$\psi_{2}: \mathrm{PSL}(2,\rr) \times Z \longrightarrow 
 \mathrm{Hom}_{0}(\Gamma_{g}, \mathrm{PSL}(2, \rr))$$
 $$\psi_{2}(g, z) = gzg^{-1}.$$
 Therefore, the orbit space 
 $\mathrm{Hom}_{0}(\Gamma_{g}, \mathrm{PSL}(2, \rr))/ \mathrm{PSL}(2, \rr)$ is diffeomorphic 
 to $Z$. \hfill \qedsymbol 
\end{proof}
\begin{remark}
 \normalfont
 Note that a different choice of generators for $\Gamma_{g}$ will give the same 
 structure of smooth manifold 
 on $\mathrm{Hom}_{0}(\Gamma_{g}, \mathrm{PSL}(2, \rr))/ \mathrm{PSL}(2, \rr)$. 
\end{remark}

\begin{remark}
 \normalfont
 We were only made aware of Earle and Eells' paper \cite{EE69}, where 
 they only give a sketch proof of Step 1 at the end of this thesis research. 
 Many thanks to Johannes Ebert. 
 \end{remark}

 \subsection{Tangent spaces to the Teichm\"{u}ller space $\mathscr{T}(\Sigma_{g})$}
\label{tangentteichmuellerspace}
 \subsubsection{Cohomological description}
\label{cohomo}
Let $\Gamma$ be a finitely generated group and $G$ be a connected Lie group with 
Lie algebra $\mathfrak{g}$. 
We can obtain a (linear) action
of $\Gamma$ on $\mathfrak{g}$ by fixing a homomorphism $\rho_{0}:\Gamma \longrightarrow G$ and composing $\rho_{0}$ with the adjoint
representation of $G$ and hence make $\mathfrak{g}$ a $k \Gamma$-module where
$k=\rr$ or $\cc$. We denote $\mathfrak{g}$ with the above-mentioned $\Gamma$-module structure 
by $\mathfrak{g}_{\mathrm{Ad}\rho_{0}}$. A map $c: \Gamma \longrightarrow \mathfrak{g}$ is called a \textit{1-cocycle} if
\begin{equation}
\label{cocycle}
 c(\gamma_{1} \gamma_{2})=c(\gamma_{1})+\mathrm{Ad}(\rho_{0}(\gamma_{1}))c(\gamma_{2}), \quad \forall \gamma_{1}, \gamma_{2} \in \Gamma.
\end{equation}
$c$ is a \textit{1-coboundary} if it has the following form
\begin{equation}
\label{coboundary}
 c(\gamma)= u-\mathrm{Ad}(\rho_{0}(\gamma))u
\end{equation}
for some $u \in \mathfrak{g}$. The (real vector) space of 1-cocycles is denoted by 
$Z^{1}(\Gamma; \mathfrak{g}_{\mathrm{Ad}\rho_{0}})$ 
and the (real vector) space of 1-coboundaries is denoted by
$B^{1}(\Gamma; \mathfrak{g}_{\mathrm{Ad}\rho_{0}})$. Their quotient is the group cohomology
\begin{equation*}
 H^{1}(\Gamma; \mathfrak{g}_{\mathrm{Ad}\rho_{0}})= Z^{1}(\Gamma; \mathfrak{g}_{\mathrm{Ad}\rho_{0}})/ 
 B^{1}(\Gamma; \mathfrak{g}_{\mathrm{Ad}\rho_{0}}).
\end{equation*}
When $\Gamma=\pi_{1}(M)$ for a topological space $M$, $H^{1}(\Gamma; \mathfrak{g}_{\mathrm{Ad}\rho_{0}})$ 
can be identified with $H^{1}(M; \mathfrak{g}_{\mathrm{Ad}\rho_{0}})$, 
the first cohomology of $M$ with coefficients in the local system given 
by $\mathfrak{g}_{\mathrm{Ad}\rho_{0}}$. 
For more details on group cohomology, the reader is referred to \cite{brown}. 
We are interested in the case when $\Gamma = 
\Gamma_{g}$, $G= \mathrm{PSL}(2, \rr)$, and $\mathfrak{g}$ is the Lie algebra of $\mathrm{PSL}(2, \rr)$. 
\begin{prop}[\protect{\cite[Theorem 2.6]{LM85}, \cite[Chapter VI]{raghu}}]
 \label{tangent1}
$$T_{[\rho_{0}]} \mathrm{Hom}_{0}(\Gamma_{g}, \mathrm{PSL}(2, \rr))/ \mathrm{PSL}(2, \rr)
\cong H^{1}(\Gamma_{g}; \mathfrak{g}_{\mathrm{Ad}\rho_{0}}).$$
\end{prop}
\begin{proof}
We construct a linear map 
$$\Psi: T_{[\rho_{0}]} \mathrm{Hom}_{0}(\Gamma_{g}, \mathrm{PSL}(2, \rr))/ \mathrm{PSL}(2, \rr)
\longrightarrow 
 H^{1}(\Gamma_{g}; \mathfrak{g}_{\mathrm{Ad}\rho_{0}})$$ as follows: 
 to the first order, a curve of maps 
 $(\rho_{t})_{t \in [0, \epsilon)}$ 
 in $\mathrm{Hom}_{0}(\Gamma_{g}, \mathrm{PSL}(2, \rr))$ 
 through the point $\rho_{0}$
 depending smoothly
 on the real parameter $t$ is described as: 
 $$\rho_{t}(\gamma)= \exp\big(tc(\gamma)+O(t^{2})\big) \rho_{0}(\gamma), \quad \forall \gamma \in 
 \Gamma_{g}.$$ 
 The infinitesimal condition for $\rho_{t}$ to be a homomorphism is given as: 
 \begin{equation*}
  \begin{split}
   \rho_{t}(\gamma_{1} \gamma_{2}) & = (e+tc_{\gamma_{1} \gamma_{2}}+O(t^{2})) \rho_{0}(\gamma_{1} \gamma_{2}) \\
                                   & = \rho_{0}(\gamma_{1} \gamma_{2}) + t c_{\gamma_{1} \gamma_{2}} \rho_{0}(\gamma_{1} \gamma_{2}) + O(t^{2}) \\
                                   & = \big(\rho_{0}(\gamma_{1})+t c_{\gamma_{1}}\rho_{0}(\gamma_{1}) \big)\big(\rho_{0}(\gamma_{2})+t c_{\gamma_{2}}\rho_{0}(\gamma_{2})\big)+O(t^{2})\\
                                   & = \rho_{0}(\gamma_{1} \gamma_{2}) + t \big(\rho_{0}(\gamma_{1}) c_{\gamma_{2}} \rho_{0}(\gamma_{2}) + c_{\gamma_{1}} \rho_{0}(\gamma_{1})\rho_{0}(\gamma_{2})\big) +O(t^{2})\\
                                   & = \rho_{0}(\gamma_{1} \gamma_{2}) + t \big( \rho_{0}(\gamma_{1}) c_{\gamma_{2}} + c_{\gamma_{1}}\rho_{0}(\gamma_{1})     \big) \rho_{0}(\gamma_{2}) +O(t^{2})\\
                                   & = \rho_{0}(\gamma_{1} \gamma_{2}) + t \big( \rho_{0}(\gamma_{1}) c_{\gamma_{2}} \rho_{0}(\gamma_{1})^{-1} \rho_{0}(\gamma_{1}) + c_{\gamma_{1}}\rho_{0}(\gamma_{1})     \big) \rho_{0}(\gamma_{2}) +O(t^{2})\\
                                   & = \rho_{0}(\gamma_{1} \gamma_{2}) + t \big( \rho_{0}
                                   (\gamma_{1}) c_{\gamma_{2}} \rho_{0}(\gamma_{1})^{-1}  
                                   + c_{\gamma_{1}}     \big) \rho_{0}(\gamma_{1}) 
                                   \rho_{0}(\gamma_{2}) +O(t^{2})\\
                                   & = \rho_{0}(\gamma_{1} \gamma_{2}) + t \big(
                      \mathrm{Ad}(\rho_{0}(\gamma_{1})) c_{\gamma_{2}}+ c_{\gamma_{1}}\big)
                      \rho_{0}(\gamma_{1} \gamma_{2}) + O(t^{2}). 
  \end{split}
\end{equation*}
From the above equation, notice that
$$ c_{\gamma_{1} \gamma_{2}}= \mathrm{Ad}(\rho_{0}(\gamma_{1})) c_{\gamma_{2}}+ c_{\gamma_{1}}.$$
 Therefore, we define $\Psi \big(\frac{d}{dt} \rho_{t}|_{t=0}\big)$ 
 to be the cocycle $c \in Z^{1}(\Gamma_{g}; \mathfrak{g}_{\mathrm{Ad}\rho_{0}})$. 
 We show next that $\Psi$ is injective. Suppose that the cocycle $c$ determined 
 by $\rho_{t}$ is a coboundary, i.e., $c(\gamma)= u-\mathrm{Ad}(\rho_{0}(\gamma))u$ for 
 some $u \in \mathfrak{g}$ (see (\ref{coboundary})). Then the 
 curve $\rho_{t}(\gamma)= g_{t} \rho_{0}(\gamma)g_{t}^{-1}$ 
 induced by a path $g_{t}=e+t u+O(t^{2}), u \in \mathfrak{g}$
 is tangent at $t=0$ to the orbit $\mathrm{PSL}(2, \rr)\rho_{0}$ for all $\gamma \in 
 \Gamma_{g}$. 
 Moreover, $\Psi$ is surjective because of the fact that 
 $\mathrm{dim} H^{1}(\Gamma_{g}; \mathfrak{g}_{\mathrm{Ad}\rho_{0}})= 6g-6$. 
 The fact follows from a non-trivial result (see \cite{cohgol}) 
 that given a connected Lie group $G$ and 
 $\rho_{0} \in \mathrm{Hom}(\Gamma_{g}, G)$, 
 $$\mathrm{dim}Z^{1}(\Gamma_{g}; \mathfrak{g}_{\mathrm{Ad}\rho_{0}}) = 
 (2g-1) \mathrm{dim} G+\mathrm{dim}C_{G}(\rho(\Gamma_{g})),$$ where 
 $C_{G}(\rho(\Gamma_{g}))$ denotes the the centralizer of $\rho(\Gamma_{g})$ in $G$ and 
 $$\mathrm{dim}B^{1}(\Gamma_{g}; \mathfrak{g}_{\mathrm{Ad}\rho_{0}}) = \mathrm{dim}G - 
 \mathrm{dim}C_{G}(\rho(\Gamma_{g})).$$
For the case of our interest, i.e., when $G= \mathrm{PSL}(2, \rr)$, 
$C_{G}(\rho(\Gamma_{g}))$ is trivial (see Remark \ref{cyclic}). Therefore, 
$$\mathrm{dim}Z^{1}(\Gamma_{g}; \mathfrak{g}_{\mathrm{Ad}\rho_{0}}) = 
 (2g-1) \mathrm{dim} G = 6g-3, \quad 
 \mathrm{dim}B^{1}(\Gamma_{g}; \mathfrak{g}_{\mathrm{Ad}\rho_{0}}) = \mathrm{dim}G =3.$$ 
 \hfill \qedsymbol
 \end{proof}
\begin{remark}
 \normalfont
 Note that $\rho_{0} \in \mathrm{Hom}_{0}(\Gamma_{g}, \mathrm{PSL}(2, \rr))$ 
 can be lifted to a homomorphism $\widetilde{\rho_{0}}:\Gamma_{g} \longrightarrow \mathrm{SL}(2, \rr)$ 
 because the Euler class $e(\rho_{0})$ of the oriented $\ss^{1}$-bundle associated to $\rho_{0}$ equals 
 twice the Euler number of $\rr^{2}$-bundle associated to $\widetilde{\rho_{0}}$, i.e., 
 $e(\rho_{0})=2g-2$. See \cite[Appendix C]{milsta} for more details. As a result, 
 in the above proof, the expressions of $\rho_{t}(\gamma_{1}\gamma_{2})$ and 
 $g_{t}$ are justified. 
\end{remark}
\subsubsection{Analytic description: Holomorphic quadratic differentials}
\label{thespaceofhqd}
Let $K_{\Sigma_{g}}$ be the canonical line bundle, that is, the line bundle over 
$\Sigma_{g}$ such 
that the fiber $K_{x}$ over any point
$x \in \Sigma_{g}$ is the complex cotangent space $T_{x}^{\ast}\Sigma_{g}$ to 
$\Sigma_{g}$ at $x$. Let 
$Q_{\Sigma_{g}}$ be the tensor square of the canonical line bundle 
$K_{\Sigma_{g}}$. The bundle $Q_{\Sigma_{g}}$ and its sections
provide a glimpse into one of the important aspects of the Teichmueller theory. 
\begin{defn}
\label{defnofhqd}
\normalfont
 A holomorphic quadratic differential on $\Sigma_{g}$ is a holomorphic 
 section of $Q_{\Sigma_{g}}$.
\end{defn}
We will denote a holomorphic quadratic differential on $\Sigma_{g}$ by $q$. Locally, 
$q$ on $\Sigma_{g}$ as 
specified in any atlas $\{(U_{i}, z_{i})\}$ can be 
described as $f_{i}(z_{i})dz_{i}^{2}$, where each $f_{i}$ is a holomorphic function
on $U_{i}$ of $\Sigma_{g}$ and $dz_{i}^{2}:= dz_{i} \otimes dz_{i}$ is a 
section of $Q_{\Sigma_{g}}$. Let's denote the space of holomorphic quadratic differentials on $\Sigma_{g}$
by $\mathrm{HQD}(\Sigma_{g})$. Since $K_{\Sigma_{g}}$ has degree $2g-2$, 
the Riemann-Roch formula (see \cite{farkaskra}) implies 
that 
$$\mathrm{dim}(\mathrm{HQD}(\Sigma_{g}))=\mathrm{deg}(Q_{\Sigma_{g}})-g+1=3g-3.$$ 
Note that the bundle $Q_{\Sigma_{g}}$ appears 
in a splitting of the bundle $S^{2}(T\Sigma_{g})$ of (real) symmetric bilinear 
forms on $T\Sigma_{g}$. This splitting is described as follows:
one summand is the 
$1$-dimensional real vector subbundle spanned by the 
everywhere nonzero section of the hyperbolic metric $\textbf{g}$ on $\Sigma_{g}$. 
The other summand is the image of the 
bundle of quadratic differentials under the following embedding:
\begin{equation}
\label{quadmap}
 \psi: \mathrm{hom}_{\mathbb{C}} (T\Sigma_{g} \otimes_{\mathbb{C}} T\Sigma_{g}, \mathbb{C}) 
\longrightarrow S^{2}(T\Sigma_{g}) 
\end{equation}
where $\psi(q)$ is the real part of $q$, 
viewed as a (family) of symmetric $\mathbb{R}$-bilinear forms. 
This subbundle is the
\textit{trace-free} summand by definition. It is a 
2-dimensional subbundle of a 3-dimensional (real) 
vector bundle which comes with a structure of 1-dimensional 
complex 
vector bundle. We illustrate the above splitting as follows:
\begin{example}
 \normalfont
Let $U$ be an open subset of $\mathbb{C}$ with the complex structure induced from $\cc$. 
Then $TU$ is identified
with a trivial bundle $\mathbb{C} \times U \longrightarrow U$ and therefore, $\mathrm{hom}_{\cc}(TU \otimes_{\cc} TU)$ is also identified with a trivial bundle 
$\cc \times U \longrightarrow U$. Therefore, quadratic differentials on $U$ whether holomorphic or not, are identified with complex valued functions on $U$.
For such a function $f$, we get
$$ \psi(f)(z)= \begin{bmatrix}
                \Re(f(z)) & -\Im(f(z)) \\
                -\Im(f(z)) & - \Re(f(z))
               \end{bmatrix},
$$
where $\psi$ is the map given in (\ref{quadmap}). 
This is very easy to check. The preferred ordered basis of $T_{z}U \cong \cc$ as a 
real vector 
space is $\{1, \iota\}$. If $f(z)=x+y \iota$ then $\Re(1 \cdot f(z) \cdot 1)=x$, 
 $\Re(\iota \cdot f(z) \cdot \iota)=-x$, $\Re(1 \cdot f(z) \cdot \iota )=-y$.
\end{example}
From the above discussion, it follows automatically that a holomorphic 
quadratic differential 
$q$ on $\Sigma_{g}$ gives a one parameter family 
$\{g(t)\}_{t \in [0, \epsilon)}$ 
of 
deformations of $\textbf{g}$ on $\Sigma_{g}$ such that $g(0)=\textbf{g}$ and 
$\frac{d g(t)}{dt}\big|_{t=0}=\psi(q)$. In other words, for $t$ close to $0$,
$g(t) = \textbf{g} + t \psi(q)$. 
We view $g(t)$ as a curve in the space $\mathcal{M}$
of Riemannian metrics on $\Sigma_{g}$. Recall that a Riemannian metric on $\Sigma_{g}$ determines
an almost complex structure $J$ on $\Sigma_{g}$ which further determines a complex 
structure on $\Sigma_{g}$. This is due to the Korn-Lichtenstein theorem. 
Consequently, we get a one parameter family of 
complex curves 
$\{\Sigma_{g}^{t}\}_{t \in [0, \epsilon)}$. From the 
``Uniformization theorem'', each of these complex curves 
in the family 
$\{\Sigma_{g}^{t}\}_{t \in [0, \epsilon)}$ has a preferred hyperbolic metric. 
Hence, we view $\{\Sigma_{g}^{t}\}_{t \in [0, \epsilon)}$ 
as a smooth curve in the Teichmueller space 
$\mathrm{Hom}_{0}(\Gamma_{g}, \mathrm{PSL}(2, \rr))/ \mathrm{PSL}(2, \rr)$ such that 
$$\frac{d \Sigma_{g}^{t} }{dt} \bigg|_{t=0} \in 
T_{[\rho_{0}]} \mathrm{Hom}_{0}(\Gamma_{g}, \mathrm{PSL}(2, \rr))/ 
\mathrm{PSL}(2, \rr).$$ 
In summary, we have a linear map from $\mathrm{HQD}(\Sigma_{g})$ to 
$T_{[\rho_{0}]} \mathrm{Hom}_{0}(\Gamma_{g}, \mathrm{PSL}(2, \rr))/ \mathrm{PSL}(2, \rr)$. 
The injectivity of this linear map follows from \cite{sampson}, \cite{wolf}. 
Furthermore, 
this map is a bijective linear map because the dimension of $\mathrm{HQD}(\Sigma_{g})$ and 
$T_{[\rho_{0}]} \mathrm{Hom}_{0}(\Gamma_{g}, \mathrm{PSL}(2, \rr))/ \mathrm{PSL}(2, \rr)$ 
(as real vector spaces) is same.

\section{Explicit expressions of harmonic vector fields on $\hh$}
\label{chapter3}
\subsection{Harmonic maps}
\label{section13}
\textit{Conventions:}
All manifolds are finite dimensional, connected, 
and Riemannian of class $C^{\infty}$, unless otherwise stated. 
All vector bundles and their sections 
are smooth, unless otherwise specified. 
Now we review some basic notions from the theory of harmonic maps. 
We make an effort to do our computations coordinate free first and then in 
coordinates. 
The reader is referred to
the textbook \cite{jost1} for proofs and much more details on harmonic maps. Other 
references on harmonic maps include \cite{eellslemaire}, 
\cite{lemaire4}, \cite{yau}, and \cite{jcwood}. 
\smallskip
\newline
Let $(M,g)$ and $(N, h)$ be $m$ and $n$ dimensional manifolds with the Levi-Civita 
connections $\nabla^{g}$ and $\nabla^{h}$, respectively. Let 
$\phi: M \longrightarrow N$ be a smooth map. The differential 
$$d \phi \in \Gamma(M,T^{\ast}M \otimes \phi^{-1}TN)$$ can be viewed as a $\phi^{-1}(TN)$-valued 1-form 
on $M$, i.e., $d \phi \in \mathscr{A}^{1}(\phi^{-1}(TN))$. Before we define the notion of a \textit{harmonic map}, observe the following:
\begin{enumerate}
 \item There exists a unique connection, $\phi^{-1}\nabla^{h}$, induced by $\phi$ on $\phi^{-1}(TN)$. Note that $\phi^{-1}(TN)$ is a vector bundle on $M$
 defined by $\phi$.
 \item The bundle $T^{\ast} M \otimes \phi^{-1} TN$ has 
 a connection $\nabla$, naturally induced by $\nabla^{g}$ and $\phi^{-1}\nabla^{h}$.
\end{enumerate}
\begin{defn}
\normalfont
$\nabla d \phi \in \Gamma(M, \otimes^{2} T^{\ast}M \otimes \phi^{-1}TN)$ is called the 
\textit{second fundamental form of $\phi$}.
 
\end{defn}
\begin{defn} \label{tensionfieldharmonic}
\normalfont
$\mathrm{Trace}(\nabla d \phi) \in \Gamma(M, \phi^{-1}TN)$ is called the \textit{tension field} of $\phi$. It is usually denoted by $\tau(\phi)$.
\end{defn}

\begin{defn}
\label{totgeo}
\normalfont
 $\phi$ is said to be \textit{totally geodesic} if $\nabla d \phi =0$.
\end{defn}
\begin{defn}
\label{harmonicdefn}
\normalfont
 $\phi$ is said to be \textit{harmonic} if
\begin{equation}
\label{equ2}
\tau(\phi)=0.
\end{equation}
We call $\tau$ the
\textit{Eells-Sampson Laplacian}.
\end{defn}
\textit{In co-ordinate form:} 
By taking coordinate charts, 
the second fundamental form of $\phi$ at $x= (x^{1},\ldots, x^{m}) \in U \subset M$ can be represented as:
\begin{equation} \label{trace}
 (\nabla d \phi)^{\alpha}_{ij}(x)=
 \frac{\partial^{2} \phi^{\alpha}}{\partial x^{i} \partial x^{j}}(x) -
 \Gamma^{k}_{ij} \frac{\partial \phi^{\alpha}}{\partial x^{k}}(x) + 
  \Upsilon^{\alpha}_{\beta \gamma}(\phi(x)) 
  \frac{\partial \phi^{\beta}}{\partial x^{i}}(x) 
  \frac{\partial \phi^{\gamma}}{\partial x^{j}} (x), 
\end{equation}
where $\Gamma^{k}_{ij}$, $\Upsilon^{\alpha}_{\beta \gamma}$ denote 
the Christoffel symbols of $\nabla^{g}$ and $\nabla^{h}$. Note that we have used the 
Einstein-Summation convention in (\ref{trace}).
In coordinate charts,
\begin{equation*}
 \tau^{\alpha}(\phi)(x)= g^{ij}\big(\nabla(d \phi)^{\alpha}_{\ij}(x)\big),
\end{equation*}
where $g^{ij}$ denotes the inverse of the metric tensor $g_{ij}$. 
(\ref{equ2}) in co-ordinate 
form can be expressed as: 
\begin{equation}
\label{euler}
\sum_{i,j=1}^{m} g^{ij} 
 \Bigg( \frac{\partial^{2} \phi^{\alpha}}{\partial x^{i} \partial x^{j}} - 
 \sum_{k=1}^{m}
 \Gamma^{k}_{ij}(x) \frac{\partial \phi^{\alpha}}{\partial x^{k}} + 
 \sum_{\beta, \gamma=1}^{n}
  \Upsilon^{\alpha}_{\beta \gamma}(\phi(x)) \frac{\partial \phi^{\beta}}{\partial x^{i}}
  \frac{\partial \phi^{\gamma}}{\partial x^{j}} \Bigg)=0, \quad 1 \leq \alpha \leq n.
\end{equation}
Note that in (\ref{euler}), the term 
$$\sum_{i,j=1}^{m} g^{ij} \bigg(\frac{\partial^{2} 
\phi^{\alpha}}{\partial x^{i} \partial x^{j}} - \sum_{k=1}^{m}
\Gamma^{k}_{ij} \frac{\partial \phi^{\alpha}}{\partial x^{k}}\bigg)$$
is the Laplace-Beltrami operator on $M$, a contribution of $\nabla^{g}$ in $T^{\ast}M$ 
and the other term 
$$\sum_{i,j=1}^{m} g^{ij} \bigg(\sum_{\beta, \gamma=1}^{n}
  \Upsilon^{\alpha}_{\beta \gamma}(\phi(x)) 
  \frac{\partial \phi^{\beta}}{\partial x^{i}}
  \frac{\partial \phi^{\gamma}}{\partial x^{j}} \bigg)$$
which is 
a non-linear term containing the Christoffel symbols of $\nabla^{h}$
is a contribution of $\phi^{-1}\nabla^{h}$ in $\phi^{-1}TN$.
(\ref{euler}) is the \textit{Euler-Lagrange equation} for the \textit{energy} $E$ of $\phi$ which can be defined under some conditions, 
for example when $M$ is compact, as:
\begin{equation*}
E(\phi)= \int_{M} e(\phi) d \mu_{g},
\end{equation*}
where $d \mu_{g}$ denotes the measure on $M$ induced by $g$ and 
$e(\phi)$ is the \textit{energy density} of $\phi$. The energy density $e(\phi)$ 
of $\phi$ is 
defined by 
$$e(\phi)(x)= \frac{1}{2} \lVert d\phi(x)\rVert^{2} = 
\frac{1}{2} \mathrm{trace}(\phi^{\ast}h)(x),$$
where $ \lVert d\phi(x)\rVert$ is the Hilbert-Schmidt norm of the differential map 
$$d\phi(x): T_{x}M \longrightarrow T_{\phi(x)}N.$$ 
The energy density $e(\phi)$ of $\phi$ has the following expression in local coordinates
\begin{equation}\label{energydensity}
e(\phi)= \frac{1}{2} g^{ij}(x) h_{\beta \gamma}(\phi) 
\frac{d \phi^{\beta}}{d x^{i}} \frac{d \phi^{\gamma}}{d x^{j}}, \quad x \in M.
\end{equation}
When $M$ is compact, we can define $\phi$ to be a harmonic 
map if it's a critical point of $E$.



\begin{remark}
\normalfont 
Harmonic maps are
critical points of the energy functional and hence should not be seen as 
energy minimizers. Below we give the formulation of the 
energy extremal problem in the case of harmonic maps: 
\smallskip
\\Given a smooth map $\phi: (M,g) \longrightarrow (N, h)$, let
\begin{equation*}
E^{\ast}[\phi]=\mathrm{inf} \{ E(\phi'): \phi'= \mathrm{smooth}, \phi' 
\hspace{2pt} \text{is homotopic to} \hspace{2pt} \phi \}
 \end{equation*}
A smooth map $\phi$ such that $E(\phi) = E^{\ast}[\phi]$ is called an energy minimizer. 
For the existence and the uniqueness of energy minimizers when  
the target manifold is equipped with a strictly negatively curved metric, 
see \cite{ells}, \cite{hart}. 
\end{remark}
Now, if we have two complex manifolds $\Sigma_{1}$ and $\Sigma_{2}$ for $M$ and $N$, and on these manifolds, we have conformal metrics,
$$\sigma(z)^{2} dz d \bar{z}= \sigma(z)^{2}(dx^{2}+dy^{2}) \quad (z=x+ \iota y)$$
and
$$\rho(u)^{2} du d \bar{u}= \rho(u)^{2}(du_{1}^{2}+du_{2}^{2}) \quad (u=u_{1}+ \iota u_{2}) $$
then the Laplace-Beltrami operator on 
$\Sigma_{1}$ is given by 
$\frac{1}{\sigma(z)^{2}}\frac{\partial}{\partial z} 
\frac{\partial}{\partial \bar{z}}$. According to J. Jost (see \cite[Chapter 1]{jost1}), 
(\ref{euler}) in these coordinates takes the form
\begin{equation}
\label{equ3}
 \frac{1}{\sigma(z)^{2}}\phi_{z \bar{z}}+  \frac{1}{\sigma(z)^{2}} \frac{2 \rho_{\phi}}{\rho} \phi_{z} \phi_{\bar{z}}=0,
\end{equation}
where a subscript denotes a partial derivative and $\rho_{\phi}$ denotes 
the Wirtinger derivative of $\rho$ at the point $\phi(z)$. 
Therefore, a conformal map between 
Riemann surfaces with conformal metrics is harmonic. 
From (\ref{equ3}), we can see that in the case of a smooth map
$\phi: (\Sigma_{1}, \sigma(z)^{2} dz d \bar{z}) \longrightarrow (\Sigma_{2}, \rho(u)^{2} du d \bar{u}) $
between Riemann surfaces with conformal Riemannian metrics, the Riemannian metric on $\Sigma_{1}$ is not needed to decide
whether $\phi$ is harmonic but the Riemannian metric on $\Sigma_{2}$ matters. More generally, it is also true for a smooth map from a
Riemann surface to a Riemannian manifold. In summary, we see harmonic maps as a very efficient tool to compare the Riemannian metric structure
of $\Sigma_{2}$ to the
conformal structure of $\Sigma_{1}$. Next we discuss some basic examples of harmonic maps.
\begin{example}
\normalfont
Totally geodesic maps are harmonic. Clear from Definition (\ref{totgeo}).
\end{example}

\begin{example}
 \normalfont
The identity map $(M, g) \longrightarrow (M, g)$ is harmonic.  
\end{example}

\begin{example}
 \normalfont
 Let $M=\mathbb{S}^{1}$ and $N$ is compact without boundary, then every homotopy class of 
 maps of $M$ into $N$ contains a closed geodesic, 
 hence a harmonic map. 
\end{example}
To discuss the next two examples we will make a small investment in algebra which 
will lead us to 
consider natural quantities. Recall the definition of an almost complex
structure on $\Sigma_{g}$ from the introduction. 
Extending an almost complex structure 
$J: T\Sigma_{g} \longrightarrow T\Sigma_{g}$
on $\Sigma_{g}$ to the complexified tangent bundle 
$(T\Sigma_{g})^{c}:=T\Sigma_{g} \otimes_{\rr} \cc$
amounts to having a decomposition of the complexified tangent space $(T_{x}\Sigma_{g})^{c}$ 
at each $x \in \Sigma_{g}$
into $(T_{x}\Sigma_{g})^{(1,0)}$ and 
$(T_{x}\Sigma_{g})^{(0, 1)}$ corresponding to eigenvalues $\iota$ 
and $-\iota$. That is, 
$$(T_{x}\Sigma_{g})^{(1,0)} =\{v \in (T_{x}\Sigma_{g})^{c} | Jv=\iota v \}, \quad  
(T_{x}\Sigma_{g})^{(0, 1)} = \{ v \in (T_{x}\Sigma_{g})^{c} | Jv=-\iota v\}.$$
$(T_{x}\Sigma_{g})^{(1,0)}$ and $(T_{x}\Sigma_{g})^{(0, 1)}$ are called 
holomorphic and antiholomorphic tangent spaces, spanned by 
\begin{equation*}
  \frac{\partial}{\partial z} = \frac{1}{2} \bigg( \frac{\partial}{\partial x} - 
  \iota \frac{\partial}{\partial y} \bigg), \quad   \frac{\partial}{\partial \bar{z}}  = \frac{1}{2} \bigg( \frac{\partial}{\partial x} 
   + \iota \frac{\partial}{\partial y} \bigg),
\end{equation*}
where $z=x+\iota y$. In a similar fashion, we can complexify the dual tangent bundle 
$T^{\ast}\Sigma_{g}$ and again, for every $x \in \Sigma_{g}$, 
we can decompose $(T_{x}^{\ast}\Sigma_{g})^{c}$ into its 
$\pm \iota$ eigenspaces - $(T_{x}^{\ast}\Sigma_{g})^{(1,0)}$ and 
$(T_{x}^{\ast}\Sigma_{g})^{(0, 1)}$. $(T_{x}^{\ast}\Sigma_{g})^{(1,0)}$ and 
$(T_{x}^{\ast}\Sigma_{g})^{(0, 1)}$ are 
spanned by 
$$dz=dx+\iota dy, \quad d \bar{z}=dx-\iota dy$$ respectively. 
Using the above decompositions, we can then decompose 
complexified tensor bundles and hence sections of tensor bundles. 
Most importantly, we will consider
a symmetric tensor $s$ in the complexification of the bundle $T^{\ast} \Sigma_{g} 
\otimes T^{\ast} \Sigma_{g}$. Note that $s$ can be written in terms of $dz^{2}:= 
dz \otimes dz$, $d\bar{z}^{2}:= d\bar{z} \otimes d\bar{z}$, and $|dz^{2}|:= 
\frac{1}{2} (dz \otimes d\bar{z} + d\bar{z} \otimes dz)$. 
Tensors that have only 
$(2, 0)$ part can be written locally as $fdz^{2}$ for some locally defined complex valued 
function $f$ and are famously known as 
quadratic differentials (see \textbf{\cref{thespaceofhqd}}).

\begin{example}
\label{harmexamplenice}
 \normalfont
 When $M=\Omega \subset \mathbb{R}^{n}$ and $N=\mathbb{R}$, then the harmonic map 
equations are the harmonic function equations, i.e.,
 $$ \Delta \phi =0.$$
 If $M$ is a surface with a complex structure and $N=\mathbb{R}$, then in the complex language the Laplace equation
 can be written as:
 $$ 4\frac{\partial}{\partial \bar{z}} \frac{\partial}{\partial z}\phi = 0.$$
 Let's try to observe something really important by rewriting the above equation as follows:
 \begin{equation}
 \label{equ4}
  \frac{\partial}{\partial \bar{z}} \bigg(\frac{\partial}{\partial z}\phi \bigg) = 0.
 \end{equation}
 We can also write (\ref{equ4}) in more fancy way as follows:
 $$ \bar{\partial} \big((d \phi)^{(1,0)}\big)=0,$$
 where the object in the parentheses is a ``holomorphic object'' (if and only if the equation holds).
 In other words, tied to the harmonicity of a map $\phi$ on a surface (with a complex structure) is a 
 ``holomorphic object'' which is a
 holomorphic 1-form in the present case. 
\end{example}

\begin{example}
\label{harmimplieshol}
 \normalfont
 A diffeomorphism $\phi: (\Sigma_{1}, \sigma(z)^{2} dz d \bar{z}) \longrightarrow 
 (\Sigma_{2}, \rho(u)^{2} du d \bar{u})$ is harmonic iff 
 $(2, 0)$-part of the pullback metric $\phi^{\ast}(\rho(u)^{2} du d \bar{u})$ is holomorphic. 
 This can be seen as follows: we denote the conformal metric 
 $\rho(u)^{2} du d \bar{u}$ on $\Sigma_{2}$ by $h$. 
 The pullback of $h$ by $\phi$ has the following local expression:
 \begin{equation}
 \label{equ5}
  \begin{split}
      \phi^{\ast} (h) & = 
      (\phi^{\ast} (h))^{(2, 0)} + 
      (\phi^{\ast} (h))^{(1, 1)} + 
      (\phi^{\ast} (h))^{(0, 2)} \\
      & =   \langle \phi_{\ast} \partial_{z}, 
      \phi_{\ast} \partial_{z}
      \rangle_{h} dz^{2} +  
   \big(\Arrowvert \phi_{\ast} \partial_{z} \Arrowvert^{2}_{h} + 
   \Arrowvert \phi_{\ast} \partial_{\bar{z}} \Arrowvert^{2}_{h} 
   \big)\sigma^{2}(z) dz d \bar{z} + 
   \langle \phi_{\ast} \partial_{\bar{z}}, \phi_{\ast} \partial_{\bar{z}} \rangle_
   {h} d\bar{z}^{2}.
     \end{split}
\end{equation}
Note that in the first equality we used the complex eigenspace decomposition 
$$\phi^{\ast} (h) = 
      (\phi^{\ast} (h))^{(2, 0)} + 
      (\phi^{\ast} (h))^{(1, 1)} + 
      (\phi^{\ast} (h))^{(0, 2)}$$ under the action of $J$ on $T\Sigma_{g}$. Also, 
      \begin{equation}
      \label{simplified1}
      \begin{split}
      \langle \phi_{\ast} \partial_{z}, 
      \phi_{\ast} \partial_{z}
      \rangle_{h} dz^{2} & = h \Big(\frac{\partial \phi}{\partial z},  
      \frac{\partial \phi}{\partial z}\Big) dz^{2} \\
      & =  
  \bigg(  h \Big(\frac{\partial \phi}{\partial x},  
      \frac{\partial \phi}{\partial x}\Big) - h \Big(\frac{\partial \phi}{\partial y},  
      \frac{\partial \phi}{\partial y}\Big) - 2 \iota 
      h \Big(\frac{\partial \phi}{\partial x},  
      \frac{\partial \phi}{\partial y}\Big) \bigg)dz^{2} \\
      & = \big( |\phi_{x}|^{2} - |\phi_{y}|^{2} - 2 \iota h(\phi_{x}, \phi_{y}) \big) dz^{2} \\
      & = 4 \rho^{2} \phi_{z} \bar{\phi}_{z} dz^{2}
      \end{split}
      \end{equation} 
      and 
      \begin{equation}
      \label{simplified2}
       \big(\Arrowvert \phi_{\ast} \partial_{z} \Arrowvert^{2}_{h} + 
   \Arrowvert \phi_{\ast} \partial_{\bar{z}} \Arrowvert^{2}_{h} 
   \big)\sigma^{2}(z)= e(\phi),
      \end{equation}
   the energy density of $\phi$, expressed locally in (\ref{energydensity}). 
Now, (\ref{equ5}) has the following form using the simplified expressions in 
(\ref{simplified1}) and
(\ref{simplified2})
\begin{equation*}
 \label{pullbackquad}
 \phi^{\ast} (h) = 4 \rho^{2} \phi_{z} \bar{\phi}_{z} dz^{2} + e(\phi) dz d \bar{z} + 
                                         \overline{4 \rho^{2} \phi_{z} \bar{\phi}_{z} d z^{2}}
\end{equation*}
Now, from \cite[Lemma 1.1]{jost1}, it is easy to see that 
\begin{equation*}
 \label{pullbackquad1}
 \begin{split}
  \partial_{\bar{z}}((\phi^{\ast} (h))^{(2, 0)}) & =  
  \partial_{\bar{z}}(4 \rho^{2} \phi_{z} \bar{\phi}_{z} dz^{2} ) \\
  & = \rho^{2} \big(\bar{\phi}_{z} \tilde{\tau}(\phi) + \phi_{z} \overline{\tilde{\tau}(\phi)}  \big), 
 \end{split}
\end{equation*}
where $\tilde{\tau}(\phi)= \phi_{z \bar{z}}+  \frac{2 \rho_{\phi}}{\rho} \phi_{z} \phi_{\bar{z}}$.
Therefore, $\partial_{\bar{z}}((\phi^{\ast} (h))^{(2, 0)})=0$ when 
$\phi$ is harmonic, i.e., when
$\tau(\phi)=0$ (see (\ref{equ3})) and hence $\tilde{\tau}(\phi)=0$. 
We denote $(\phi^{\ast} (h))^{(2, 0)}$ by $q$. Conversely, if
$q$ is holomorphic, i.e., 
$$\bar{\phi}_{z} \tau(\phi) + \phi_{z} \overline{\tau(\phi)} =0 $$
and if $\tau(\phi) \neq 0$ at a point $p \in \Sigma_{1}$, 
this would imply 
$|\phi_{z}|=|\bar{\phi}_{z}|=|\phi_{\bar{z}}|$ and hence 
the Jacobian at $p$ is zero which contradicts the fact 
that the Jacobian is non zero everywhere since $\phi$ is a diffeomorphism. 
Furthermore, $q \equiv 0$ iff $\phi$ is conformal. 
\end{example}
\subsection{The notion of a harmonic vector field}
\label{chapter3section2}
We introduce the notion of a \textit{harmonic vector field} on a Riemannian manifold $M$ which 
is regarded as the infinitesimal generator of local harmonic 
diffeomorphisms. Note that some sources use the term \textit{harmonic vector field} 
to mean
vector fields which have harmonic associated 1-form \cite{yano}
and vector fields as sections of the tangent 
bundle with \textit{lift metrics} \cite{nauhaud}.
Let $U$ be an open subset of a Riemannian manifold $M$. Let 
$\{\phi_{t}\}_{t \in [0, \epsilon)}$ be a 
smooth family of smooth  maps
$$\phi_{t} : U \longrightarrow M$$
where $\phi_{0}$ is the inclusion. Then $\xi= \frac{d \phi_{t}}{dt} \vert_{t=0}$ is 
a vector field on $U$.

\begin{defn}[Harmonic vector field]
\label{defn12}
 \normalfont The vector field $\xi$ on $U$ is harmonic if there exists a 
 smooth family of smooth maps
 $\{\phi_{t}: U \longrightarrow M\}_{t \in [0, \epsilon)}$ which satisfies 
 the following:
 \begin{enumerate}
\item $\phi_{0}$ is the inclusion map,
\item $\displaystyle\frac{d \phi_{t}}{dt}\Big\vert_{t=0} = \xi$,
\item $\forall x \in U:~\displaystyle\frac{d}{dt}\Big\vert_{t=0}~\tau(\phi_{t})(x)=0\,.$ 
\end{enumerate}
\end{defn}
\begin{remark}
 \normalfont
  Given $\xi$ we can always find the family 
  $\{ \phi_{t} \}_{t \in [0, \epsilon]}$ satisfying (1) and (2) in 
  Definition \ref{defn12}.
\end{remark}
\begin{remark}
 \normalfont
 The choice of $\{\phi_{t}\}_{t \in [0, \epsilon)}$ is unimportant 
  in (3) in Definition \ref{defn12}.
\end{remark}
Here $\tau$ is the \textit{Eells-Sampson Laplacian} which has been introduced in 
(\ref{equ2}). 
Condition 3 in Definition \ref{defn12} can be expressed in co-ordinate 
form as:
\begin{equation}
\label{newharm}
\frac{d}{dt}\bigg|_{t=0}\Bigg(\sum_{i,j=1}^{m} g^{ij}(x)
 \Bigg( \frac{\partial^{2} \phi_{t}^{\alpha}}{\partial x^{i} \partial x^{j}} - 
 \sum_{k=1}^{m}
 \Gamma^{k}_{ij}(x) \frac{\partial \phi_{t}^{\alpha}}{\partial x^{k}} + \sum_{\beta, \gamma=1}^{m}
  \Gamma^{\alpha}_{\beta \gamma}(\phi_{t}(x)) \frac{\partial \phi_{t}^{\beta}}{\partial x^{i}}
  \frac{\partial \phi_{t}^{\gamma}}{\partial x^{j}} \Bigg) \Bigg)=0, \quad 1 \leq \alpha \leq m.
\end{equation}
Now, for each $1 \leq i \leq m$, 
$\nabla_{t} \frac{\partial \phi_{t}}{\partial x^{i}} = \nabla_{i} 
\frac{\partial \phi_{t}}{\partial t}$. Therefore, (\ref{newharm}) becomes
\begin{equation*}
\resizebox{0.97\hsize}{!}{$
\begin{split}
 \sum_{i,j=1}^{m} g^{ij}(x)
 \Bigg( \frac{\partial^{2}}{\partial x^{i} \partial x^{j}} 
 \bigg(\frac{d \phi_{t}^{\alpha}}{dt}\bigg|_{t=0}\bigg) - 
 \sum_{k=1}^{m}
 \Gamma^{k}_{ij}(x) \frac{\partial}{\partial x^{k}} 
 \bigg(\frac{d \phi_{t}^{\alpha}}{dt}\bigg|_{t=0}\bigg) 
 + \sum_{\beta, \gamma=1}^{m} 
 \frac{d}{dt}\bigg|_{t=0} \bigg( \Gamma^{\alpha}_{\beta \gamma}(\phi_{t}(x))
  \frac{\partial \phi_{t}^{\beta}}{\partial x^{i}}
  \frac{\partial \phi_{t}^{\gamma}}{\partial x^{j}} 
 \bigg)\Bigg)=0, 
  \end{split}
  $}
\end{equation*}
where $ 1 \leq \alpha \leq m$.
Since $\xi^{\alpha}= \frac{d \phi_{t}^{\alpha}}{dt}\big|_{t=0}$, we have 
\begin{equation}
\label{newharm1}
\begin{split}
 \sum_{i,j=1}^{m} g^{ij}(x)
 \Bigg( \frac{\partial^{2} \xi^{\alpha}}{\partial x^{i} \partial x^{j}}  - 
 \sum_{k=1}^{m}
 \Gamma^{k}_{ij}(x) \frac{\partial \xi^{\alpha}}{\partial x^{k}} 
 + \sum_{\beta, \gamma=1}^{m} \bigg( \big(\Gamma^{\alpha}_{\beta \gamma}\big)'(\phi_{0}(x)) \cdot \xi \bigg)
 \frac{\partial \phi_{0}^{\beta}}{\partial x^{i}} 
  \frac{\partial \phi_{0}^{\gamma}}{\partial x^{j}} \\
  +
  \Gamma^{\alpha}_{\beta \gamma}(\phi_{0}(x))
  \bigg(
  \frac{\partial \xi^{\beta}}{\partial x^{i}} 
  \frac{\partial \phi_{0}^{\gamma}}{\partial x^{j}} + 
  \frac{\partial \phi_{0}^{\beta}}{\partial x^{i}} 
  \frac{\partial \xi^{\gamma}}{\partial x^{j}} 
   \bigg)\Bigg)=0,
   \end{split}
\end{equation}
where $1 \leq \alpha \leq m$ and $\big(\Gamma^{\alpha}_{\beta \gamma}\big)'$ denotes the derivative of 
$\Gamma^{\alpha}_{\beta \gamma}$. Since $\phi_{0}: U \longrightarrow M$ is the inclusion map, 
we rewrite (\ref{newharm1}):
\begin{equation}
\label{newharm2}
\sum_{i,j=1}^{m} g^{ij}(x)
 \Bigg( \frac{\partial^{2} \xi^{\alpha}}{\partial x^{i} \partial x^{j}}  - 
 \sum_{k=1}^{m}
 \Gamma^{k}_{ij}(x) \frac{\partial \xi^{\alpha}}{\partial x^{k}} 
 + \sum_{\beta, \gamma=1}^{m} \bigg( \big(\Gamma^{\alpha}_{\beta \gamma}\big)'(x) \cdot \xi \bigg)
 \delta_{\beta i} \delta_{\gamma j}
  +
  \Gamma^{\alpha}_{\beta \gamma}(x)
  \bigg(
  \frac{\partial \xi^{\beta}}{\partial x^{i}} 
  \delta_{\gamma j} + 
  \delta_{\beta i}
  \frac{\partial \xi^{\gamma}}{\partial x^{j}} 
   \bigg)\Bigg)=0,
\end{equation}
where $1 \leq \alpha \leq m$, $\frac{\partial \phi_{0}^{\gamma}}{\partial x^{j}}= 
\delta_{\gamma j}$ and $\frac{\partial \phi_{0}^{\beta}}{\partial x^{i}}= 
\delta_{\beta i}$. 
\bigskip
\newline
We now assume that $M$ is $\hh$ with the standard hyperbolic metric $\textbf{g}_{\hh}$, 
coordinatized as an open subset of $\cc$. 
Rewriting (\ref{newharm2}), we get
\begin{equation}
\label{equ7}
\begin{split}
 \sum_{i,j=1}^{2} \textbf{g}_{\hh}^{ij}(x) \bigg( \frac{\partial^{2} \xi^{\alpha}}{\partial x^{i} \partial x^{j}} 
- \sum_{k=1}^{2} \Gamma^{k}_{ij}(x) \frac{\partial \xi^{\alpha}}{\partial x^{k}} 
+ \sum_{\beta, \gamma=1}^{2} \bigg( \big(\Gamma^{\alpha}_{\beta \gamma}\big)'(x) \cdot \xi \bigg)
 \delta_{\beta i} \delta_{\gamma j} \\ 
+
\sum_{\beta, \gamma=1}^{2} \Gamma^{\alpha}_{\beta \gamma}(x)\bigg( 
\frac{\partial \xi^{\beta}}{\partial x^{i}} \delta_{\gamma j}+\delta_{\beta i} 
\frac{\partial \xi^{\gamma}}{\partial x^{j}}\bigg) \bigg)=0, 
\end{split}
\end{equation}
where $1 \leq \alpha \leq 2$. 
The Christoffel symbols $\Gamma_{11}^{1}$, $\Gamma_{22}^{1}$, $\Gamma_{12}^{2}$ and $\Gamma_{21}^{2}$ 
for $\textbf{g}_{\hh}$ vanish. Also $g_{\hh}^{11}=g_{\hh}^{22}=y^{2}$ and 
$g_{\hh}^{12}=g_{\hh}^{21}=0$.
Hence (\ref{equ7}) simplifies to:
\begin{equation}
\label{equ8}
\begin{split}
y^{2}\frac{\partial^{2} \xi^{\alpha}}{\partial^{2} x} + 
y^{2}\frac{\partial^{2} \xi^{\alpha}}{\partial^{2} y}-\bigg(y^{2} \Gamma^{2}_{11} 
\frac{\partial \xi^{\alpha}}{\partial y}+ y^{2}
\Gamma^{2}_{22} \frac{\partial \xi^{\alpha}}{\partial y}\bigg) + 
y^{2} \big(\big(\Gamma^{\alpha}_{11}\big)'(x) \cdot \xi \big)+ 
y^{2} \big( \big(\Gamma^{\alpha}_{22}\big)'(x) \cdot \xi \big) \\
+ y^{2}\bigg( \Gamma^{\alpha}_{11}\bigg(\frac{\partial \xi^{1}}{\partial x}+
\frac{\partial \xi^{1}}{\partial x}\bigg)+
\Gamma^{\alpha}_{12}\bigg(0+\frac{\partial \xi^{2}}{\partial x}\bigg)+
\Gamma^{\alpha}_{21}\bigg(\frac{\partial \xi^{2}}{\partial x}+0\bigg) \bigg) \\
+ y^{2} \bigg( \Gamma^{\alpha}_{12}\bigg(\frac{\partial \xi^{1}}{\partial y}+0\bigg)+
\Gamma^{\alpha}_{21}\bigg(0+\frac{\partial \xi^{1}}{\partial y}\bigg)+
\Gamma^{\alpha}_{22}\bigg(\frac{\partial \xi^{2}}{\partial y}+
\frac{\partial \xi^{2}}{\partial y}\bigg) \bigg)=0; 
\quad 1 \leq \alpha \leq 2.
\end{split}
\end{equation}
The other Christoffel symbols for $\textbf{g}_{\hh}$ are given as follows:
$$\Gamma^{1}_{12}=\Gamma^{1}_{21}= \Gamma^{2}_{22}= -\frac{1}{y}, \quad \Gamma^{2}_{11}= \frac{1}{y}. $$
Substituting these values into (\ref{equ8}), we obtain the following two 
equations which describe the conditions for $\xi$ to be a \textit{harmonic vector field} on $U$:
\begin{equation}
\label{equ9}
\xi^{1}_{xx}+\xi^{1}_{yy}- \frac{2}{y} ( \xi^{2}_{x}+\xi^{1}_{y})=0
\end{equation}
\begin{equation}
\label{equ10}
\xi^{2}_{xx}+\xi^{2}_{yy}+ \frac{2}{y} (\xi^{1}_{x}-\xi^{2}_{y})=0
\end{equation}
If the flow $(\phi_t)$ and the vector field $\xi$ are related as above, 
then we can describe $\phi_{t}$ up to the first order in terms of $\xi$ using 
the standard coordinates in $\mathbb{H}^{2}\subset\cc$:
\begin{equation*}
\phi_{t}(p) \approx p+t \xi(p)
\end{equation*}
(for $p\in U$ and sufficiently small $t$). 
We define a family of Riemannian metrics on $U$ as follows:
\begin{equation}
\label{equ11}
t \longmapsto \rho_{t} = \phi_{t}^{\ast} \textbf{g}_{\hh}
\end{equation}
More precisely the map in (\ref{equ11}) can be viewed as
\begin{equation}
\label{equ12}
t  \longmapsto (D \phi_{t} : T_{p} U \longrightarrow T_{\phi_{t}(p)} \mathbb{H}^{2})^{\ast} \textbf{g}_{\hh}
\end{equation}
(\ref{equ12}) to the \textit{first order} can be expressed as follows:
\begin{equation*}
t \longmapsto (\text{Id}+t\cdot d\xi : T_{p} U \longrightarrow T_{\phi_{t}(p)} 
\mathbb{H}^{2})^{\ast} \textbf{g}_{\hh},
\end{equation*}
where $d\xi$ is the derivative of $\xi$ (where $\xi$ is viewed as a smooth map 
$\cc \longrightarrow \cc$) at $p$, and it is an $\rr$-linear map.
Using the \textit{first order} approximation, $\rho_{t}$
is given as:
\begin{equation*}
\begin{split}
\rho_{t} & \approx (\text{Id}+t\cdot d\xi)^{T} (\textbf{g}_{\hh}+t \textbf{g}'_{\hh}(\xi)
(\text{Id}+t\cdot d\xi) \\
         & \approx \textbf{g}_{\hh} + t\cdot d\xi^{T} \textbf{g}_{\hh} + 
         t \textbf{g}'_{\hh}(\xi) + t\cdot d\xi \textbf{g}_{\hh} \\
         & \approx \textbf{g}_{\hh} + (t\cdot d\xi^{T} +t\cdot d\xi) \textbf{g}_{\hh} + 
         t \textbf{g}'_{\hh}(\xi)
\end{split}
\end{equation*}
In the above expression, $d\xi^{T}$ denotes the transpose of $d\xi$ when 
written in the local coordinates.
Calculating 
\[ \frac{d \rho_{t}}{dt} \Big\vert_{t=0} \]
gives us a section of $S^{2}(TU)$, the vector bundle of 
(real) symmetric bilinear forms on $TU$ and this is denoted by 
$\mathcal{L}_{\xi}\textbf{g}_{\hh}$, the Lie 
derivative of $\textbf{g}_{\hh}$ w.r.t $\xi$. Therefore, 
\begin{equation}
\label{equ13}
 \mathcal{L}_{\xi}\textbf{g}_{\hh} = (d\xi^{T} +d\xi) \textbf{g}_{\hh} 
 + \textbf{g}'_{\hh}(\xi)
\end{equation}
in our preferred coordinates. 
Now, to obtain a local expression for 
$\mathcal{L}_{\xi}\textbf{g}_{\hh} \in \Gamma(S^{2}(TU))$, 
we represent $d\xi$ by the following matrix 
\begin{equation*}
 d\xi=
 \begin{bmatrix}
  \xi^{1}_{x} & \xi^{1}_{y} \\
  \xi^{2}_{x} & \xi^{2}_{y}
 \end{bmatrix}
\end{equation*}
Using the above expression for $d\xi$, the right-hand side of (\ref{equ13}) can be represented as
\begin{equation*}
\begin{aligned}
 \mathcal{L}_{\xi}\textbf{g}_{\hh}  & =  \frac{1}{y^{2}}\begin{bmatrix}
                      2\xi^{1}_{x} & \xi^{1}_{y}+\xi^{2}_{x} \\
                      \xi^{2}_{x}+\xi^{1}_{y} & 2\xi^{2}_{y}
                      \end{bmatrix}+
                      \begin{bmatrix}
                     \frac{-2}{(\xi^{2})^{3}} & 0 \\
                     0 & \frac{-2}{(\xi^{2})^{3}}
                     \end{bmatrix} \\
                 & =  \underbrace{\frac{1}{y^{2}}\begin{bmatrix}
                      \xi^{1}_{x} - \xi^{2}_{y} & \xi^{1}_{y} +\xi^{2}_{x} \\
                      \xi^{1}_{y} +\xi^{2}_{x} &   \xi^{2}_{y} - \xi^{1}_{x}
                     \end{bmatrix}}_{\mathrm{TF}} +
                     \frac{1}{y^{2}}\begin{bmatrix}
                      \xi^{1}_{x}+\xi^{2}_{y} & 0 \\
                      0 & \xi^{1}_{x}+\xi^{2}_{y}
                     \end{bmatrix} +
                     \begin{bmatrix}
                     \frac{-2}{(\xi^{2})^{3}} & 0 \\
                     0 & \frac{-2}{(\xi^{2})^{3}}
                     \end{bmatrix}
\end{aligned}
\end{equation*}
Recall from \textbf{\cref{thespaceofhqd}} that
the bundle $S^{2}(TU)$ of (real) symmetric bilinear 
forms on $TU$ splits into
1-dimensional real vector subbundle spanned by the 
everywhere nonzero section $\textbf{g}_{\hh}$ and the image of
 the embedding (recall (\ref{quadmap}))
$$ \psi: \mathrm{hom}_{\mathbb{C}} (TU \otimes_{\mathbb{C}} TU, \mathbb{C}) 
\longrightarrow S^{2}(TU),$$
where $\psi(q)$ is the real part of $q=fdz^{2}$. In particular,
$\psi\big((\mathcal{L}_{\xi}\textbf{g}_{\hh})^{(2, 0)}\big)$ 
is the real part of 
$(\mathcal{L}_{\xi}\textbf{g}_{\hh})^{(2, 0)}$. 
Using the above splitting it is straightforward to check that the 
trace-free component of $\mathcal{L}_{\xi}\textbf{g}_{\hh}$ 
is $\psi\big((\mathcal{L}_{\xi}\textbf{g}_{\hh})^{(2, 0)}\big)=\psi(fdz^2)$ 
where $f(z) = \mathrm{TF}_{11}- \iota \mathrm{TF}_{12}$. Notice that 
\begin{equation}
 \label{similarwolpertequation}
 \overline{f(z)} =  \mathrm{TF}_{11}+ \iota \mathrm{TF}_{12}=
\frac{2}{y^{2}} \frac{\partial \xi}{\partial \bar{z}} = \frac{-8}{(z-\bar{z})^{2}} \frac{\partial \xi}{\partial \bar{z}}.
\end{equation}
Furthermore, (\ref{similarwolpertequation}) is equivalent (upto to a constant factor) 
to the following \textit{potential equation} (see Appendix \ref{genesis})
described by S. Wolpert in his paper \cite{wolpert}
 \begin{equation}
\label{wolpertequation}
\overline{f(z)} = \frac{1}{(z-\bar{z})^{2}} \frac{\partial \xi}{\partial \bar{z}}.
\end{equation}
Moreover, (\ref{equ9}) and (\ref{equ10}) are precisely the conditions that the 
corresponding quadratic differential $(\mathcal{L}_{\xi}\textbf{g}_{\hh})^{(2, 0)}$ 
is holomorphic, i.e., $f$ is holomorphic. Therefore,
we can summarize our discussion as follows: 
\begin{prop}
\label{thm2}
$\xi$ is a harmonic vector field on $U$ iff the quadratic differential 
$(\mathcal{L}_{\xi}\textbf{g}_{\hh})^{(2, 0)}$ associated with 
it is holomorphic. In the standard coordinates, $(\mathcal{L}_{\xi}\textbf{g}_{\hh})^{(2, 0)} = fdz^{2}$ where 
$\overline{f(z)} = \frac{-8}{(z-\bar{z})^{2}} \frac{\partial \xi}{\partial \bar{z}}$.
\end{prop}
\begin{remark}
 \normalfont Proposition \ref{thm2} is an infinitesimal version of Lemma 1.1 in 
 \cite{jost1} and Example \ref{harmimplieshol}. 
 In fact the statement in \cite{jost1} is more general since it applies to harmonic maps between oriented 
  2-dimensional Riemannian manifolds.
\end{remark}

\begin{coro}
\label{rem1}
Every holomorphic vector field on $U$ is harmonic. 
\end{coro}
\begin{proof}
Let $\xi$ be a holomorphic vector field on $U$. Then 
$\mathcal{L}_{\xi}\textbf{g}_{\hh}$ in (\ref{equ13}) has diagonal form and 
therefore $(\mathcal{L}_{\xi}\textbf{g}_{\hh})^{(2, 0)}$, the trace-free part of 
$\mathcal{L}_{\xi}\textbf{g}_{\hh}$, is zero. 
 \hfill \qedsymbol
\end{proof}
\subsubsection{Constructing harmonic vector fields on $U \subseteq \hh$}
\label{harmonicexplicit}
\begin{theorem} \label{sess}
Let $\mathcal{HOL}$ denote the sheaf of holomorphic vector fields on $\hh$, $\mathcal{HARM}$ denote the sheaf of harmonic vector fields on $\hh$ and 
$\mathcal{HQD}$ denote the sheaf of holomorphic quadratic differentials on $\hh$. 
Then the following sequence of sheaves
\begin{equation}
\label{shortexact}
\xymatrix{ 
\mathcal{HOL} \ar[r]^-{\alpha} &
\mathcal{HARM} \ar[r]^-{\beta} &
\mathcal{HQD} 
}
\end{equation}
is a short exact sequence of sheaves on $\hh$. In (\ref{shortexact}), $\alpha$ 
is the inclusion map and $\beta$ is given by the formula in Proposition \ref{thm2}. 
\end{theorem}
Before we prove Theorem \ref{sess}, we discuss the following result 
by S. Wolpert 
\cite[Section 2]{wolpert}: let $\eta$ be the vector field on $\hh$ given by $\eta(z)=(1,0)$ everywhere. 
Given a holomorphic quadratic differential 
$q=f(z)dz^{2}$ on $\hh$, there exists a global solution $\xi$ 
of the potential 
equation $\frac{\partial \xi}{\partial \bar{z}} = (z-\bar{z})^{2} \overline{f(z)}$ 
(see (\ref{wolpertequation})) and an explicit formula 
for $\xi$ is given as: 
\begin{equation}
 \label{wolpertformula}
 \xi(z)= \Bigg( \overline{\int_{w}^{z} (\bar{z}-\zeta)^{2} f(\zeta) d\zeta} \Bigg) \eta(z), 
\end{equation}
where $w \in \hh$ is fixed and $\zeta, z \in \hh$. The formula for $\xi$ in 
(\ref{wolpertformula}) is path independent since the integrand is holomorphic. 
\bigskip
\newline
\paragraph{\textit{Proof of Theorem \ref{sess}:}}
Exactness at the term $\mathcal{HARM}$ in (\ref{shortexact}) follows 
from Theorem \ref{thm2} and Corollary \ref{rem1}. Now, Let $q=f(z)dz^{2}$ be defined in a 
neighborhood $V$ of $w \in \hh$, where 
 $w \in \hh$ is fixed. To prove the local surjectivity 
 of $\beta$ we have to get a solution for a harmonic vector field $\xi$ whose associated 
 holomorphic quadratic differential is $q$ in
 a possibly smaller neighborhood $U \subset V$ of $w$. It is clear that 
 (\ref{wolpertformula}) gives the required solution for $\xi$ upto a constant factor. 
 \hfill \qedsymbol

\begin{coro}
\label{coroharmonic}
 If a sequence of harmonic vector fields defined on an open set $U$ in $\hh$ 
 converges 
 uniformly on compact subsets of $U$, and if all of them determine the 
 same holomorphic quadratic differential $q$ on $U$, then the limit vector field 
 is again harmonic and still determines the same holomorphic quadratic 
 differential $q$ on $U$. 
\end{coro} 
We will now describe a more pedestrian approach to finding harmonic vector fields with 
prescribed holomorphic quadratic differential. This has certain advantages over Wolpert's 
formula, as we will see in \textbf{\cref{extendharmonic}}. 
First, we give an explicit expression for a 
harmonic vector field on $U \subset \hh$ 
whose associated
holomorphic quadratic differential $q$ is given. 
\begin{lemma}
\label{lemma2}
Let $U$ be an open subset of $\mathbb{H}^{2}$ with the usual hyperbolic metric. 
Let $\eta$ be the vector field on $U$ given by $\eta(z)=(1,0)$ everywhere:
vectors parallel to the real axis, pointing left to right, of euclidean length 1. 
Let $f$ be a holomorphic function on $U$. 
The quadratic differential $q$ 
associated to the vector field $\xi=y^{n}f \eta$
is represented as:
\begin{equation}
\label{equ14}
 q=-n \iota y^{n-3} \overline{f} dz^{2}, \quad n \geq 3.
\end{equation}
\end{lemma}
\begin{proof}
We use the recipe in Proposition \ref{thm2} to prove 
the Lemma. And it suffices to prove for $n=3$. From (\ref{similarwolpertequation}), we have
\begin{equation*}
 \begin{split}
  \frac{\partial \xi}{\partial \bar{z}} & = \frac{\partial }{\partial \bar{z}} (y^{3}f) =
  \frac{\partial }{\partial \bar{z}} \bigg( \frac{(z- \bar{z})^{3}}{-8 \iota} f \bigg) =
  \frac{3 (z-\bar{z})^{2}}{8 \iota} f = \frac{3 \iota (z-\bar{z})^{2}}{-8} f = 
  \frac{(z-\bar{z})^{2}}{-8} (\overline{-3\iota \bar{f}}), 
 \end{split}
\end{equation*}
so that $q= -3\iota\bar{f}dz^{2}$. \hfill \qedsymbol
\end{proof}
\bigskip
Using Lemma \ref{lemma2}, we can find an explicit expression for a harmonic vector 
field $\xi$ on $\hh$ whose associated holomorphic
quadratic differential is 
$$q= z^{n} dz^{2} (n \geq 0)$$ using the obvious expression 
$z^{n} = (\overline{z}+2 \iota y)^{n} = \sum_{k=0}^{n} 
\binom{n}{k} (2 \iota y)^{n-k} \bar{z}^{k}$.
\begin{lemma}
\label{lemma3}
An explicit expression for a 
harmonic vector field $\xi$ on $ \hh$ whose associated holomorphic 
quadratic differential is $q= f(z) dz^{2}$, where
$f(z) = z^{n}$, for some $n \geq 0$ (a holomorphic function on $ \hh$), is given as:
\begin{equation*}
\begin{split}
 \xi(z) & = \Bigg( \sum_{k=0}^{n} \binom{n}{k} \frac{(-2)(-2 \iota)^{n-k-1}}{n-k+3} y^{n-k+3}z^{k} \Bigg) \eta(z) \\
        & = \Bigg( \sum_{k=0}^{n} \binom{n}{k} (-2)(-2 \iota)^{n-k-1} \Big(\int_{0}^{y} \zeta^{n-k+2} d \zeta \Big) z^{k} \Bigg) \eta(z) \\
        & = \Bigg( \int_{0}^{y} \bigg( \sum_{k=0}^{n} \binom{n}{k} (-2)(-2 \iota)^{n-k-1} \zeta^{n-k+2} z^{k} \bigg) d \zeta \Bigg) \eta(z) \\
        & = \Bigg( \int_{0}^{\Im(z)} -\iota \zeta^{2} (z-2 \iota \zeta)^{n} d \zeta \Bigg) \eta(z)
\end{split}
\end{equation*}
\end{lemma} 

\begin{lemma}
\label{lemma4}
 An explicit expression for a harmonic vector field $\xi$ on $U \subset \hh$ whose 
 associated holomorphic quadratic differential is $q= f(z) dz^{2}$, where
$f(z) = (z-a)^{n} (n \geq 0)$ is a holomorphic function on $U \subset \hh$ and $a \in \hh$ fixed, is given as:
\begin{equation}
\label{explicit}
\begin{split}
 \xi(z) & = \Bigg( \sum_{k=0}^{n} \binom{n}{k} (-\bar{a})^{n-k} \bigg( \int_{0}^{\Im(z)} -\iota \zeta^{2} (z-2 \iota \zeta)^{k} d \zeta \bigg) \Bigg) \eta(z) \\
        & = \Bigg( \int_{0}^{\Im(z)} -\iota \zeta^{2} \bigg( \sum_{k=0}^{n} \binom{n}{k} (-\bar{a})^{n-k} (z-2 \iota \zeta)^{k} \bigg) d \zeta \Bigg) \eta(z) \\
        & = \Bigg(  \int_{0}^{\Im(z)} -\iota \zeta^{2} \big(z-\bar{a}-2 \iota \zeta \big)^{n} d \zeta \Bigg) \eta(z) \\
        & = \Bigg( \int_{0}^{\Im(z)} -\iota \zeta^{2} \overline{f(\bar{z}+2 \iota \zeta)} d \zeta \Bigg) \eta(z).
\end{split}
\end{equation}
\end{lemma}
\bigskip
\paragraph{\textit{Another Proof of Theorem \ref{sess}:}}
Exactness at the term $\mathcal{HARM}$ in (\ref{shortexact}) follows 
from Theorem \ref{thm2} and Corollary \ref{rem1}. Let $q=f(z)dz^{2}$ be defined in a 
neighborhood $V$ of $a \in \hh$, where 
 $a \in \hh$ is fixed. To prove the local surjectivity 
 of $\beta$ we have to get a solution for a harmonic vector field whose associated 
 holomophic quadratic differential is $q$ in
 a possibly smaller neighborhood $U \subset V$ of $a$. Note that we can't use 
 the expression in (\ref{explicit}). As $\zeta$ runs from $0$ to $\Im(z)$, 
 $f(\bar{z}+2 \iota \zeta)$ does not even make sense when $\zeta=0$. 
 We try the following
 \begin{equation}
 \label{explicit1}
  \xi_{c}(z)= \Bigg(  \int_{c}^{\Im(z)} -\iota \zeta^{2} \overline{f(\bar{z}+2 \iota \zeta)} 
  d \zeta \Bigg) \eta(z),
 \end{equation}
where $c$ is any positive real number. 
\begin{figure}[h]
\centering
 \includegraphics[height=4cm]{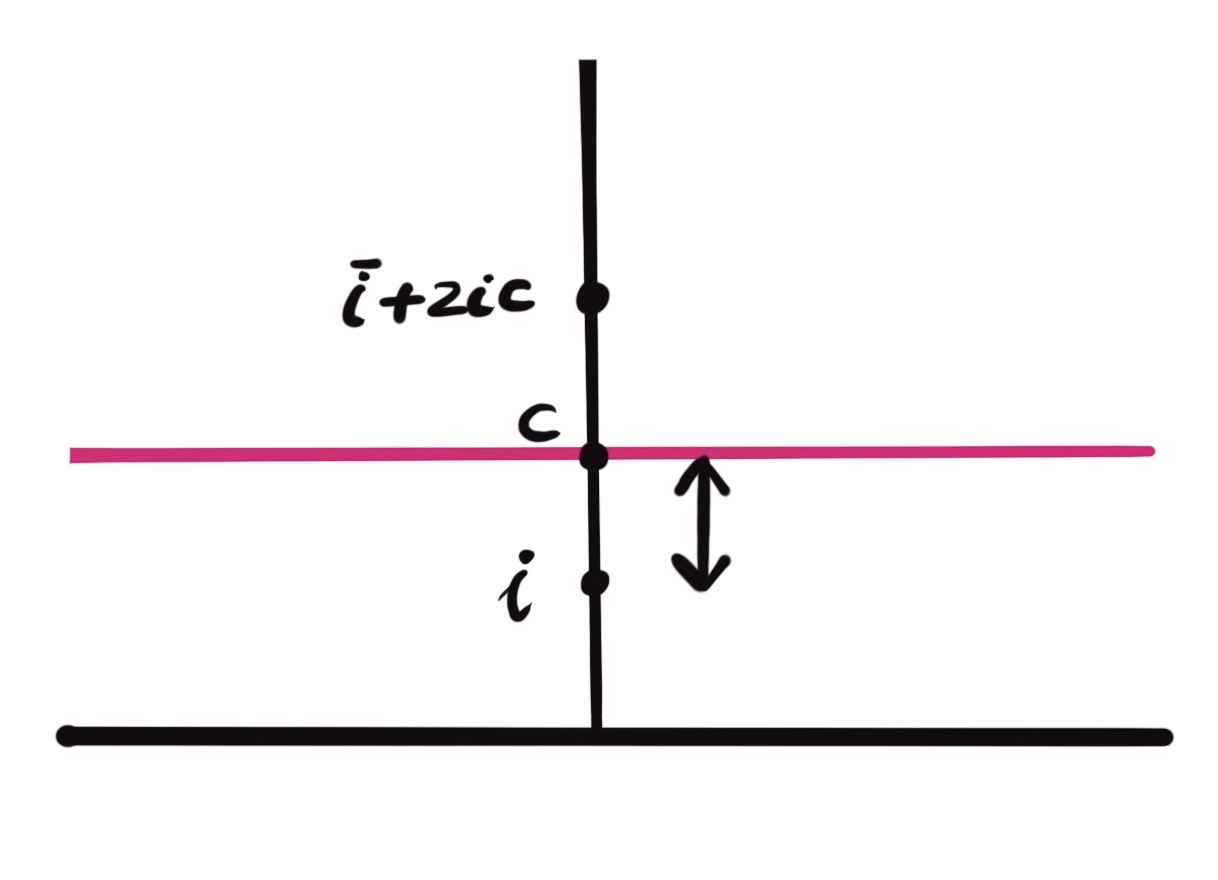}
 \caption{The expression for $\xi_{c}(\iota)$ defined along the 
 hyperbolic line joining $\iota$ and $c$}
 \label{pic6}
\end{figure}
But there is a caveat:
as $\zeta$ runs from $c$ to $\Im(z)$, $f(\bar{z}+2 \iota \zeta)$ may not be defined on 
the upper half plane 
since we are assuming that $f$ is defined only on $V \subset \hh$. By 
making the best 
possible choice of $c$ which is $\Im(a)$ in this case, 
we get the required solution as 
follows
\begin{equation}
\label{rightsol}
 \xi_{\Im(a)}(z)= \Bigg(  \int_{\Im(z)}^{\Im(a)} \iota \zeta^{2} \overline{f(\bar{z}+2 \iota \zeta)} d \zeta \Bigg) \eta(z),
\end{equation}
defined on
$$U=\{ z \in V | \bar{z}+2 \iota t \in V \hspace{2pt} \mathrm{for} \hspace{1pt} 
\mathrm{all} \hspace{2pt} t \in [\Im(z), \Im(a)] \}.$$
Evaluating the expression in (\ref{rightsol}) at $a$, we get
\begin{equation*}
\begin{split}
 \xi_{\Im(a)}(a) & = \Bigg( \int_{\Im(a)}^{\Im(a)} \iota \zeta^{2} 
 \overline{f(\bar{a}+2 \iota \zeta)} d \zeta \Bigg) \eta(a) \\
                 & = 0. 
 \end{split}
 \end{equation*}\hfill \qedsymbol
\begin{remark}
\label{wecanextend}
\normalfont
Let $q$ be a quadratic differential which is defined 
everywhere on $\hh$ and is bounded in the hyperbolic metric $\textbf{g}_{\hh}$, i.e.,
 $$\Arrowvert q \Arrowvert_{\textbf{g}_{\hh}} = \arrowvert f(z) \arrowvert  
 \Arrowvert dz^{2} \Arrowvert_{\textbf{g}_{\hh}} \leq D,$$
where $\Arrowvert dz^{2} \Arrowvert_{\textbf{g}_{\hh}}= \Im(z)^{2}$ and 
$D$ is a positive real number. Note that $\xi_{c}$ in (\ref{explicit1}) has a 
continuous extension on $\rr$. 
In other words, for $z$ such that 
$\Im(z)=0$, we define
\begin{equation}
 \label{limitimprop}
 \xi_{c}(z) = \lim_{\epsilon \rightarrow 0}\Bigg(  \int_{\epsilon}^{c} \iota \zeta^{2} 
\overline{f(z+2 \iota \zeta)} 
  d \zeta \Bigg) \eta(z).
\end{equation}
 To prove that the above limit exists, we use 
 the Cauchy criterion of convergence of improper integrals, 
 \begin{equation*}
  \begin{split}
   \bigg | \int_{\epsilon_{1}}^{\epsilon_{2}} \iota \zeta^{2} 
   \overline{f(\bar{z}+2 \iota \zeta)} d \zeta \bigg| & \leq 
   \int_{\epsilon_{1}}^{\epsilon_{2}}  \zeta^{2} \frac{D}{4 \zeta^{2}} d\zeta \\
   & = \frac{D}{4} (\epsilon_{2} -\epsilon_{1}).
   \end{split}
\end{equation*}
From the above estimate, it is clear that the limit in (\ref{limitimprop}) exists. 
 \end{remark}
\begin{theorem} \label{thmglobharmvf} Let $q=f(z)dz^{2}$ be a 
holomorphic quadratic differential on $\hh$. Suppose that 
$q$ satisfies the following boundedness
conditions 
\begin{enumerate}
 \item $q$ is bounded in the hyperbolic metric $\textbf{g}_{\hh}$, i.e.
\begin{equation}
\label{equ25}
 \Arrowvert q \Arrowvert_{\textbf{g}_{\hh}} = \arrowvert f(z) \arrowvert  
 \Arrowvert dz^{2} \Arrowvert_{\textbf{g}_{\hh}} \leq D,
\end{equation}
where $\Arrowvert dz^{2} \Arrowvert_{\textbf{g}_{\hh}}= \Im(z)^{2}$ and $D$ is a positive real number.
\item The first and second covariant derivative of $q$ w.r.t the linear connection $\nabla$ 
on $T^{\ast} \hh \otimes_{\cc} T^{\ast} \hh$,
are bounded in the hyperbolic metric $\textbf{g}_{\hh}$. 
\end{enumerate}
Then there exists a harmonic vector field $\xi^{\mathrm{reg}}$ on $\hh$ such 
that $\beta(\xi^{\mathrm{reg}})=q$, where $\beta$ is introduced in 
Theorem \ref{sess}. An explicit formula is 
\begin{equation}
\label{equ62}
  \xi^{\mathrm{reg}}(z) = \lim_{c \to \infty} \Bigg( \xi_{c}(z) - 
  \bigg( \xi_{c}(\iota) + \frac{\partial \xi_{c}}{\partial z}\bigg|_{z=\iota}\cdot 
  (z-\iota)\bigg) \Bigg),
  \end{equation}
where 
$$\xi_{c}(z)= \Bigg(  \int_{\Im(z)}^{c} \iota \zeta^{2} 
\overline{f(\bar{z}+2 \iota \zeta)} d \zeta \Bigg) \eta(z)$$ and 
$c$ is a positive real number.
\end{theorem} 

\begin{remark}
 \normalfont
 We have introduced a simple terminology $\mathrm{reg}$ short 
for ``regularisation'' to characterise our required harmonic vector field. 
\end{remark}

\begin{remark}
\normalfont
 The boundedness conditions on $q$ in the above theorem are satisfied if 
 $q$ is invariant under the action of a discrete cocompact subgroup 
 $\Gamma$ of $\mathrm{PSL}(2, \rr)$, i.e.,
 $$ f(\gamma(z)) \gamma'(z)^{2}=f(z), \quad z \in \hh, \quad \forall \gamma \in \Gamma.$$
\end{remark}

\begin{remark}
\label{conditioncov1}
 \normalfont
 In Theorem \ref{thmglobharmvf}, $\nabla$ is a first order linear differential operator
\begin{equation*}
\label{equ91}
 \mathscr{A}^{0}(\hh, T^{\ast}\hh \otimes_{\cc} T^{\ast}\hh) \longrightarrow \mathscr{A}^{1}(\hh, T^{\ast}\hh \otimes_{\cc} T^{\ast}\hh),
\end{equation*}
where on the L.H.S. we have sections of the vector bundle $T^{\ast}\hh \otimes_{\cc} T^{\ast}\hh \longrightarrow \hh$ and on the R.H.S we have
the space of $T^{\ast}\hh \otimes_{\cc} T^{\ast}\hh$-valued 1-forms, i.e., sections of the vector bundle
$\mathrm{\mathbf{hom}}(T\hh, T^{\ast}\hh \otimes_{\cc} T^{\ast}\hh).$ 
Recall that the Levi-Civita connection $\nabla$ of the hyperbolic plane
can be extended complex linearly to the complexification of the tangent and cotangent bundles - 
$(T\hh)^{c}$ and $(T^{\ast}\hh)^{c}$ - 
of the plane and their tensor products, and then decomposed as 
\begin{equation*}
 \nabla = \nabla_{\frac{\partial}{\partial z}} \oplus \nabla_{\frac{\partial}{\partial \bar{z}}}.
\end{equation*}
Recall the discussion just before Example \ref{harmexamplenice}. We view 
$\frac{\partial}{\partial z}$ and $\frac{\partial}{\partial \bar{z}}$ as sections of the 
complexified tangent bundle $(T\hh)^{c}$, and $dz$ and $d \bar{z}$ as sections of the complexified 
cotangent bundle $(T^{\ast}\hh)^{c}$. Furthermore, 
$dz\big( \frac{\partial}{\partial z}\big) = 1$ and 
$dz \big( \frac{\partial}{\partial \bar{z}}\big) =0$. 
For example, applied to functions $f: \hh \longrightarrow \cc$, 
we have $\nabla_{\frac{\partial}{\partial z}} f = f_{z} dz$ and 
$\nabla_{\frac{\partial}{\partial \bar{z}}} f = f_{\bar{z}} d \bar{z}$. 
Now, for the hyperbolic plane with the hyperbolic metric 
$\textbf{g}_{\hh}=\rho^{2}dz d\bar{z}$, where $\rho(z)=1 / \Im(z)$,
 we get the following: 
 \begin{equation}
 \label{attempt2}
 \begin{split}
 \nabla \frac{\partial}{\partial z} = 
 \frac{2 \rho_{z}}{\rho} dz \otimes \frac{\partial}{\partial z}, \quad 
 \nabla dz = dz \otimes \nabla_{\frac{\partial}{\partial z}} dz = 
 - \frac{2 \rho_{z}}{\rho} dz \otimes dz \\
 \nabla_{ \frac{\partial}{\partial \bar{z}}}  \frac{\partial}{\partial z}=0, \quad 
 \nabla_{\frac{\partial}{\partial z}} \frac{\partial}{\partial z} = 
 \frac{2 \rho_{z}}{\rho} \frac{\partial}{\partial z}. 
 \end{split}
\end{equation}
Equations in (\ref{attempt2}) are taken from \cite{mlee}. 
To get boundedness conditions on $f_{z}$ and $f_{zz}$ from boundedness conditions on $q$ and on 
the first and second covariant derivative of 
$q=fdz^{2}$, i.e.,
\begin{equation}
\begin{split}
 \lVert q \lVert_{\textbf{g}_{\hh}} & \leq D \\
 \lVert \nabla q \lVert_{\textbf{g}_{\hh}} & \leq D_{1} \\
 \lVert \nabla^{2} q \lVert_{\textbf{g}_{\hh}} & \leq D_{2},
\end{split}
\end{equation}
$D_{1}$ and $D_{2}$ are positive real numbers, we need to compute $\nabla q$ and $\nabla^{2} q$. 
Consider the first covariant derivative of $q$ w.r.t $\nabla$:
\begin{equation}
\label{firstcovariant}
 \begin{split}
  \nabla f dz^{2} & = \nabla (f) dz^{2} + f \nabla(dz \otimes dz) \\
                  & = f_{z} dz^{3} + f_{\bar{z}} d \bar{z} \otimes dz^{2} + 
                  f (\nabla dz \otimes dz + dz \otimes \nabla dz) \\
                  & = f_{z} dz^{3} + 0 + f \cdot -\frac{4 \rho_{z}}{\rho} dz^{3},
 \end{split}
\end{equation}
where the last equality follows from (\ref{attempt2}) and the fact that 
$f$ is a holomorphic function. 
From (\ref{firstcovariant}), we have 
$$ \lVert f_{z} dz^{3} + f \cdot -
\frac{4 \rho_{z}}{\rho} dz^{3} \rVert_{\textbf{g}_{\hh}} \leq D_{1}$$
which implies
\begin{equation}
\label{firstcov}
 |f_{z}| \leq \frac{K_{1}}{\Im(z)^{3}},
\end{equation}
where $K_{1}$ is a positive constant that depends upon the bounds for $f$. 
Now, using (\ref{attempt2}) and (\ref{firstcovariant}) consider the 
second covariant derivative of $q$ w.r.t $\nabla$: 
\begin{equation}
\label{secondcov}
 \begin{aligned}
\nabla \big( f_{z} dz^{3} + f \cdot -\frac{4 \rho_{z}}{\rho} dz^{3}\big) & = 
  \nabla f_{z} dz^{3} + f_{z} \nabla(dz \otimes dz \otimes dz) + \nabla f \cdot 
  -\frac{4 \rho_{z}}{\rho} dz^{3} \\
  & \qquad  - f \cdot \nabla \bigg(\frac{4 \rho_{z}}{\rho} \bigg) dz^{3} 
  + f \cdot -\frac{4 \rho_{z}}{\rho} \nabla(dz \otimes dz \otimes dz) \\
  & = f_{zz}dz^{4}+f_{z\bar{z}}d\bar{z} \otimes dz^{3}+f_{z} \cdot -
  \frac{6 \rho_{z}}{\rho} dz^{4} + f_{z} \cdot -\frac{4 \rho_{z}}{\rho} dz^{4} \\
  & \qquad  + f_{\bar{z}} 
  \cdot -\frac{4 \rho_{z}}{\rho} d \bar{z} \otimes dz^{3} 
   -f \cdot \rho^{2} dz^{4} + f \cdot \frac{4 \rho_{z}}{\rho} \cdot 
  \frac{6 \rho_{z}}{\rho} dz^{4}  \\
  & = f_{zz}dz^{4} + 0 + f_{z} \cdot -\frac{6 \rho_{z}}{\rho} dz^{4} + 
  f_{z} \cdot -\frac{4 \rho_{z}}{\rho} dz^{4} + 0 \\
 & \qquad -f \cdot \rho^{2} dz^{4} + f \cdot \frac{24 \rho_{z}^{2}}{\rho^{2}} dz^{4} \\
 & = f_{zz}dz^{4} + f_{z} \cdot -\frac{10 \rho_{z}}{\rho} dz^{4}
  -f \cdot \rho^{2} dz^{4} + f \cdot \frac{24 \rho_{z}^{2}}{\rho^{2}} dz^{4}.
\end{aligned}
\end{equation}
From (\ref{secondcov}), the second covariant derivative of $q$ being bounded in the 
hyperbolic metric implies the following: 
\begin{equation}
\label{secondcov1}
 |f_{zz}| \leq \frac{K_{2}}{\Im(z)^{4}},
\end{equation}
where $K_{2}$ is a positive constant that depends upon the bounds for $f$ and $f_{z}$.  
\end{remark}
Before we begin with the proof of Theorem \ref{thmglobharmvf} which 
establishes the global surjectivity of $\beta$ in Theorem \ref{sess}, 
we discuss the following abortive attempts to get a (global) harmonic 
vector field on the whole upper half plane $\hh$.
\begin{remark}
\label{firstattempt}
\normalfont
Assume that $q$ is bounded in the hyperbolic metric, i.e.
\begin{equation*}
 \Arrowvert q \Arrowvert_{\textbf{g}_{\hh}} = \arrowvert f(z) \arrowvert  
 \Arrowvert dz^{2} \Arrowvert_{\textbf{g}_{\hh}} \leq D,
\end{equation*}
where $D$ is a positive real number.
We try to define 
\begin{equation}
\label{equ19}
\xi(z)= \Bigg(  \int_{\Im(z)}^{\infty} \iota \zeta^{2} \overline{f(\bar{z}+2 \iota \zeta)} d \zeta \Bigg) \eta(z) 
= \lim_{c \to \infty} \Bigg( \int_{\Im(z)}^{c} \iota \zeta^{2} \overline{f(\bar{z}+2 \iota \zeta)} d \zeta \Bigg) \eta(z),
\end{equation}
hoping that the above limit exists. In this case, we say 
that the improper integral in (\ref{equ19}) \textit{converges} and 
its value is that of the limit.
From the above mentioned boundedness condition on $q$ we get the following
\begin{equation}
\label{equ26}
 \arrowvert f(z) \arrowvert \leq \frac{D}{\Im(z)^{2}}, \forall z \in \hh.
\end{equation}
From the Cauchy criterion of convergence of improper integrals, 
the improper integral
$$\int_{\Im(z)}^{\infty} \iota \zeta^{2} \overline{f(\bar{z}+2 \iota \zeta)} d \zeta$$
in (\ref{equ19}) converges iff for every $\epsilon > 0$ there is a
$K \geq \Im(z)$ so that for all $A, B \geq K$ we have 
$$\bigg \arrowvert \int_{A}^{B} \iota \zeta^{2} \overline{f(\bar{z}+2 \iota \zeta)} d \zeta \bigg \arrowvert < \epsilon.$$
Using (\ref{equ26}), we have
\begin{equation}
\label{inequalities1}
 \begin{split}
  \bigg \arrowvert \int_{A}^{B} \iota \zeta^{2} \overline{f(\bar{z}+2 \iota \zeta)} d \zeta \bigg \arrowvert & \leq \int_{A}^{B}  \zeta^{2} \big \arrowvert \overline{f(\bar{z}+2 \iota \zeta)} \big \arrowvert d \zeta  \\
               & \leq  \int_{A}^{B} \frac{D \zeta^{2}}{(2 \zeta - \Im(z))^{2}} d \zeta 
 \end{split}
\end{equation}
Now, we assume that $A \geq \Im(z)$. Then the denomiator $(2 \zeta - \Im(z))^{2}$ in the second inequality 
in (\ref{inequalities1}) is atleast as big as $\zeta^{2}$. Rewriting (\ref{inequalities1}), we get 
\begin{equation*}
 \begin{split}
  \bigg \arrowvert \int_{A}^{B} \iota \zeta^{2} \overline{f(\bar{z}+2 \iota \zeta)} 
  d \zeta \bigg \arrowvert & \leq \int_{A}^{B} \frac{D \zeta^{2}}{ \zeta^{2}} d \zeta \\
               & = \int_{A}^{B} D d \zeta \\
               & = D(B-A).
 \end{split}
\end{equation*}
From the above estimate, there is no conclusion that limit in (\ref{equ19}) exists.
\end{remark}


\begin{remark}
\normalfont
Assume that both $q$ and its first covariant derivative w.r.t $\nabla$ 
are 
bounded in the hyperbolic metric $\textbf{g}_{\hh}$. From Remark 
\ref{conditioncov1} and (\ref{firstcov}), 
the covariant derivative of $q$ (w.r.t $\nabla$)
being bounded in the hyperbolic metric $\textbf{g}_{\hh}$ implies the following:
\begin{equation}
 \label{equ27}
 \arrowvert f_{z} \arrowvert \leq \frac{K_{1}}{\Im(z)^{3}},
\end{equation}
where $f_{z}$ denotes the first 
complex derivative of $f$, $f$ being a holomorphic function on $\hh$.
 We try to define 
\begin{equation}
\label{equ20}
  \xi(z)=\lim_{c \to \infty} (\xi_{c}(z)-\xi_{c}(\iota))
\end{equation}
hoping that the above limit exists. We view $\xi_{c}(\iota)$ as 
the zeroth order Taylor approximation of 
$\xi_{c}(z)$ at $z=\iota$. Moreover, 
$\xi_{c}(\iota)$ is a constant vector field, hence a holomorphic vector field, depending on $c$. Note that the expression in (\ref{equ20}) 
resembles the idea of Weierstrass in constructing the Weierstrass 
$\mathcal{P}$-function. Naively speaking, we want to compare the integral along a 
vertical hyperbolic line $\mathcal{L}_{1}$ joining some point $z$ to $\bar{z}+2 \iota c$ 
with the integral along a vertical hyperbolic line $\mathcal{L}_{2}$ joining $\iota$ to $(2c-1)\iota$. 
Infact, $\mathcal{L}_{1}$ and $\mathcal{L}_{2}$ are asymptotic lines in the 
hyperbolic plane $\hh$. 
Let's first spell out the expression $\xi_{c}(z)-\xi_{c}(\iota)$ on the R.H.S of 
(\ref{equ20}). 
\paragraph{Case I: $2c \geq 1 \geq \Im(z)$}
\begin{align}
   \xi_{c}(z)-\xi_{c}(\iota) & = \bigg( \int_{\Im(z)}^{c} \iota \zeta^{2} \overline{f(\bar{z}+2 \iota \zeta)}d \zeta-\int_{1}^{c} \iota \zeta^{2}\overline{f(\bar{\iota}+2 \iota \zeta)}d \zeta \bigg)  \eta(z) \nonumber \\
                             & = \bigg( \int_{\Im(z)}^{1} \iota \zeta^{2} \overline{f(\bar{z}+2 \iota \zeta)}d \zeta + \int_{1}^{c} \iota \zeta^{2}\overline{f(\bar{z}+2 \iota \zeta)} d \zeta -\int_{1}^{c} \iota \zeta^{2}\overline{f(\bar{\iota}+2 \iota \zeta)} d \zeta \bigg)  \eta(z) \nonumber  \\
                             & = \bigg( \underbrace{\int_{1}^{c} \iota \zeta^{2} \bigg( \overline{f(\bar{z}+2 \iota \zeta) - f(\bar{\iota}+2 \iota \zeta)} \bigg)d \zeta}_{I_{c}}- \int_{1}^{\Im(z)} \iota \zeta^{2}\overline{f(\bar{z}+2 \iota \zeta)} d \zeta \bigg) \eta(z) \label{esti}
  \end{align}
  \paragraph{Case II: $2c \geq \Im(z) \geq 1$}
  \begin{align}
   \xi_{c}(z)-\xi_{c}(\iota) & =  \bigg( \int_{\Im(z)}^{c} \iota \zeta^{2} \overline{f(\bar{z}+2 \iota \zeta)}d \zeta-\int_{1}^{c} \iota \zeta^{2}\overline{f(\bar{\iota}+2 \iota \zeta)}d \zeta \bigg)  \eta(z) \nonumber \\
                               & = \bigg( \int_{\Im(z)}^{c} \iota \zeta^{2} \overline{f(\bar{z}+2 \iota \zeta)}d \zeta-\int_{1}^{\Im(z)} \iota \zeta^{2}\overline{f(\bar{\iota}+2 \iota \zeta)} d \zeta -\int_{\Im(z)}^{c} \iota \zeta^{2}\overline{f(\bar{\iota}+2 \iota \zeta)} d \zeta \bigg)  \eta(z) \nonumber  \\
                               & = \bigg( \underbrace{\int_{\Im(z)}^{c} \iota \zeta^{2} \bigg( \overline{f(\bar{z}+2 \iota \zeta) - f(\bar{\iota}+2 \iota \zeta)} \bigg)d \zeta}_{II_{c}}- \int_{1}^{\Im(z)} \iota \zeta^{2}\overline{f(\bar{\iota}+2 \iota \zeta)} d \zeta \bigg) \eta(z) \label{equ21}
  \end{align}
  Since $ \int_{1}^{\Im(z)} \iota \zeta^{2}\overline{f(\bar{z}+2 \iota \zeta)} d \zeta$ and 
$ \int_{1}^{\Im(z)} \iota \zeta^{2}\overline{f(\bar{\iota}+2 \iota \zeta)} d \zeta$
  on the R.H.S of (\ref{esti}) and (\ref{equ21}) are independent
 of $c$ we only work with $I_{c}$ and $II_{c}$ to determine whether the limit in 
 (\ref{equ20}) exists or not.  Now if $A, B \geq c$, we have
 \begin{equation}
  \label{attemptnumber3}
  I_{B}-I_{A} = \int_{A}^{B} \iota \zeta^{2} \bigg( \overline{f(\bar{z}+2 \iota \zeta) 
  - f(\bar{\iota}+2 \iota \zeta)} \bigg)d \zeta, 
 \end{equation}
and 
\begin{equation}
\label{attemptnumber3a} 
 II_{B}-II_{A} = \int_{A}^{B} \iota \zeta^{2} \bigg( \overline{f(\bar{z}+2 \iota \zeta) 
 - f(\bar{\iota}+2 \iota \zeta)} \bigg)d \zeta.
\end{equation} 
Using (\ref{equ27}), we have the following estimate for (\ref{attemptnumber3}) and 
(\ref{attemptnumber3a})
\begin{equation}
   \label{inequalities2}
   \begin{aligned}
\bigg \arrowvert \int_{A}^{B} \iota \zeta^{2} \bigg( \overline{f(\bar{z}+2 \iota \zeta) 
  - f(\bar{\iota}+2 \iota \zeta)} \bigg)
d \zeta \bigg \arrowvert  
  & \leq \int_{A}^{B} \zeta^{2} \cdot \frac{K_{1}}{(2 \zeta - \Im(z))^{3}}
  \cdot |\bar{z} - \bar{\iota} | d \zeta, 
   \end{aligned}
\end{equation}
where the inequality in (\ref{inequalities2}) follows from (\ref{equ27}). 
Now, we assume that $A \geq \Im(z)$. 
Then the denomiator $(2 \zeta - \Im(z))^{3}$ in the inequality 
in (\ref{inequalities2}) is atleast as big as $\zeta^{3}$. Rewriting (\ref{inequalities2}), we get 
\begin{equation*}
   \begin{aligned}
\bigg \arrowvert \int_{A}^{B} \iota \zeta^{2} \bigg( \overline{f(\bar{z}+2 \iota \zeta) 
  - f(\bar{\iota}+2 \iota \zeta)} \bigg)
d \zeta \bigg \arrowvert  & \leq  |\bar{z} - \bar{\iota} | \int_{A}^{B} \zeta^{2} \cdot \frac{K_{1}}{ \zeta^{3}} d \zeta \\
  & = |\bar{z} - \bar{\iota} | \cdot K_{1}\log \bigg(\frac{B}{A} \bigg).
\end{aligned}
\end{equation*}
Observe that the attempt in (\ref{equ20}) is much better than the 
attempt in (\ref{equ19}). 
But it does not serve our purpose. 
\end{remark}
\paragraph{\textit{Proof of Theorem \ref{thmglobharmvf}}:}
 Recall (\ref{secondcov1}). To begin with we note that the second 
 boundedness condition on $q$ can be translated as follows:
 \begin{equation}
\label{equ63}
 \arrowvert f_{zz} \arrowvert \leq \frac{K_{2}}{\Im(z)^{4}}, 
 \quad \forall z \in \hh,
\end{equation}
where $f_{zz}$ denote the second complex 
derivative of $f$, 
$f$ being a holomorphic function on $\hh$. To prove that 
$\xi^{\mathrm{reg}}(z)$ converges we use the Cauchy criterion of 
convergence of improper integrals which has been stated in Remark \ref{firstattempt}.
We notice that 
$$\xi_{c}(\iota)+\frac{\partial \xi_{c}}{\partial z}\bigg|_{z=\iota} \cdot (z- \iota)$$
in (\ref{equ62}) is the \textit{holomorphic part} of the first order Taylor
approximation of $\xi_{c}(z)$ at $z=\iota$. Let's denote it by $T^{\mathrm{hol}}_{1, \iota}(\xi_{c}(z))$. 
Also, $\frac{\partial \xi_{c}}{\partial z}\big|_{z=\iota}$ is nothing complicated but a complex number because
$\xi'_{c}(z)|_{z=\iota}$ as an $\rr$-linear map from $\cc$ to $\cc$ can be written uniquely 
as a sum of a $\cc$-linear map and a $\cc$-conjugate linear map. 
Let's denote the integrand $ \iota \zeta^{2} \overline{f(\bar{z}+2 \iota \zeta)}$ in 
the expression of $\xi_{c}(z)$ by $F(\zeta, z)$.
As both $F(\zeta, z)$ and its partial derivatives are continuous in
$\zeta$ and $z$, we can express $\xi'_{c}(z)|_{z=\iota}$ using 
the Leibniz rule as follows:
\begin{equation}
\label{equ71}
\resizebox{1.06\hsize}{!}{$
\begin{aligned}
 \xi'_{c}(z)|_{z=\iota} & = \Bigg(- \iota \Im(z)^{2} \overline{f(\bar{z}+2 \iota \Im(z))} \cdot \Im'(z) 
                              + \int_{\Im(z)}^{c} \iota \zeta^{2} \bigg(\frac{\partial}{\partial z}  \overline{f(\bar{z}+2 \iota \zeta) } dz + \frac{\partial}{\partial \bar{z}}\overline{f(\bar{z}+2 \iota \zeta) } d\bar{z}\bigg)d \zeta\Bigg) \bigg|_{z=\iota} \\
                        & = \Bigg(- \iota \Im(z)^{2} \overline{f(z)} \cdot \Im'(z) +
                                    \int_{\Im(z)}^{c} \iota \zeta^{2} \frac{\partial}{\partial z}  \overline{f(\bar{z}+2 \iota \zeta)} d \zeta\Bigg) \bigg|_{z=\iota} \\
                        & = \underbrace{- \iota \overline{f(\iota)} \cdot \Im'(z)|_{z=\iota}}_{K} +
                                   \Bigg( \int_{\Im(z)}^{c} \iota \zeta^{2} \frac{\partial}{\partial z}  \overline{f(\bar{z}+2 \iota \zeta)} d \zeta\Bigg) \bigg|_{z=\iota}           
\end{aligned}
$}
\end{equation}
where the second equality in (\ref{equ71}) follows from the fact that $f$ is 
a holomorphic function, hence we get
\begin{equation*}
 \frac{\partial}{\partial \bar{z}}  \overline{f(\bar{z}+2 \iota \zeta)} =\overline{ \Bigg( \frac{\partial}{\partial z} f(\bar{z}+2 \iota \zeta) \Bigg)} = 0.
\end{equation*}
Note that we have omitted $dz$ in $\frac{\partial}{\partial z}  \overline{f(\bar{z}+2 \iota \zeta) } dz$ 
because $dz$ as a linear map can be viewed as the $2 \times 2$ identity matrix.
Since the summand $\iota \overline{f(\iota)} \cdot \Im'(z)|_{z=\iota}$ in (\ref{equ71}) 
does not depend on $c$, therefore it does not hurt to drop it in the expression of 
$T^{\mathrm{hol}}_{1, \iota}(\xi_{c}(z))$ for convergence investigation. We will denote 
the corrected term by $\Psi_{c}(z)$. 
Using (\ref{equ71}), $\Psi_{c}(z)$ can be written as:
\begin{equation}
\label{equ72}
 \Psi_{c}(z) = \Bigg( \int_{1}^{c} \iota \zeta^{2} \bigg( \overline{f(\bar{\iota}+2 \iota \zeta) } + \bigg( \frac{\partial}{\partial z}  \overline{f(\bar{z}
                        +2 \iota \zeta) } \bigg) \bigg|_{z=\iota} \cdot (z-\iota) \bigg)d \zeta \Bigg) \eta(z).
\end{equation}
Then 
\begin{equation}
\label{newxi}
 \xi^{\mathrm{reg}}(z) = \lim_{c \rightarrow \infty}\big(\xi_{c}(z)-\Psi_{c}(z) - K \big).
\end{equation}
Let's first spell out the expression $\xi_{c}(z) - \Psi_{c}(z)$. 
\smallskip
\paragraph{\textbf{Case I:} $2c \geq 1 \geq \Im(z)$}
\begin{equation}
\label{anotheresti}
\begin{aligned}
\xi_{c}(z) - \Psi_{c}(z)  & = \Bigg( \int_{\Im(z)}^{c} \iota \zeta^{2} 
   \overline{f(\bar{z}+2 \iota \zeta)}d \zeta   \\
 & \qquad  -\int_{1}^{c} \iota \zeta^{2} \bigg( \overline{f(\bar{\iota}+2 \iota \zeta) } 
   + \bigg( \frac{\partial}{\partial z}  \overline{f(\bar{z}
+2 \iota \zeta) } \bigg) \bigg|_{z=\iota} \cdot (z-\iota) \bigg)d \zeta \Bigg) 
\eta(z)  \\
       & = \Bigg( \int_{\Im(z)}^{1} \iota \zeta^{2} 
       \overline{f(\bar{z}+2 \iota \zeta)}d \zeta + \int_{1}^{c} 
       \iota \zeta^{2}\overline{f(\bar{z}+2 \iota \zeta)} d \zeta  \\
    & \qquad  - \int_{1}^{c} \iota \zeta^{2} \bigg( \overline{f(\bar{\iota}+
    2 \iota \zeta) } 
      + \bigg( \frac{\partial}{\partial z}  \overline{f(\bar{z}
    +2 \iota \zeta) } \bigg) \bigg|_{z=\iota} \cdot (z-\iota) \bigg)d \zeta  
    \Bigg) \eta(z)  \\
  & = \Bigg( \underbrace{\int_{1}^{c} \iota \zeta^{2} \bigg( 
  \overline{f(\bar{z}+2 \iota \zeta) - f(\bar{\iota}+2 \iota \zeta)} - 
  \bigg( \frac{\partial}{\partial z}  \overline{f(\bar{z}  +2 \iota \zeta) } 
  \bigg) \bigg|_{z=\iota} \cdot (z-\iota) \bigg)d \zeta}_{I_{c}} \\
                      & \qquad      - \int_{1}^{\Im(z)} \iota 
                                   \zeta^{2}\overline{f(\bar{z}+2 \iota \zeta)} 
                                   d \zeta \Bigg) \eta(z). 
\end{aligned} 
  \end{equation}
  \paragraph{\textbf{Case II:} $2c \geq \Im(z) \geq  1$}
  \begin{equation}
  \label{anotherequ21}
   \begin{aligned}
   \xi_{c}(z)-\Psi_{c}(z) & =  \Bigg( \int_{\Im(z)}^{c} \iota \zeta^{2} 
   \overline{f(\bar{z}+2 \iota \zeta)}d \zeta \\
       & \qquad -\int_{1}^{\Im(z)} \iota \zeta^{2} 
       \bigg( \overline{f(\bar{\iota}+2 \iota \zeta) } 
       + \bigg( \frac{\partial}{\partial z}  \overline{f(\bar{z}
        +2 \iota \zeta) } \bigg) \bigg|_{z=\iota} \cdot (z-\iota) \bigg)d \zeta \\
              & \qquad -\int_{\Im(z)}^{c} \iota \zeta^{2} 
              \bigg( \overline{f(\bar{\iota}+2 \iota \zeta) } + 
              \bigg( \frac{\partial}{\partial z}  \overline{f(\bar{z}
                        +2 \iota \zeta) } \bigg) 
                        \bigg|_{z=\iota} \cdot (z-\iota) \bigg)d \zeta  \Bigg) 
                        \eta(z) \\
  & = \Bigg( \underbrace{\int_{\Im(z)}^{c} \iota \zeta^{2} 
  \bigg( \overline{f(\bar{z}+2 \iota \zeta) - f(\bar{\iota}+2 \iota \zeta)} - 
  \bigg( \frac{\partial}{\partial z}  \overline{f(\bar{z}
         +2 \iota \zeta) } \bigg) \bigg|_{z=\iota} \cdot (z-\iota) 
         \bigg)d \zeta}_{II_{c}}  \\ 
                        & \qquad - \int_{1}^{\Im(z)} \iota \zeta^{2} \bigg( \overline{f(\bar{\iota}+2 \iota \zeta) } + \bigg( \frac{\partial}{\partial z}  \overline{f(\bar{z}
                        +2 \iota \zeta) } \bigg) \bigg|_{z=\iota} 
                        \cdot (z-\iota) \bigg)d \zeta \Bigg) \eta(z). 
  \end{aligned}
  \end{equation}
Since the integrals 
$$\int_{1}^{\Im(z)} \iota \zeta^{2}\overline{f(\bar{z}+2 \iota \zeta)} d \zeta$$
and 
$$\int_{1}^{\Im(z)} \iota \zeta^{2} \bigg( \overline{f(\bar{\iota}+2 \iota \zeta) } + \bigg( \frac{\partial}{\partial z}  \overline{f(\bar{z}
                        +2 \iota \zeta) } \bigg) \bigg|_{z=\iota} \cdot (z-\iota) \bigg)d \zeta$$
                        in R.H.S of (\ref{anotheresti}) and (\ref{anotherequ21}) 
                        are independent of $c$, we work with $I_{c}$ and $II_{c}$ in (\ref{anotheresti}) and (\ref{anotherequ21}) to prove the convergence of $\xi^{\mathrm{reg}}$.
                        Now if $A, B \geq c$, we have
 \begin{equation*}
  I_{B}-I_{A} = \int_{A}^{B} \iota \zeta^{2} \bigg( \overline{f(\bar{z}+2 \iota \zeta) 
  - f(\bar{\iota}+2 \iota \zeta)} - \bigg( \frac{\partial}{\partial z}  \overline{f(\bar{z}
                                   +2 \iota \zeta) } \bigg) \bigg|_{z=\iota} 
                                   \cdot (z-\iota) \bigg)d \zeta, 
 \end{equation*}
and 
\begin{equation*}
 II_{B}-II_{A} = \int_{A}^{B} \iota \zeta^{2} \bigg( \overline{f(\bar{z}+2 \iota \zeta) 
 - f(\bar{\iota}+2 \iota \zeta)} - \bigg( \frac{\partial}{\partial z}  \overline{f(\bar{z}
                        +2 \iota \zeta) } \bigg) \bigg|_{z=\iota} \cdot (z-\iota) 
                        \bigg)d \zeta.
\end{equation*}
Using the Remainder Estimation Theorem for $f$, we have
\begin{equation}
 \label{inequalities3}
  \arrowvert I_{B}-I_{A} \arrowvert = \arrowvert II_{B}-II_{A} \arrowvert  \leq  
                                     \int_{A}^{B} \zeta^{2} \cdot 
                                    \mathrm{max}_{w}\arrowvert f^{(2)}(w) 
                                    \arrowvert 
                                    \cdot |(\bar{z}+2 \iota \zeta)-
                                    (\bar{\iota}+2 \iota \zeta)|^{2}  
                                    d \zeta, 
\end{equation} 
where $w$ is varying on the line segment connecting $\bar{z}+2 \iota \zeta$ 
and $\bar{\iota}+2 \iota \zeta$. We assume $A, B > \Im(z)$. Using (\ref{equ63}), 
we rewrite (\ref{inequalities3}) as follows:
\begin{equation}
\label{inequalities4}
 \begin{aligned}
   \arrowvert I_{B}-I_{A} \arrowvert = \arrowvert II_{B}-II_{A} \arrowvert  
   & \leq \int_{A}^{B} \zeta^{2} \cdot 
                                    \mathrm{max}_{w}\frac{K_{2}}{(\Im(w))^{4}} 
                                    \cdot |\bar{z}-\bar{\iota}|^{2} 
                                    d \zeta \\
                                    & \leq |\bar{z}-\bar{\iota}|^{2} 
                                    \int_{A}^{B} \zeta^{2} \cdot \frac{K_{2}}{(2 \zeta- \Im(z))^{4}} d \zeta 
 \end{aligned}
\end{equation}
 Also, the denomiator $(2 \zeta- \Im(z))^{4}$ is atleast as big as 
$\zeta^{4}$. As a result (\ref{inequalities4}) has the following form: 
\begin{equation}
\label{inequalities5}
 \begin{aligned}
  \arrowvert I_{B}-I_{A} \arrowvert = \arrowvert II_{B}-II_{A} \arrowvert  
   & \leq |\bar{z}-\bar{\iota}|^{2}
                                    \int_{A}^{B} \zeta^{2} \cdot \frac{K_{2}} {\zeta^{4}} d \zeta \\
                                    & = |\bar{z}-\bar{\iota}|^{2} \cdot 
                                    \frac{K_{2}}{4} \bigg(-\frac{1}{B}+\frac{1}{A}\bigg).
 \end{aligned}
 \end{equation}
 These estimates show that $\xi^{\mathrm{reg}}$ is a well defined vector field. But they also show that 
$\xi^{\mathrm{reg}}$ is locally a uniform limit of harmonic vector fields which determine 
the same holomorphic quadratic differential. Therefore, $\xi^{\mathrm{reg}}$ is a harmonic vector field 
by Corollary \ref{coroharmonic}. \hfill \qedsymbol
\subsection{Extending harmonic vector fields on $\hh$ to the boundary circle $\mathbb{S}^{1}$}
\label{extendharmonic}
We refer to
the extended real axis $\overline{\rr} := \rr \cup \{\infty\}$ as the 
boundary at infinity of $\hh$.
We are using the unit disc model so that we have a well defined notion of the 
tangent space at the point 
$\{\infty\} \in \partial \hh$ as there is a natural 1-1 correspondence between 
$\partial \dd$ and $\partial \hh$. The starting point is to compare the length 
of a vector $v \in T_{z} \hh$ 
for some $z \in \hh$ (measured
in the Euclidean metric) with the
length of the pushforward of $v$ (measured in the Euclidean metric) by a 
conformal map between $\hh$ and $\dd$. 
Consider the Cayley transformation
 \begin{equation}
 \label{cayley}
  C(z)= \frac{z-\iota}{z+\iota}
 \end{equation}
 mapping the upper half plane model of $\hh$ to the unit disc model $\dd$ of $\hh$. 
We have
 \begin{equation}
  \label{endeq}
  |dC_{z} (v)| = \frac{| v |}{|z|^{2}}, \quad \forall v \in T_{z}\hh.
 \end{equation}
\begin{theorem}
\label{boundary}
The harmonic vector field $\xi^{\mathrm{reg}}$ in Theorem \ref{thmglobharmvf}, 
transformed from $\hh$ to the open unit disc $\dd \subset\cc$ by 
the Cayley transform $C$ given by (\ref{cayley}) 
extends to a continuous vector field, say $\chi$, 
on $\overline{\dd}$ defined as follows:
\begin{equation}
 \label{extendedvf}
 \chi(C(z)) = \begin{cases}
            C_{\ast}(\xi^{\mathrm{reg}}(z)) & \quad z \in \hh \\
            C_{\ast}(\xi^{\mathrm{reg}}(z)) & \quad z \in \partial \hh \setminus \{\infty\} \\
            0                     & \quad z = \{ \infty \}
          \end{cases}
          \end{equation}
          where $C_{\ast}(\xi^{\mathrm{reg}}(z))$ is 
the pushforward of $\xi^{\mathrm{reg}}(z)$ by the Cayley transform $C$. 
\end{theorem} 
Before we prove Theorem \ref{boundary}, we discuss the one and only 
disadvantage of Wolpert's formula (\ref{wolpertformula}) in the following remark:

\begin{remark}
 \normalfont
 Recall Wolpert's global solution $\xi$ (see (\ref{wolpertformula})) for the potential equation 
 (\ref{wolpertequation}). Given
 that $q=fdz^{2}$ is bounded in the hyperbolic metric $\textbf{g}_{\hh}$, i.e., 
$|f(z)| \leq \frac{D}{\Im(z)^{2}}$ where $D$ is a positive constant, 
$\xi$ extends to the real line $\rr$. This can be seen as follows: 
$f$ is not defined for $z$ such that $\Im(z)=0$. So the integral in (\ref{wolpertformula}) 
is an improper integral, so for $z$ such that $\Im(z)=0$, we define 
\begin{equation*}
 \label{estimatewol}
 \xi(z)= \lim_{\epsilon \rightarrow 0} \Bigg( 
\Bigg( \overline{\int_{w}^{z+\iota \epsilon} 
(\overline{z+\iota \epsilon}-\zeta)^{2} f(\zeta) d\zeta} 
\Bigg) \eta(z) \Bigg).
\end{equation*}
The above limit exists, as can be seen by taking $w$ to be $\iota$ and 
using 
$|f(z)| \leq \frac{D}{\Im(z)^{2}}$. We have no reason to believe that $\xi$
extends to the point $\{\infty \}$ in the boundary $\rr \cup \{\infty \}$. Here 
is an argument: for the sake of convenience, we choose the line segment 
from 
$w=\iota$ to $z=c\iota$
as the 
path of integration in the expression of $\xi$, where $c>1$ is a positive real number. Then,
$$\xi(c \iota) = \overline{\int_{\iota}^{c\iota} (\overline{c\iota}-\zeta)^{2} 
f(\zeta) d \zeta}.$$
Therefore, 
\begin{equation*}
 \begin{split}
  |\xi(c \iota)| & \leq D\int_{\iota}^{c\iota} 
\frac{|\overline{c\iota}-\zeta|^{2}}{\Im(\zeta)^{2}} d\zeta \\
& = D\int_{\iota}^{c\iota} 
\frac{|c\iota-\zeta|^{2}}{\Im(\zeta)^{2}} d\zeta \\
& \leq D \cdot |c\iota - \iota| \cdot \mathrm{max}_{\zeta} \frac{|c\iota-\zeta|^{2}}{\Im(\zeta)^{2}}, 
 \end{split}
\end{equation*}
where $\zeta$ is varying on the line segment from $\iota$ to $c\iota$. From 
the above estimate, it is clear that $\xi$ is $O(|z|^{3})$ at 
the 
point $\{ \infty \}$ in the boundary $\rr \cup \{\infty \}$.
\end{remark}
\bigskip
\paragraph{\textit{Proof of Theorem \ref{boundary}:}}
 Recall Remark \ref{wecanextend}. For $z$ such that
$\Im(z)=0$, 
the definition of $\xi^{\mathrm{reg}}$ makes perfectly good sense because the convergence 
of the improper integral in the expression of $\xi^{\mathrm{reg}}$ for $z$ such that 
$\Im(z)=0$ follows from the conditions given in (\ref{equ26}), (\ref{equ27}), 
and (\ref{equ63}). 
Now, we claim that for a sequence $\{z_{n}\}$ of points in $\hh$ such 
that $|z_{n}| \rightarrow \infty$, 
where $|\cdot|$ denotes the absolute value 
\begin{equation}
 \label{approachinfinity}
 \lim_{|z_{n}| \rightarrow \infty} | C_{\ast}(\xi^{\mathrm{reg}}(z_{n})) | =0,
\end{equation}
where 
$|C_{\ast}(\xi^{\mathrm{reg}}(z))|$ denotes the 
length of the pushforward of $\xi^{\mathrm{reg}}(z)$ measured in the Euclidean 
metric. Using (\ref{endeq}), we rewrite (\ref{approachinfinity}) as follows
\begin{equation}
 \label{approachinfinity1}
 \lim_{|z_{n}| \rightarrow \infty} \frac{| \xi^{\mathrm{reg}}(z_{n})|}
 {|z_{n}|^{2}} =0.
\end{equation}
The main idea is to split the integral 
$\int_{0}^{c} \iota \zeta^{2} \overline{f(z_{n}+2 \iota \zeta)} d \zeta$
at height $h$ such that $h = |z_{n}|$ and 
estimate the resulting integrals in different ways. 
Using (\ref{equ71}), (\ref{equ72}), and (\ref{newxi}), our expression for $\xi^{\mathrm{reg}}(z_{n})$ takes the following form:
\begin{equation*}
\label{equ120}
 \xi^{\mathrm{reg}}(z_{n})  =  \underbrace{\xi_{h}(z_{n}) - \bigg( \xi_{h}(\iota) +
\frac{\partial \xi_{h}}{\partial z}\bigg|_{z=\iota} \cdot 
(z_{n}-\iota)\bigg) }_{\xi_{1}^{\mathrm{reg}}(z_{n})} 
                                      +  
                                      \underbrace{\lim_{c \to \infty} 
                                      \big( \xi_{h, c}(z_{n}) - \Psi_{h, c}(z_{n}) -K
\big)}_{\xi_{2}^{\mathrm{reg}}(z_{n})},
\end{equation*}
where
\begin{equation*}
\begin{split}
 \xi_{h}(z_{n}) = \bigg( \int_{0}^{h} \iota \zeta^{2} 
 \overline{f(z_{n}+2 \iota \zeta)} d \zeta \bigg) \eta(z_{n}), & \quad
 \xi_{h, c}(z_{n}) = \bigg( \int_{h}^{c} \iota \zeta^{2} 
 \overline{f(z_{n}+2 \iota \zeta)} d \zeta \bigg) \eta(z_{n}), \\  
 \xi_{h}(\iota) = \bigg( \int_{1}^{h} \iota \zeta^{2} 
 \overline{f(\bar{\iota}+2 \iota \zeta)} d \zeta \bigg) \eta(\iota), & \quad
 \frac{\partial \xi_{h}}{\partial z}\bigg|_{z=\iota} = \bigg( \int_{1}^{h} \iota \zeta^{2} 
 \bigg( \frac{\partial}{\partial z}\overline{f(z_{n}+2 \iota \zeta)} \bigg)\bigg|_{z=\iota}
 d \zeta \bigg) \eta(\iota),
 \end{split}
\end{equation*}
\begin{equation*}
 \Psi_{h, c}(z_{n}) = \Bigg( \int_{h}^{c} \iota \zeta^{2} \bigg( \overline{f(\bar{\iota}+2 \iota \zeta) } + \bigg( \frac{\partial}{\partial z}  \overline{f(z_{n}
                        +2 \iota \zeta) } \bigg) \bigg|_{z=\iota} \cdot (z_{n}-\iota) \bigg)d \zeta \Bigg) \eta(z_{n}).
\end{equation*}
Note that we have treated $h=|z_{n}|$ as a constant independent of $z_{n}$.  
Using (\ref{firstcov}), (\ref{secondcov1}), and (\ref{equ26}), 
each individual term - $\xi_{h}(z_{n})$, $\xi_{h}(\iota)$ and 
$\frac{\partial \xi_{h}}{\partial z}|_{z=\iota} \cdot (z_{n}-\iota)$ - in
the expression of $\xi_{1}^{\mathrm{reg}}(z_{n})$ satisfies the 
following inequalities when estimated in the Poincare metric $\textbf{g}_{\hh}$:
\begin{equation*}
\label{equ129}
\begin{aligned}
 \arrowvert \xi_{h}(z_{n}) \arrowvert & \leq \frac{D}{4} |z_{n}|, \\
 \arrowvert \xi_{h}(\iota) \arrowvert & \leq \frac{D}{4} |z_{n}|, \\
 \bigg \arrowvert \frac{\partial \xi_{h}}{\partial z}\big|_{z=\iota} \cdot (z_{n}-\iota)  \bigg 
 \arrowvert &  \leq  \frac{K_{1}}{8} |z_{n}|.
\end{aligned}
\end{equation*}
At this point (\ref{endeq}) comes in handy and show us 
immediately that $C_{\ast}(\xi_{1}^{\mathrm{reg}}(z_{n})) \rightarrow 0$ 
as 
$|z_{n}| \rightarrow \infty$. 
From 
the estimate given in (\ref{inequalities5}) in the proof of Theorem \ref{thmglobharmvf} 
we have 
$C_{\ast}(\xi_{2}^{\mathrm{reg}}(z_{n})) \rightarrow 0$ as $|z_{n}| \rightarrow \infty$. 
\hfill \qedsymbol

\section{Going from the analytic description to the cohomological description}
\label{analtocohom}
\subsection{Vector fields on $\dd$ and $\ss^{1}$}
\label{vectorfieldsonthecircle} 
We will denote the Hilbert space of measurable functions 
$f$ on $\ss^{1}$ such that 
$$\int_{\ss^{1}} | f(x)|^{2} dx < + \infty $$
modulo the equivalence relation of almost-everywhere equality by $L^{2}(\ss^{1})$.
We are not going to prove the 
completeness of $L^{2}(\ss^{1})$. The main idea to prove completeness of 
$L^{2}(\ss^{1})$ is that 
a Cauchy sequence of 
$L^{2}$-functions has 
a subsequence that converges pointwise off a set of measure $0$. 
There is a different definition of $L^{2}(\ss^{1})$, namely
the completion of 
$C^{0}(\ss^{1})$, 
the space of continuous $\cc$-valued functions on $\ss^{1}$, with respect to the norm 
\begin{equation}
\label{l2norm}
 \lVert f \rVert:= \frac{1}{\sqrt{2 \pi}} \bigg(\int_{\ss^{1}} |f(z)|^{2} dz \bigg)^{1/2}
\end{equation}

The Fourier basis elements 
are the exponential functions $\psi_{k}(z):=z^{k}$ for $z \in \ss^{1}$.  
The exponential functions $\{\psi_{k}|k \in \zz\}$ form an 
orthonormal set in $L^{2}(\ss^{1})$. 
But it's not clear immediately that they form an orthonormal Hilbert basis (see \cite{beals}). 
From orthonormality, 
the \textit{Fourier coefficients} $a_{k} \in \cc$ of $f$ are the 
inner products 
\begin{equation*}
 a_{k}= \langle f, \psi_{k} \rangle = \frac{1}{2 \pi}\int_{\ss^{1}} 
f(z) \overline{\psi_{k}(z)} dz.
\end{equation*}

The Fourier expansion of $f \in L^{2}(\ss^{1})$ is 
$$f(z)= \sum_{k \in \zz} a_{k} \psi_{k}(z)$$
where the equality means convergence of the partial sums to $f$ in the $L^{2}$-norm, 
or 
$$\lim_{N \rightarrow \infty} \frac{1}{\sqrt{2 \pi}} \int_{\ss^{1}} \bigg|\sum_{k=-N}^{N} a_{k} \psi_{k}(z)-f(z)\bigg|^{2} dz=0.$$

The convenient algebraic property of $\psi_{k}$ is that the basis is 
multiplicative. And multiplication of functions corresponds to the 
\textit{convolution of Fourier series;} this is actually obvious in our context since 
\begin{equation}
 \label{convolution}
 \psi_{k} \cdot \psi_{l} = \psi_{k+l}.
\end{equation} 

From now on we will denote $L^{2}(\ss^{1})$ by 
$\mathcal{H}$.
There is 
an orthogonal sum splitting $\mathcal{H}= \mathcal{H}^{1} \oplus \mathcal{H}^{2}$ 
where $ \mathcal{H}^{1}$ is the closure of the span of $\{\psi_{k} | k <0 \}$ 
and consequently, 
$ \mathcal{H}^{2}$ is the closure of the span of $\{\psi_{k} | k \geq 0 \}$. 
An element of 
$ \mathcal{H}^{2}$, say 
$$f:= \sum_{k \geq 0} a_{k} f_{k}$$
has a canonical extension to a function (in the $L^{2}$-sense) defined on the unit 
disk $\dd$ in $\cc$ by the 
formula 
$$ z \longmapsto \sum_{k \geq 0} a_{k} z^{k}.$$ 

This is in fact a convergent power series in the open unit disk $\dd$, so defines 
a holomorphic 
function on the open unit disk $\dd$ in $\cc$. So we should see $ \mathcal{H}^{2}$ 
as the linear subspace 
of $\mathcal{H}$ consisting of those $L^{2}$-functions on $\ss^{1}$ which extend 
holomorphically to the 
open unit disk $\dd$ in $\cc$. Equivalently, think of $ \mathcal{H}^{2}$ as the 
complex vector space of 
$L^{2}$-vector fields on $\ss^{1}$ 
which extend holomorphically to the 
open unit disk, i.e.
$$\mathcal{H}^{2}= \{X: \ss^{1} \longrightarrow \rr^{2} | X \hspace{1pt} \mathrm{is} 
\hspace{1pt} L^{2}, 
X(z) \in T_{z} 
\rr^{2} \cong \rr^{2} \cong \cc, \forall z \in \ss^{1}\},$$
where the norm on $X$ is taken in the sense of (\ref{l2norm}). 

\begin{remark}
\label{l2cont}
\normalfont
A smooth or continuous vector field $X$ on the open unit disk $\dd$
has an \textit{$L^{2}$-extension} to the closed disk $\overline{\dd}$ if the 
following holds:
for every $\epsilon > 0$, 
we get a continuous vector field $X_{\epsilon}$ on $\ss^{1}_{1-\epsilon}$, 
a circle of radius $1-\epsilon$ (which can be identified
canonically with $\ss^{1}$ by streching), by restricting $X$ to 
$\ss^{1}_{1-\epsilon}$. 
Now, letting $\epsilon \rightarrow 0$, we get 
a sequence $\{X_{\epsilon}\}$ in the Hilbert space of $L^{2}$-vector fields on the boundary circle 
$\ss^{1}$. 
And if $\{X_{\epsilon}\}$ converges to 
a $L^{2}$-vector field on the boundary circle $\ss^{1}$, then $X$ has an $L^{2}$-
extension to the closed disk $\overline{\dd}$. 
 \end{remark}
\begin{defn}
\label{defnoftan}
 \normalfont
 A vector field on $\ss^{1}$ with values in $\rr^{2}$ or $\cc$ is called \textit{tangential} if 
 it makes $\ss^{1}$ flows into itself. 
\end{defn}
We denote the space of 
 tangential vector fields on $\ss^{1}$ by $\mathfrak{X}_{\textrm{tangential}}(\ss^{1})$. 
 It is a real vector space. 
 To get more insight, consider the following 
 example: 
 \begin{example}
 \label{examoftan}
  \normalfont
  Consider the following complex-valued vector field on $\ss^{1}$: 
  $$X(x, y)= -y \frac{\partial}{\partial x}+x \frac{\partial}{\partial y}.$$
  In complex coordinates, we express $X$ as $X(z) = \iota z$. 
  It is clear that $X$ is a tangential vector field on $\ss^{1}$ since 
  $$\sigma(t, (x,y)) = 
  (x \cos t- y \sin t, x \sin t+y \cos t)$$ is a flow generated by $X$ and the flow through $(x, y)$ is 
  a circle whose centre is at origin. Clearly, $\sigma(t, (x,y))=(x, y)$ 
  if $t=2n \pi, n \in \zz$. 
 See L.H.S of Figure \ref{pic26}.
  \begin{figure}[h]
\centering
 \includegraphics[height=5cm]{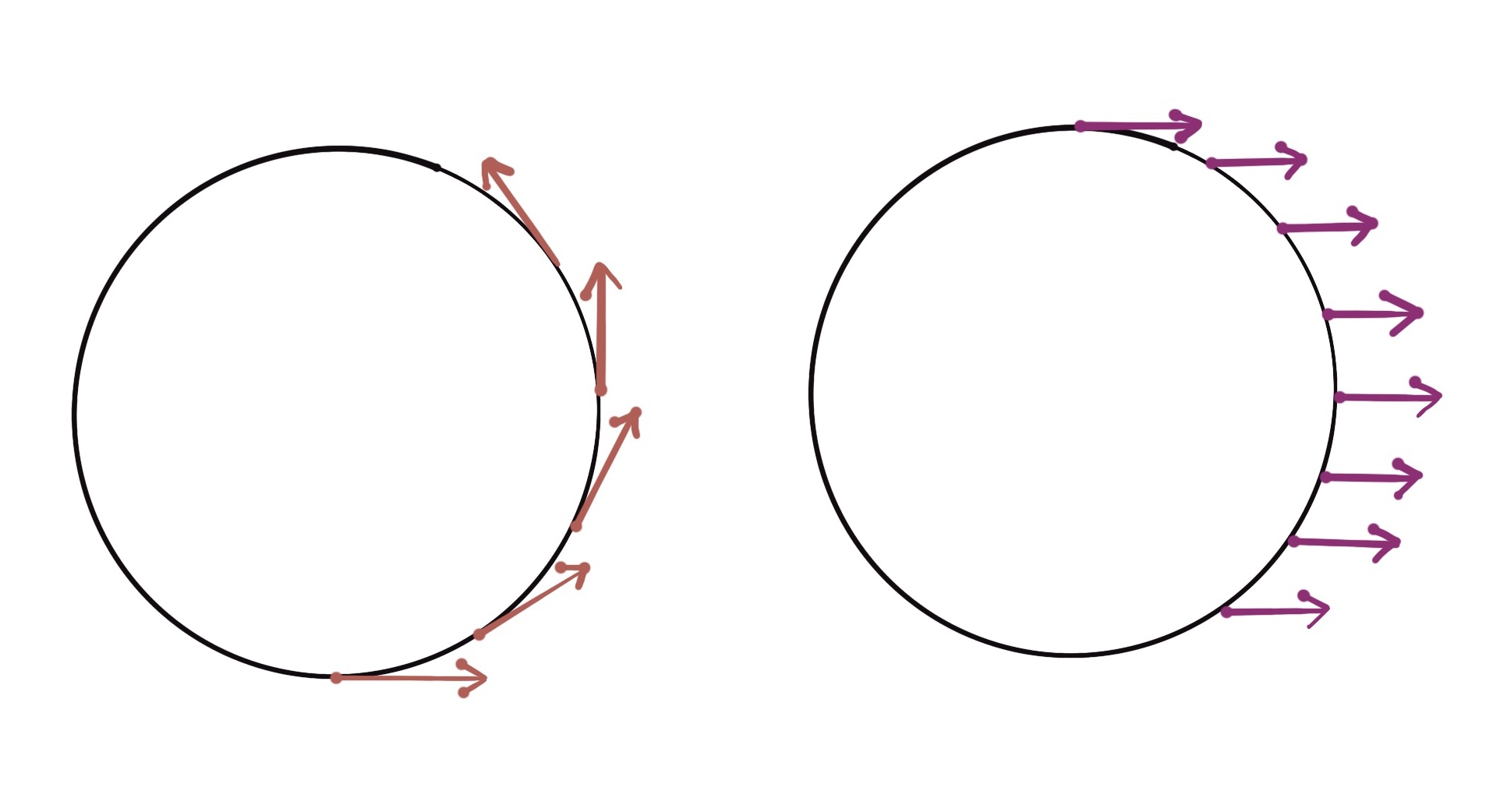}
 \caption{An example of a tangential vector field on $\ss^{1}$}
 \label{pic26}
\end{figure}
 \end{example}
Note that the above example is only one solution of tangential vector fields on $\ss^{1}$. 
But we 
get all other solutions by multiplying $X$ in Example \ref{examoftan} with any real valued function on $\ss^{1}$. Note 
that vector fields can be multiplied with functions. For simplicity, 
we think of 
multiplication of $L^{2}$-vector fields on $\ss^{1}$ with real valued functions on $\ss^{1}$
as multiplication of functions with 
functions. 
\smallskip

Recall that we have expressed an $L^{2}$-function $f$ on $\ss^{1}$ with values 
in $\cc$
as $\sum_{k \in \zz}a_{k} \psi_{k}$. It's a routine exercise in 
Fourier analysis to show that $f$ is real 
valued iff $a_{k}= \overline{a_{-k}}$ for all $k$. Therefore the 
corresponding (real) Fourier 
expansion of $f$ is 
$$f(x)= \frac{1}{2}a'_{0} + \sum_{k=1}^{\infty} a'_{k} \cos(kx) + b'_{k} \sin(kx),$$
where $a'_{k}=a_{k}+a_{-k}$ and $b'_{k}= \iota(a_{k}-a_{-k})$. 
So, the real-valued functions 
$$\{1, \cos(kx), \sin(kx) | k=1, 2, 3, \ldots  \}$$
also form an orthogonal basis of the space $\mathcal{H}$, since 
$$\cos(kx)= \frac{\exp(\iota k x) +  \exp(-\iota k x)}{2} 
= \frac{z^{k}+z^{-k}}{2},$$
$$\sin(kx)= \frac{\exp(\iota k x) - \exp(-\iota k x)}{2 \iota} 
= \frac{z^{k}-z^{-k}}{2 \iota}.$$

Using (\ref{convolution}), i.e., the fact that the Fourier transform of the product 
of functions is the convolution of 
the Fourier transforms, we have the following real Hilbert basis of 
$\mathfrak{X}_{\textrm{tangential}}(\ss^{1})$:
\begin{equation}
\label{realbasis}
 \bigg \{\iota z, \frac{\iota z^{1+k}+\iota z^{1-k}}{2}, \frac{z^{1+k}-z^{1-k}}{2} 
\bigg| k=1, 2, 3, \ldots \bigg \}.
\end{equation}

Also, Killing vector fields on $\dd$ are the infinitesimal 
generators of isometries of 
 $\dd$, hence Killing vector fields on $\dd$ are tangential 
 vector fields on $\ss^{1}$. 
We will denote the three dimensional real vector space whose elements are Killing 
vector fields  
 on $\ss^{1}$ 
by $\mathfrak{X}_{\textrm{Killing}}(\ss^{1})$. 
 \begin{theorem}
\label{infinitedimproblem}
We have
 \begin{enumerate}
  \item $\mathfrak{X}_{\textrm{tangential}}(\ss^{1}) \cap  \mathcal{H}^{2} = 
 \mathfrak{X}_{\textrm{Killing}}(\ss^{1}).$
 \item $\mathfrak{X}_{\textrm{tangential}}(\ss^{1}) +  \mathcal{H}^{2}$ is the 
 vector space of 
 all $L^{2}$-vector fields on $\ss^{1}$.
 \end{enumerate}
\end{theorem}
\begin{proof}
 \normalfont
 (1). As any complex vector space has an underlying real vector space so the 
real Hilbert basis of the 
space $\mathcal{H}^{2}$ is given as 
$$\{z^{k}, \iota z^{k} | k \geq 0 \}.$$
The basis for $\mathfrak{X}_{\textrm{tangential}}(\ss^{1})$ is given by 
(\ref{realbasis}). 
Assume $X \in \mathfrak{X}_{\textrm{tangential}}(\ss^{1}) \cap  \mathcal{H}^{2}$. 
Then $X = \sum_{k \geq 0}a_{k} z^{k} + b_{k} i z^{k}$ and 
$X = a'_{0}\iota z + \sum_{k \geq 1} a'_{k} \frac{\iota z^{1+k}+\iota z^{1-k}}{2} + 
b'_{k} \frac{z^{1+k}-z^{1-k}}{2}$. Since 
$$\sum_{k \geq 0}a_{k} z^{k} + b_{k} i z^{k} =  a'_{0}\iota z + \sum_{k \geq 1} a'_{k} \frac{\iota z^{1+k}+\iota z^{1-k}}{2} + 
b'_{k} \frac{z^{1+k}-z^{1-k}}{2},$$
comparing the coefficients of $z^{k}$ and $\iota z^{k}$ in each expression, 
we obtain $a'_{0}=b_{1}$, 
$b_{2}=b_{0}=\frac{a'_{1}}{2}, a_{2}= \frac{b'_{1}}{2} = -a_{0}$, and all other 
coefficients are zero. Therefore, $X$ is a linear combination with real coefficients 
of 
$\iota z$, $\frac{z^{2}-1}{2}$, and $\frac{\iota z^{2}+\iota}{2}$ 
 Note that 
 $\iota z$, $\frac{z^{2}-1}{2}$, and $\frac{\iota z^{2}+\iota}{2}$ are linearly 
 independent. 
 Hence the vector space 
 $\mathfrak{X}_{\textrm{tangential}}(\ss^{1}) \cap  \mathcal{H}^{2}$ is a 
 $3$-dimensional space which is nothing but 
 $\mathfrak{X}_{\textrm{Killing}}(\ss^{1}).$ 
\bigskip
\newline
(2) A real Hilbert basis of the space of $L^{2}$-vector fields on $\ss^{1}$ is 
given by $\{z^{k}, \iota z^{k} | k \in \zz \}$. Then it is very easy to see that 
\begin{equation*}
\begin{aligned}
 X(z) - \Bigg( \Bigg( \sum_{k \in \{2, 3, \ldots\}} 
 b_{1-k}\big(\iota z^{1+k}+\iota z^{1-k}\big)\Bigg) - 
 \sum_{k \in \{2, 3, \ldots\}} a_{1-k}\big(z^{1+k}-z^{1-k}\big)
 \Bigg)\\
 \qquad = 
 a_{0}+b_{0}\iota + a_{1}z+b_{1} \iota z + a_{2} z^{2}+b_{2} \iota z^{2} + 
 (a_{3}+a_{-1})z^{3}+(b_{3}-b_{-1})\iota z^{3}+ \cdots,
\end{aligned}
\end{equation*}
where $X(z)= \sum_{k \in \zz} a_{k}z^{k} + b_{k} \iota z^{k}$, $z \in \ss^{1}$. 
Therefore, 
$X = X_{1} + X_{2}$, where $X_{1} \in \mathfrak{X}_{\textrm{tangential}}(\ss^{1})$ 
and $X_{2} \in \mathcal{H}^{2}$. 
\hfill \qedsymbol
\end{proof}
\bigskip
Before we state conclusions of this section we introduce 
some notions and conventions: 
\begin{enumerate}
 \item Let $\mathfrak{M}$ be a $\Gamma$-module, where 
 $\Gamma$ is a subgroup of 
 $\mathrm{PSU}(1, 1)$. A map 
 $c: \Gamma \longrightarrow \mathfrak{M}$ is called a cocycle
 if 
 \begin{equation*}
  \label{group1}
  c_{\gamma_{1} \circ \gamma_{2}}=  \gamma_{2}^{\ast}c_{\gamma_{1}} + c_{\gamma_{2}}, 
  \gamma_{1}, \gamma_{2} \in \Gamma,
 \end{equation*}
 $c_{\gamma}$ stands for $c(\gamma)$, $\ast$ denotes the action of $\Gamma$ on 
 $\mathfrak{M}$. If $m \in \mathfrak{M}$, its \textit{coboundary} 
 $\delta m$ is the cocycle
 \begin{equation}
  \label{group2}
  \gamma \longmapsto \gamma^{\ast}m - m, \gamma \in \Gamma. 
 \end{equation}
The first \textit{cohomology group} $H^{1}(\Gamma; \mathfrak{M})$ is the quotient 
$Z^{1}(\Gamma; \mathfrak{M}) / B^{1}(\Gamma; \mathfrak{M})$. 
\item The most important cases of $\mathfrak{M}$ from the viewpoint of this thesis 
are 
\begin{enumerate}
 \item $\mathcal{S}^{\infty}(T\dd)$, 
the vector space of smooth vector 
fields on $\dd$. $\Gamma$ acts on $\mathcal{S}^{\infty}(T\dd)$ in the following manner
\begin{equation}
\label{actiononvector}
  \gamma^{\ast}F = F(\gamma) \gamma'^{-1}, \quad \gamma \in \Gamma, F \in 
\mathcal{S}^{\infty}(T\dd).
\end{equation}
\item $\mathrm{HOL}$, the vector space of holomorphic vector fields on $\dd$. 
$\Gamma$
acts on $\mathrm{HOL}$ in the same manner as in (\ref{actiononvector}). 
\item $\mathfrak{g}$, the vector space of Killing vector fields on $\dd$. 
Note that we 
have already 
dealt with this case in \textbf{\cref{cohomo}} in
\textbf{\cref{tangentteichmuellerspace}}. 
\end{enumerate}
\item Note that 
\begin{equation*}
 \label{group3}
 \mathfrak{g} \subset \mathrm{HOL} \subset \mathcal{S}^{\infty}(T\dd).
\end{equation*}
\end{enumerate}
Recall \textbf{\cref{chapter3section2}} in \textbf{\cref{chapter3}}. 
Given a holomorphic 
quadratic differential $q$ on $\dd$ which satisfies boundedness conditions, namely, 
$q$ is bounded in the hyperbolic metric $\textbf{g}_{\dd}$ of $\dd$, and 
the first and the second covariant derivative of $q$ w.r.t the linear connection 
on $T^{\ast} \dd \otimes_{\cc} T^{\ast}\dd$ are bounded in $\textbf{g}_{\dd}$, 
we obtain a harmonic vector field $\chi$ on $\dd$ that extends continuously on 
the boundary circle $\ss^{1}$ such that 
$(\mathcal{L}_{\chi}\textbf{g}_{\dd})^{(2, 0)}= q$. 
Note that $\chi$ is not necessarily 
tangential to the boundary 
circle $\ss^{1}$. We will denote the restriction of $\chi$ to $\ss^{1}$ by
$\chi|_{\ss^{1}}$. Using Theorem \ref{infinitedimproblem} (2), we can write 
$\chi|_{\ss^{1}}$ as $\chi_{1} +\chi_{2}$, where $\chi_{1} 
\in \mathfrak{X}_{\textrm{tangential}}(\ss^{1})$ and $\chi_{2} \in \mathcal{H}^{2}$. 
Since $\chi$ is a harmonic vector field on $\dd$ whose associated holomorphic 
quadratic differential is $q$, then the holomorphic quadratic differential 
associated with the vector field $\chi_{1}$ is the same $q$. Because the  
holomorphic quadratic differential associated with $\chi_{2}$ is zero. 
Notice that
in the expression of $\chi_{1}=\chi-\chi_{2}$ we are working with 
the holomorphic extension of $\chi_{2}$ 
to the open unit disk $\dd$. 
Now, the coboundary of $\chi$, i.e.,
$$\delta \chi(\gamma) = \chi(\gamma) \gamma'^{-1} - \chi, \quad \forall \gamma \in \Gamma$$
is a cocycle with values in $\mathrm{HOL}$ because of the $\Gamma$-invariance of $q$. 
But our goal is to get a cocycle with values in 
$ \mathfrak{g}$, where $\mathfrak{g}$ is the Lie algebra 
of $\mathrm{Isom}^{+}(\dd)$. 
Using 
Theorem \ref{infinitedimproblem} (1), we can easily see that for every 
$\gamma \in \Gamma$,  
$\delta (\chi_{1})(\gamma) \in \mathfrak{X}_{\textrm{tangential}}(\ss^{1}) 
\cap  \mathcal{H}^{2}$ 
and therefore we get a cocycle in 
$\mathfrak{X}_{\textrm{Killing}}(\ss^{1}) \cong \mathfrak{g}$. 
We summarize our discussion as
\begin{theorem}
\label{summarytheo}
Given a holomorphic quadratic differential $q=fdz^{2}$ on 
the Poincar\'{e} disk $\dd$ which satisfies the following boundedness conditions:
\begin{enumerate}
 \item $q$ is bounded in the hyperbolic metric on $\dd$, i.e., 
 $$\lVert q \rVert_{\textbf{g}_{\dd}} \leq D, $$ 
 where $D$ is a positive real number. 
 \item The first and the second covariant derivative of $q$ w.r.t 
 the linear connection on $T^{\ast}\dd \otimes_{\cc} T^{\ast} \dd$ are bounded in 
 $\textbf{g}_{\dd}$. 
\end{enumerate}
Then there exists a harmonic vector field $\chi$ on $\dd$ 
which $L^{2}$-extends to the closed disk $\overline{\dd}$ such that 
$(\mathcal{L}_{\chi}\textbf{g}_{\dd})^{(2, 0)}= q$. 
Moreover, the restriction of that extension to the boundary circle $\ss^{1}$ 
is tangential and $\chi$ is unique upto the addition 
of holomorphic vector fields on $\dd$ which extend tangentially to the boundary 
circle $\ss^{1}$. From Theorem \ref{infinitedimproblem} (1), $\chi$ is 
unique upto the addition of the vector space 
$\mathfrak{g}$ of Killing vector fields on $\dd$. 
\end{theorem}

\begin{coro}
\label{summarytheocor}
Let $\Gamma$ denote a subgroup of $\mathrm{Isom}^{+}(\dd)$ where $\mathrm{Isom}^{+}(\dd)$ is 
the group of orientation preserving 
isometries of $\dd$. 
 If $q=fdz^{2}$ and $\chi$ are related as in Theorem \ref{summarytheo} and if in 
 addition to (1) and (2) in Theorem \ref{summarytheo}, 
 $q$ is $\Gamma$-invariant, i.e., 
$$f(\gamma(z))\gamma'(z)^{2}=f(z), \quad \forall \gamma \in \Gamma, z \in \dd,$$
then $\delta \chi$ defined by 
$$ \gamma \longmapsto \chi(\gamma) \gamma'^{-1}-\chi, \quad \forall \gamma \in \Gamma$$
is a $1$-cocycle $c$ with coefficients in the $\Gamma$-module $\mathfrak{g}$ 
- the Lie algebra 
of $\mathrm{Isom}^{+}(\dd)$ and its cohomology class $[c]$ depends only on $q$. 
\end{coro}

\begin{proof}
From Theorem \ref{summarytheo}, we know that $\chi$ is unique upto the addition of 
Killing vector fields on $\dd$, hence for every $\gamma \in \Gamma$, 
$\delta \chi(\gamma)$ is a holomorphic vector field which extends tangentially 
to the boundary circle $\ss^{1}$. Therefore, for every $\gamma \in \Gamma$, 
$\delta \chi(\gamma) \in \mathfrak{g}$. Recall that we have for 
every $\gamma \in \Gamma$, 
$c(\gamma) = \frac{\chi(\gamma)}{\gamma'} - \chi$.  Since $\chi$ is well-defined 
upto addition of a Killing vector field $X$ on $\dd$, it follows that 
$c$ is well defined upto addition of $\delta X$. Hence, the cohomology 
class $[c]$ of $c$ is well defined.
\hfill \qedsymbol
\end{proof}

\begin{remark}
 \normalfont
 In Corollary \ref{summarytheocor}, we view $\chi$ as a $0$-cochain with values 
 in the vector space of harmonic vector fields on $\dd$. 
\end{remark}

\begin{coro}
\label{onewaymap}
 Let $\Gamma$ in Corollary \ref{summarytheocor} be a discrete cocompact subgroup 
 of $\mathrm{Isom}^{+}(\dd)$. Then 
 we have an injective mapping 
 \begin{equation}
  \label{mainmap1}
  \begin{split}
   \varPhi: \mathrm{HQD}(\dd, \Gamma) & \longrightarrow H^{1}(\Gamma; \mathfrak{g}) \\
   q & \longmapsto [c],
  \end{split}
\end{equation}
where $\mathrm{HQD}(\dd, \Gamma)$ denotes the vector space of $\Gamma$-invariant holomorphic quadratic 
differentials on $\dd$ and $c= \delta \chi$. 
\end{coro}

\begin{proof}
We assume that $\varPhi(q)=[c]=0$. Then, there exists an element $X \in \mathfrak{g}$ 
such that $c = \delta (X)$. By setting $Y = \chi - X$ we notice that the holomorphic 
quadratic differential associated to $Y$ is $q$, and $\delta Y =0$, i.e., 
$Y$ is invariant under the action of $\Gamma$. 
Therefore, $Y$ can be viewed as a harmonic vector field on the the surface 
$\dd / \Gamma$. From \cite[Proposition 4.2]{dodson}, on a two dimensional compact orientable Riemannian manifold 
without boundary, a harmonic vector field is a conformal vector field. 
Therefore, 
$q \equiv 0$. 
 \hfill \qedsymbol
\end{proof}

\section{Going from the cohomological description to the analytic description}
\label{cohomtoanal}
First we set some conventions.  
The group $\mathrm{SU}(1, 1)$ is the set of matrices
$$\mathrm{SU}(1,1) = \bigg \{ \begin{bmatrix}
                               a & b \\
                               \bar{b} & \bar{a}
                              \end{bmatrix} \in \mathrm{GL}(2, \cc) \big| 
                              |a|^{2}-|b|^{2}=1
 \bigg \},$$
 with group multiplication given by matrix multiplication. Note that the group 
 $\mathrm{SU}(1, 1)$ is isomorphic to the group $\mathrm{SL}(2, \rr)$ of $2 \times 2$ 
 real matrices with determinant $1$. We identify the circle group $\mathrm{SO}(2)$ with 
 the subgroup of $\mathrm{SU}(1, 1)$ given by 
 $$\mathrm{SO}(2) = \bigg \{ \begin{bmatrix}
                               \exp{(\iota \theta)} & 0 \\
                               0 & \exp{(-\iota \theta)}
                              \end{bmatrix} \bigg | \quad \theta \in [0, 2 \pi)
 \bigg \}. $$ 
 
 Recall that 
$\mathrm{Aut}(\dd)$, the 
orientation preserving isometries of the Poincar\'{e} disk $\dd$
with the hyperbolic metric $\textbf{g}_{\dd}$,  
is identified with 
$$\mathrm{PSU}(1, 1) = \mathrm{SU}(1, 1) / \{ \pm \mathrm{Id}\}$$
because every $\gamma \in \mathrm{PSU}(1,1)$ acts on $\dd$ by the following formula
\begin{equation*}
 \label{actiononthedisk}
 \gamma(z) = \frac{az+b}{\bar{b}z+\bar{a}}, \gamma = \begin{bmatrix}
                               a & b \\
                               \bar{b} & \bar{a}
                              \end{bmatrix}, 
 \quad |a|^{2}-|b|^{2}=1, \quad \forall z \in \dd.
\end{equation*}
\subsection{$\Gamma$-invariant partition of unity on $\dd$}
\label{cohomtoanalsection1}
Recall that a partition of unity subordinate to an open covering $\{ U_{i}\}$ of a manifold $M$ is a collection 
$\{\varphi_{i}\}$ of non-negative smooth functions such that 
\begin{enumerate}
 \item $\mathrm{supp}(\varphi_{i}) \subset U_{i}$. 
 \item Each $p \in M$ has a neighborhood that intersects with only finitely many $\mathrm{supp}(\varphi_{i})$. 
 \item $\sum \varphi_{i} =1$. 
\end{enumerate}

Let $\Gamma$ be a discrete cocompact subgroup of $\mathrm{PSU}(1, 1)$.
Below we give the existence of a $\Gamma$-invariant 
partition of unity on $\dd$. 
\begin{lemma}
\label{partition}
 There exists a smooth function $\varphi$ on $\dd$ such that
 \begin{enumerate}
  \item $0 \leq \varphi \leq 1$.
  \item For each $z \in \dd$, there is a neighborhood $U$ of $z$ and a finite subset 
  $S$ of $\Gamma$ such that $\varphi=0$ on $\gamma(U)$ for every $\gamma \in 
 \Gamma-S$.
 \item $\sum_{\gamma \in \Gamma} \varphi(\gamma(z))=1$ on $\dd$.
 \end{enumerate}
\end{lemma}
\begin{proof}
 We choose an open covering $\{U_{i}\}_{i \in I}$ of the closed 
 surface $\dd / \Gamma$ where 
 each $U_{i}$ is simply connected and a smooth partition of unity $\{\alpha_{i}\}$ 
 subordinate to the covering $\{U_{i}\}_{i \in I}$. 
 For each $U_{i}$, we choose a single component 
 $V_{i}$ of $\pi^{-1}(U_{i})$ where $\pi: \dd \longrightarrow \dd / \Gamma$ 
 is the projection map, and set 
 \begin{equation*}
 \phi_{i}(z) = \begin{cases}
             \alpha_{i}(\pi(z)), & \quad z \in V_{i} \\
             0, & \quad z \in \dd - V_{i}. 
            \end{cases}
 \end{equation*}
 Note that the mapping $\pi$ restricted to each component of $\pi^{-1}(U_{i})$ is a
 one-to-one covering. It's clear that $\phi_{i} \in C^{\infty}(\dd)$, and that 
 $\phi = \sum_{i} \phi_{i}(z), z \in \dd$ has the required properties. \hfill \qedsymbol 
 \end{proof}
 \begin{remark}
  \normalfont
 We suspect that Lemma \ref{partition} is a simpler version 
 of results on \textit{Kleinian groups} (see \cite{Kra}).
 \end{remark}
 \smallskip
 
 To go from the cohomological description of tangent spaces 
(to the Teichmueller space) 
to the analytic description which is given by the space of holomorphic 
quadratic differentials on 
$\Sigma_{g}$, we first construct a tangential vector field on the circle $\ss^{1}$ 
(recall \textbf{\cref{vectorfieldsonthecircle}} from \textbf{\cref{analtocohom}}) 
from a 
cocycle $c$ representing a cohomology class 
$[c] \in H^{1}(\Gamma; \mathfrak{g})$, where $\mathfrak{g}$ is the Lie algebra of 
the group of orientation preserving isometries of $\dd$. 
We use Lemma \ref{partition} to get the following: 
given any $[c] \in H^{1}(\Gamma; \mathfrak{g})$ 
we set 
$$\psi(z)= - \sum_{\gamma \in \Gamma} \varphi(\gamma(z)) 
c_{\gamma}(z), \quad z \in \dd.$$ 

\begin{lemma}[\cite{Kra}]
\label{continuousvector}
 $\psi$ is a $C^{\infty}$-vector field on $\dd$ such that for 
 $A \in \Gamma$, $z \in \dd$,
 \begin{equation}
 \label{actiononvector1}
  (A^{\ast}\psi)(z)-\psi(z)= c_{A}(z).
 \end{equation}
\end{lemma}
\begin{proof} Recall (\ref{actiononvector}). Consider the L.H.S of 
(\ref{actiononvector1}) in the Lemma, we have
 \begin{equation*}
  \begin{split}
  (A^{\ast}\psi)(z)-\psi(z) & = - \sum_{\gamma \in \Gamma} 
  \bigg( \varphi(\gamma (A z)) 
  c_{\gamma} (Az) A'(z)^{-1} - \varphi(\gamma(z)) c_{\gamma}(z)  \bigg) \\
              & = - \sum_{\gamma \in \Gamma} \bigg(\varphi(\gamma (A z)) 
              \bigg(c_{\gamma \circ A}(z)
              - c_{A}(z)\bigg) - \varphi(\gamma(z)) c_{\gamma}(z) \bigg) \\
              & = \sum_{\gamma \in \Gamma} \varphi(\gamma (Az)) c_{A}(z) = c_{A}(z). 
  \end{split}
\end{equation*}
The second equality in the above equation follows from the fact that $c$ is a 
cocycle. Therefore, 
$$\delta \psi=c.$$ 
\hfill \qedsymbol
\end{proof}
\begin{remark}
 \normalfont
Let 
$\mathcal{S}^{\infty}(T \dd)$ denote the vector space of $C^{\infty}$-vector 
fields on $\dd$. 
From Lemma \ref{continuousvector}, we have 
$H^{1}(\Gamma; \mathcal{S}^{\infty}(T \dd))= \{0\}$.

\end{remark}

\begin{coro}
\label{continuousvector1}
 If $\mathrm{HOL}$ is the vector space of holomorphic vector fields on $\dd$, 
 then for 
 every cocycle $c$ representing a cohomology class $[c] \in  
 H^{1}(\Gamma; \mathrm{HOL})$, there is a $\psi \in \mathcal{S}^{\infty}(T \dd)$ 
 such that 
 $$c= \delta \psi.$$
\end{coro}
\begin{proof}
 The injection of $\mathrm{HOL}$ into $\mathcal{S}^{\infty}(T \dd)$ induces a mapping 
 $$H^{1}(\Gamma; \mathrm{HOL}) \longrightarrow 
 H^{1}(\Gamma; \mathcal{S}^{\infty}(T \dd)).$$
 \hfill \qedsymbol
\end{proof}
\begin{remark}
\label{continuousvector2}
 \normalfont
Corollary \ref{continuousvector1} is true if we replace $\mathrm{HOL}$ 
by the vector space of Killing vector fields $\mathfrak{g}$ on $\dd$ because 
of $\mathfrak{g} \subset \mathrm{HOL} \subset \mathcal{S}^{\infty}(T\hh)$. 
\end{remark}
Let $c$ be a $1$-cocycle with values in the vector space 
$\mathfrak{g}$ of  
Killing vector fields on $\dd$. From \textbf{\cref{chapter3}} and  
\textbf{\cref{analtocohom}} we know that there exists a harmonic vector field $\chi$ 
with a tangential $L^{2}$-extension on the boundary circle $\ss^{1}$ such that 
$\delta \chi =c$.
From Lemma \ref{continuousvector}, Corollary \ref{continuousvector1}, and Remark 
\ref{continuousvector2}, we get another $0$-cochain $\psi$ in 
$\mathcal{S}^{\infty}(T \dd)$ such that 
$\delta \psi=c$. 
Therefore, $\chi-\psi$ is a $0$-cocycle in $\mathcal{S}^{\infty}(T \dd)$ and   
$\chi-\psi$ is invariant under the action of $\Gamma$, i.e.,
\begin{equation}
\label{newvectorfield}
\begin{split}
(\chi-\psi) & = \gamma^{\ast}(\chi-\psi) \\
& = \big((\chi-\psi)(\gamma)\big)\gamma'^{-1}, \quad  \forall \gamma \in \Gamma.
\end{split}
\end{equation}
Hence, $\chi-\psi$ is bounded in the hyperbolic metric on $\dd$.
\begin{coro}
\label{shortcoro}
 $\psi$ admits an $L^{2}$-extension to the closed unit disk $\overline{\dd}$ whose 
 restriction $\psi^{\sharp}$ to the boundary circle $\ss^{1}$ is tangential. 
\end{coro}
\begin{remark}
 \normalfont
 Note that in Corollary \ref{shortcoro} 
 such an extension is unique and it depends only on $c$, not on the choice of $\varphi$ in 
 Lemma \ref{partition}. 
\end{remark}

\subsection{The Poisson map adapted to vector fields}
\label{cohomtoanalsection2}
To get a vector field which is harmonic on the 
interior of $\dd$ from a tangential vector field on $\ss^{1}$, we first give the 
reincarnation of \textit{the Poisson integral formula} and then adapt it to the case 
of vector fields. Recall that the \textit{Dirichlet problem} asks for finding
a harmonic 
function $F$ on the disk $\dd$ given a continous function $f$ on the boundary circle 
$\ss^{1}$ such that they together make a continous function on the closed 
disk $\overline{\dd}$. 
The Poisson integral map is an important tool to solve the Dirichlet problem:
\begin{equation}
 \label{thepoissonmap}
 F(r e^{\iota \theta}) = \frac{1}{2 \pi} \int_{0}^{2 \pi}  f(e^{\iota \phi}) 
\frac{1-r^{2}}{1+r^{2}-2r \cos(\theta-\phi)} d \phi.
\end{equation}

The term $\frac{1-r^{2}}{1+r^{2}-2r \cos(\theta-\phi)}$ is called the 
\textit{Poisson Kernel} and denoted by $K$. When $z=r e^{\iota \theta}$ and 
$w=e^{\iota \phi}$, we have 
\begin{equation}
\label{poissonnewform}
 K(w, z) = 
\frac{|w|^{2}-|z|^{2}}{|w-z|^{2}} = \Re \bigg( \frac{w+z}{w-z}\bigg).
\end{equation}

Note that $K(w, z)$ is defined for $0 \leq |z| < |w| \leq 1$. 
We assume that $|w|=1$, then 
$$K(w, z)= 
\frac{1-|z|^{2}}{|1-z \bar{w}|^{2}},$$ 
since $|w-z|=|w\bar{w}-z\bar{w}|=|1-z \bar{w}|$. Therefore, 
$$\frac{1-|z|^{2}}{|1-z \bar{w}|^{2}} = 
\frac{1-z \bar{z}}{(1-z \bar{w})(1-\bar{z}w)} = 
\sum_{n=0}^{\infty} \bar{z}^{n} w^{n} + \sum_{n=1}^{\infty} z^{n} \bar{w}^{n}. $$
So, 
$$K(e^{\iota \phi}, r e^{\iota \theta})= \sum_{n=-\infty}^{\infty} r^{|n|} 
e^{\iota n(\phi-\theta)} = K_{r}(\phi-\theta).$$

It is obvious that $K$ is a positive function of $w$ and $z$. 
So, (\ref{thepoissonmap}) can also be written as 
$$F(re^{\iota \theta}) = \frac{1}{2 \pi} \int_{0}^{2 \pi} K_{r}(\theta-\phi) 
f(e^{\iota \phi}),$$
where $K_{r}(\theta-\phi) = K_{r}(\phi-\theta)$.  

\subsubsection{Reincarnation of the Poisson integral formula}
\label{invariancepoisson}
We denote the space of continuous functions on the circle $\ss^{1}$ 
by $C^{0}(\ss^{1})$ and 
the space of continuous functions on the 
open unit disk $\dd$ by $C^{0}(\dd)$. To construct and characterise the 
Poisson map 
$$P:C^{0}(\ss^{1}) \longrightarrow C^{0}(\dd) $$
given in 
(\ref{thepoissonmap}) which is continuous w.r.t to the 
topology of uniform convergence on both the source and the target space, we first 
observe that $P(f)(0)$ is nothing but the \textit{normalised Haar integral}\footnote{Let 
$G$ denote a locally compact group. The real vector space of the 
real valued continuous 
functions on $G$ with compact support is denoted by $C_{c}(G)$. The set of 
nonnegative 
functions in $C_{c}(G)$ is denoted by $C_{c}^{+}(G)$. 
A continuous linear functional 
$I: C_{c}(G) \longrightarrow \rr$ is called a \textit{Haar integral} if the following 
hold: 1) if $f \in C_{c}^{+}(G)$, then $I(f) \geq 0$, 2) if $g \in G$ and 
$f \in C_{c}(G)$, then $I(gf)=I(f)$, 3) 
there exists a function $f \in C_{c}^{+}(G)$ with $I(f) > 0$. Note that for 
$r >0$, $rI$ is again a Haar integral. For more information, see 
\cite{kap}, \cite{stro}.
} 
$$\frac{1}{2\pi}\int_{\ss^{1}} f.$$

By convention, integral of the constant function $1$ 
over $\ss^{1}$ is $2\pi$. 
Therefore, $P(f)(0)$ is linear, positive, 
continous, and invariant under the circle group. To obtain the expression for 
$P(f)(z)$, $z \in \dd$, we use the transitivity of the action of $\mathrm{PSU}(1, 1)$ 
on the open unit disk $\dd$, i.e., $P(f)(z)= P(f)(\gamma(0))$ for some 
$\gamma \in \mathrm{PSU}(1, 1)$ such that $\gamma(0)=z$. Moreover, 
\begin{equation}
\label{reincarnationpoisson}
 P(f)(z)= P(f)(\gamma(0)) = P(f \cdot \gamma)(0) = \frac{1}{2\pi} \int_{\ss^{1}} f \cdot \gamma,
\end{equation}

where the second equality follows from the fact that the Poisson map $P$ is
\textit{ 
$\mathrm{PSU}(1, 1)$-equivariant}, i.e., 
 $P(f \cdot \gamma) = P(f) \cdot \gamma$, 
 for all $\gamma \in \mathrm{PSU}(1, 1)$ and all $f \in  C^{0}(\ss^{1})$, where $\cdot$ 
 denotes the action of $\mathrm{PSU}(1, 1)$ on $C^{0}(\ss^{1})$ and 
 $C^{0}(\dd)$ by pre-composition. The condition can 
 also be understood as 
 the following commutative 
 diagram:
$$\xymatrix{
 C^{0}(\ss^{1}) \ar[r]^{P} \ar[d]^{\gamma \cdot} & C^{0}(\dd) \ar[d]^{\gamma \cdot} \\
 C^{0}(\ss^{1}) \ar[r]^{P} &
 C^{0}(\dd) 
 }$$
 
The $\mathrm{PSU}(1, 1)$-equivariance of the Poisson map follows from the
uniqueness of solutions to the Dirichlet problem for Laplace's equation, i.e., 
for a given $f \in C^{0}(\ss^{1})$, the Dirichlet problem 
for Laplace's equation 
\begin{equation*}
 \begin{split}
  \Delta F & = 0 \hspace{3pt} \text{on} \hspace{3pt} \dd \\
  F & = f \hspace{3pt} \text{on} \hspace{3pt} \ss^{1}
 \end{split}
\end{equation*}
has atmost one 
solution $F \in C^{2}(\dd) \cap C^{1}(\overline{\dd})$. Transforming 
$f \in C^{0}(\ss^{1})$ by an element $\gamma \in \mathrm{PSU}(1, 1)$ gives 
us a new harmonic extension $F_{1}$ of $f \cdot \gamma$ on $\dd$. 
From the weak maximum principle applied to the harmonic function $F \circ \gamma-F_{1}$, we 
have $F\circ \gamma-F_{1} \leq \mathrm{max}_{\ss^{1}} (F\circ \gamma-F_{1})=0$. 
Thus, $F\circ \gamma \leq F_{1}$ on 
$\dd$. Similarly, we get $F_{1} \leq F\circ \gamma$. Therefore, $F \circ \gamma$ and $F_{1}$ coincide. 
Note that the last equality in (\ref{reincarnationpoisson}) follows from the 
fact that $P(f)(0)$ is the Haar integral. $P(f)(z)$ is well-defined, i.e., it does not 
depend on $\gamma \in \mathrm{PSU}(1, 1)$ and is unique upto a positive 
scaling factor because if we take $z$ to be the origin again, then the stabilizer 
subgroup of $\mathrm{PSU}(1, 1)$ w.r.t to the origin is the circle group $SO(2)$ and the 
Haar integral is invariant under rotations. 
We list 
the following properties which are satisfied by $P$: 
\begin{enumerate}
\item $P$ is linear,
 \item $P$ is continous, 
 \item $P$ is $\mathrm{PSU}(1, 1)$-equivariant.
 \end{enumerate}
 \begin{prop}
\label{zerotoz}
 Given a point $z \in \dd$, the map $w \longmapsto \frac{w+z}{w\bar{z}+1}$ is a hyperbolic isometry that sends the 
 origin to the point $z$.
\end{prop} 

\begin{proof}
We check that indeed $\gamma(0)=z$. 
Let $\gamma(z)= \frac{w+z}{w\bar{z}+1}$, 
and let $\gamma(w)=f(w)+\iota g(w)$. By differentiating, we get 
$\frac{d\gamma(w)}{dw}= \frac{1-|z|^{2}}{(w\bar{z}+1)^{2}}$. Observe that 
$$df(w)^{2}+dg(w)^{2}= (df(w)+\iota dg(w))(df(w)-\iota dg(w)) = 
d\gamma(w) \overline{d\gamma(w)}.$$ Therefore, 
\begin{equation*}
\resizebox{1.03\hsize}{!}{$
 \frac{2 \sqrt{df(w)^{2}+dg(w)^{2}}}{1-f(w)^{2}-g(w)^{2}} = 
\frac{2\sqrt{d\gamma(w) \overline{d\gamma(w)}}}{1-|\gamma(w)|^{2}} = 
\frac{2 \sqrt{\frac{d\gamma(w)}{dw} dw \frac{\overline{d\gamma(w)}}{\overline{dw}} 
\overline{dw}}}{1-|\gamma(w)|^{2}} = \frac{2 \sqrt{ \frac{(1-|z|^{2})^{2}}
{(w\bar{z}+1)^{2}(\bar{w}z+1)^{2}} }}{1-|\gamma(w)|^{2}} \sqrt{dx^{2}+dy^{2}}.
$}
\end{equation*}
Simplifying $1-|\gamma(w)|^{2}$ further, we get 
$$\frac{2 \sqrt{ \frac{(1-|z|^{2})^{2}}
{(w\bar{z}+1)^{2}(\bar{w}z+1)^{2}} }}{1-|\gamma(w)|^{2}} = \frac{2}{1-|w|^{2}}.$$ 
Therefore, $\gamma$ is a 
hyperbolic isometry. The final and remaining thing is to check that 
$\gamma$ maps $\dd$ to itself. Suppose that $|w| <1$. 
We want to show that 
$|\frac{w+z}{w\bar{z}+1}|<1$. This is equivalent 
to showing that $|w+z| <  |w\bar{z}+1|$. 
Furthermore, it is enough to show that 
$(w+z)(\bar{w}+\bar{z}) < (w\bar{z}+1)(\bar{w}z+1)$, or equivalently, 
$w\bar{w} + z \bar{z} < w\bar{w} z \bar{z} +1$, or $(1-w\bar{w})(1-z \bar{z}) >0$, 
which is true since $1-w\bar{w}$ and $1-z \bar{z}$ are both positive. 
 \hfill \qedsymbol
\end{proof} 
\bigskip
We summarize our discussion as follows:
\begin{prop}
\label{poissonrecons}
Every continuous linear map 
 $F: C^{0}(\ss^{1}) \longrightarrow C^{0}(\dd)$ which is 
 $\mathrm{PSU}(1, 1)$-equivariant is a scalar multiple of the continous linear map 
 $P: C^{0}(\ss^{1}) \longrightarrow C^{0}(\dd)$ 
given by the following formula
\begin{equation*}
 \label{poissonre}
 P(f)(z) = \frac{1}{2\pi} \int_{\ss^{1}} f \cdot \gamma,
\end{equation*}
where $\gamma \in \mathrm{PSU}(1, 1)$ is given in Proposition \ref{zerotoz} 
such that $\gamma(0)=z$ and $f  \in C^{0}(\ss^{1})$. 
\end{prop}
\begin{remark}
\label{acharemark}
 \normalfont
 Alternatively, we can construct such a linear map 
 $F: C^{0}(\ss^{1}) \longrightarrow 
 C^{0}(\dd)$ in Proposition \ref{poissonrecons}
 by plugging the Dirac distribution $\delta$ 
 at the point 
  $1 \in \ss^{1}$ into the formula for $P$ instead of a 
  continous function $f$ on the circle $\ss^{1}$. 
  We adopt the view that $\delta$ is the limit of step functions $\{\epsilon^{-1}
  g_{\epsilon}\}$ where $g_{\epsilon}$ is the characteristic function
  of an arc of length $\epsilon$ centered at $1 \in \ss^{1}$. 
  Therefore, we define $\delta \cdot \gamma = \gamma^{\ast}\delta$ to be the 
  Dirac distribution at 
  the point
  $\gamma^{-1}(1)$ times $\big|\big(\gamma'(\gamma^{-1}(1))\big)^{-1}\big|$. 
 This suggests  
   \begin{equation*}
   F(\delta)(z)  = \frac{1}{2 \pi}\int_{\ss^{1}} \delta \cdot \gamma,
\end{equation*}
where $\gamma(0)=z$ and the explicit form of $\gamma$ is given by 
Proposition \ref{zerotoz}. Using $\gamma(w) = \frac{w+z}{w \bar{z}+1}$, 
we see that 
$$2 \pi \cdot \big(F(\delta)(z)\big) = \big|\big(\gamma'(\gamma^{-1}(1))\big)^{-1}\big| = 
 \frac{1-|z|^{2}}{|1-\bar{z}|^{2}} = \frac{1-|z|^{2}}{|1-z|^{2}}.
 $$
 We denote the real valued (positive) function $z \longmapsto \frac{1-|z|^{2}}{|1-z|^{2}}$ 
 defined on $\dd$ for $0 \leq |z| < 1$ by $K$. 
 The intuition is $F(\delta)=\frac{K}{2 \pi}$ and therefore, we define 
 \begin{equation}
 \label{newpois}
  F(f) = f \ast K,
 \end{equation}
where $f \in C^{0}(\ss^{1})$ and $K(z)=\frac{1-|z|^{2}}{|1-z|^{2}}$, and $\ast$
denotes the convolution\footnote{The convolution of $K$ and 
$f$ is defined as:
$ (f \ast  K)(z):= \frac{1}{2 \pi} \int_{\ss^{1}} f(w) K(zw^{-1}) dw.$}
of $K$ and $f$. 
To show the $\mathrm{PSU}(1, 1)$-equivariance, we first note that 
every element $A \in \mathrm{PSU}(1, 1)$ has a unique expression $A= BC$ where 
$B \in \mathrm{SO}(2)$ and $C$ is in the two-dimensional subgroup $\mathrm{Stab}_{\mathrm{PSU}(1, 1)}(1)$ 
of $\mathrm{PSU}(1, 1)$ consisting of 
all elements which fix the element $1$ in the boundary circle $\ss^{1}$. 
Also, $\mathrm{Stab}_{\mathrm{PSU}(1, 1)}(1)$ acts transitively on $\dd$. 
The general form of elements $\gamma$ of 
the group $\mathrm{Stab}_{\mathrm{PSU}(1, 1)}(1)$  
is given by the following:
\begin{equation}
 \label{gammafix1}
 \gamma(z)=\frac{az+b}{\bar{b}z+\bar{a}}, |a|^{2}-|b|^{2}=1, a+b = \overline{a+b}.
\end{equation}
Hence, showing 
 the $\mathrm{PSU}(1, 1)$-equivariance of $F$ is equivalent to showing 
 the $\mathrm{SO}(2)$-equivariance and $\mathrm{Stab}_{\mathrm{PSU}(1, 1)}(1)$-equivariance of $F$.  
 It is easy to see that $F$ in (\ref{newpois}) is $\mathrm{SO}(2)$-equivariant. 
 To show the $\mathrm{Stab}_{\mathrm{PSU}(1, 1)}(1)$-equivariance of 
 $F$ in (\ref{newpois}), we claim that $\gamma^{\ast}K= cK$, where $c$ is a positive
 constant and $\gamma \in \mathrm{Stab}_{\mathrm{PSU}(1, 1)}(1)$. We have
 \begin{equation*}
  \begin{split}
   K(\gamma(z)) & = \frac{1-|\gamma(z)|^{2}}{|1-\gamma(z)|^{2}} = 
   \frac{1-\gamma(z)\overline{\gamma(z)}}{\big(1-\gamma(z)\big)
   \big(1-\overline{\gamma(z)}\big)}\\
   & = \frac{1 - \frac{az+b}{\bar{b}z+\bar{a}} \cdot 
 \frac{\bar{a}\bar{z}+\bar{b}}{b\bar{z}+a}}
 {\bigg(1-\frac{az+b}{\bar{b}z+\bar{a}}\bigg) \cdot 
 \bigg(1-\frac{\bar{a}\bar{z}+\bar{b}}{b\bar{z}+a}\bigg)}
   = \frac{1 - \frac{|a|^{2}|z|^{2} + az \bar{b}+b\bar{a}\bar{z}+|b|^{2}}
 {|b|^{2}|z|^{2}+az \bar{b}+b\bar{a}\bar{z}+|a|^{2}}} 
 {\bigg(\frac{(b-\bar{a})\bar{z} - (\bar{b}-a)}{b \bar{z}+a} \bigg) \bigg(
 \frac{(\bar{b}-a)z - (b-\bar{a})}{\bar{b} z+\bar{a}} \bigg)} \\
 & = \frac{1-|z|^{2}}
 {(b\bar{z}+a) (\bar{b}z+\bar{a})} \cdot 
 \Bigg(\frac{\big(b-\bar{a}\big)^{2} \cdot \big|1-\bar{z}\big|^{2}}
 {(b\bar{z}+a) (\bar{b}z+\bar{a})}\Bigg)^{-1} = 
 \frac{1-|z|^{2}}{\big(b-\bar{a}\big)^{2} \cdot \big|1-\bar{z}\big|^{2}} \\
 &  = 
 \frac{\gamma'(1)^{-1} \big(1-|z|^{2}\big)}{\big|1-\bar{z}\big|^{2}} = 
 \gamma'(1)^{-1} K(z), 
  \end{split}
\end{equation*}
where $\gamma'(1)=\big(\bar{b}+\bar{a}\big)^{-2}$. 
Note that $K$ is the real part of a
 holomorphic function, hence harmonic. Therefore, 
$F(f)$ is also harmonic. 
\end{remark}
\begin{coro}
 The map $F$ in (\ref{newpois}) is the Poisson map given in 
 (\ref{thepoissonmap}). Hence, the map $P$ in  
 Proposition \ref{poissonrecons} lands in the vector space of harmonic 
  functions on the open unit disk $\dd$.
\end{coro} 
\bigskip
Let $\mathcal{S}_{C^{0}}(T\ss^{1})$ be the Banach space 
of (tangential) continuous vector fields on 
$\ss^{1}$ and $\mathcal{S}_{C^{0}}(T\dd)$ be the 
space of continuous vector fields on the 
open disk $\dd$. 
We want to mimick the reincarnation of the Poisson map in the 
case 
of vector fields. 

\begin{prop}
\label{so2linear}
 Every continuous and $\mathrm{SO}(2)$-equivariant linear functional $\Lambda$ from 
 the real Banach space 
 of continous tangential vector fields on $\ss^{1}$ to $\cc$ has the following form: 
 $$\Lambda(X)= \bigg(\int_{\ss^{1}} X\bigg) \cdot v, $$
 where $X$ is a tangential vector field on $\ss^{1}$ and $v \in \cc$. 
\end{prop}


\begin{prop}
\label{prop412}
 Every continous linear map
$$\mathcal{F}: \mathcal{S}_{C^{0}}(T\ss^{1}) \longrightarrow 
\mathcal{S}_{C^{0}}(T\dd)$$ 
which is equivariant under the action of $\mathrm{PSU}(1,1)$
is a scalar multiple of the continous linear map 
$$\mathcal{P}: \mathcal{S}_{C^{0}}(T\ss^{1}) \longrightarrow 
\mathcal{S}_{C^{0}}(T\dd)$$ 
given by the following formula 
\begin{equation}
\label{poissonvector}
 \mathcal{P}(X)(z)=\mathcal{P}(X)(\gamma(0)) = 
 \gamma'(0) \cdot \Big(\mathcal{P}(\gamma^{\ast}X)(0) \Big) = 
 \gamma'(0) \cdot \bigg(\frac{1}{2\pi}\int_{\ss^{1}} \gamma^{\ast}X\bigg),
\end{equation}
for some $\gamma \in \mathrm{PSU}(1, 1)$ such that $\gamma(0)=z$. 
\end{prop}
\begin{remark}
 \normalfont
 The third equality in the expression of $\mathcal{P}(X)(z)$ in 
 (\ref{poissonvector}) follows from Proposition \ref{so2linear}. 
\end{remark}

\begin{remark}
 \normalfont
 The scalar in Proposition \ref{prop412} can be any complex number. Also, 
 note that $\mathcal{P}(X)(0) \in T_{0}\dd$ and the 
second equality in (\ref{poissonvector}) follows from the 
$\mathrm{PSU}(1,1)$-equivariance 
of $\mathcal{P}$, i.e., 
$$\mathcal{P}(\gamma^{\ast}X)=\gamma^{\ast}(\mathcal{P}(X)), \quad \forall \gamma 
\in \mathrm{PSU}(1, 1),$$ 
where $\gamma^{\ast}X=X(\gamma)\gamma'^{-1}, \gamma \in \mathrm{PSU}(1,1), X \in 
\mathcal{S}_{C^{0}}(T\ss^{1})$. 
\end{remark}

\begin{remark}
 \normalfont
 We can also construct such a linear map 
 $$\mathcal{F}: \mathcal{S}_{C^{0}}(T\ss^{1}) \longrightarrow 
\mathcal{S}_{C^{0}}(T\dd)$$ in Proposition \ref{prop412}
 by plugging the \textit{Dirac vector field} $\boldsymbol \delta$ 
 in the formula for $\mathcal{P}$ instead of a 
  tangential vector field $X$ on the circle $\ss^{1}$. 
  We adopt the view that $\boldsymbol \delta$ is the limit of 
  vector fields $\{\epsilon^{-1}
  g_{\epsilon}\}$ where $g_{\epsilon}$ is the norm $1$ (positively oriented) 
  tangential vector field 
  supported on an arc of length $\epsilon$ centered at $1 \in \ss^{1}$. 
 Therefore, we have 
 \begin{equation}
 \label{pullbackdirac}
 \mathcal{F}(\boldsymbol \delta)(z) = \mathcal{F}(\boldsymbol \delta)(\gamma(0)) = 
 \gamma'(0) \cdot \Big(\mathcal{F}\big(\gamma^{\ast} \boldsymbol \delta \big)(0)\Big)  = 
 \gamma'(0) \cdot \bigg(\frac{1}{2\pi}\int_{\ss^{1}} \gamma^{\ast} \boldsymbol \delta \bigg), 
 \end{equation}
 where $\gamma(0)=z$ and the explicit form of $\gamma$ is given by 
Proposition \ref{zerotoz}. (\ref{pullbackdirac}) is further simplified to 
\begin{equation}
 \label{pullbackdirac1}
 2 \pi \cdot \big(\mathcal{F}(\boldsymbol \delta)(z)\big) = 
 \gamma'(0) \cdot \Big(\iota \cdot \big|\big(\gamma'(\gamma^{-1}(1))\big)^{-1}\big| \cdot 
      \big(\gamma'(\gamma^{-1}(1))\big)^{-1}\Big).
\end{equation}
Observe that the factor $\big|\big(\gamma'(\gamma^{-1}(1))\big)^{-1}\big|$ 
accounts 
for the streching of the arc length when we pull back the Dirac vector 
field under $\gamma$
and the factor $\big(\gamma'(\gamma^{-1}(1))\big)^{-1}$ accounts 
for the streching of 
vectors. 
Using $\gamma(w) = \frac{w+z}{w \bar{z}+1}$, the expression 
$$\gamma'(0) \cdot \Big(\iota \cdot \big|\big(\gamma'(\gamma^{-1}(1))\big)^{-1}\big| 
\cdot 
      \big(\gamma'(\gamma^{-1}(1))\big)^{-1}\Big)$$ in (\ref{pullbackdirac1}) 
      simplifies to 
      \begin{equation}
       \label{finalharmonic}
       \frac{\iota (1-|z|^{2})^{3}}{|1-\bar{z}|^{2} \cdot (1-\bar{z})^{2}}.
      \end{equation}
We call the vector field given by (\ref{finalharmonic}) 
the \textit{Poisson kernel vector field} and denote it by $\textbf{K}$. 
By definition $\mathcal{F}(\boldsymbol \delta) = \frac{1}{2\pi}\textbf{K}$. 
Let $X$ be a tangential vector field on the boundary circle 
$\ss^{1}$ of
the form $f Y$ where $f$ is a real-valued continuous function on 
the boundary and $Y$ is the norm $1$
tangential vector field on $\ss^{1}$ given by $z \longmapsto \iota z$. From the above 
discussion, a vector field $\mathcal{F}(X)$ on $\dd$ is given 
by the convolution of the Poisson 
Kernel vector field $\textbf{K}$ with a given function $f$ on $\ss^{1}$, i.e.,
\begin{equation}
 \label{finals}
 \mathcal{F}(X) =  f \ast \textbf{K}.
\end{equation}
\end{remark}

\begin{prop}
\label{propcon}
 The map $\mathcal{F}$ in (\ref{finals})
 satisfies the conditions of the map $\mathcal{F}$ in Proposition \ref{prop412}. 
\end{prop}
Before we prove Proposition \ref{propcon}, we state and prove the following:
\begin{theorem}
\label{kernelvectorfieldisharmonic}
 The Poisson 
Kernel vector field $\textbf{K}$ given by (\ref{finalharmonic}) in Remark 
\ref{pullbackdirac1} is harmonic at every 
point $z \in \dd$. 
\end{theorem}
\begin{proof}
Recall 
Theorem \ref{thm2} in \textbf{\cref{chapter3section2}} in \textbf{\cref{chapter3}}  
where we show that a 
vector field $\xi$ on $\dd$
is harmonic iff the quadratic differential $(L_{\xi}\textbf{g}_{\dd})^{(2, 0)}$
associated with it is holomorphic.
We 
first prove that $\textbf{K}$ is harmonic at the origin in $\dd$. We write the Taylor 
approximation of $\textbf{K}$ up to 
the second order at 
the origin as follows:
\begin{equation}
 \label{finalharmonic1} 
 \begin{split}
         \textbf{K}(z) & = \frac{\iota (1-|z|^{2})^{3}}{|1-\bar{z}|^{2} \cdot (1-\bar{z})^{2}} \\
         & = \frac{\iota (1-|z|^{2})^{3} }{ (1-\bar{z}) \overline{(1-\bar{z})} 
         (1-\bar{z})^{2}} \\
         & = \iota (1-|z|^{2})^{3} (1- \bar{z})^{-3} (1-z)^{-1} \\
         & \approx \iota(1-3|z|^{2})(1+\bar{z}+\bar{z}^{2})^{3} (1+z+z^{2}) \\
         & \approx \iota(1-3|z|^{2}) (1+3 \bar{z} + 3 \bar{z}^{2}+3 \bar{z}^{2})(1+z+z^{2}) \\
         & \approx \iota(1+ 3 \bar{z} + 6 \bar{z}^{2} - 3|z|^{2} ) (1+z+z^{2}) \\
         & \approx \iota(1 + 3 \bar{z} + 6 \bar{z}^{2} - 3|z|^{2} + z + 3|z|^{2} + z^{2}) \\
         & = \iota (1 + z + 3 \bar{z} + z^{2} + 6 \bar{z}^{2}) \\
         & = \iota \big(1+ (x+ \iota y) + 3 (x- \iota y) + x^{2}-y^{2} + 2 \iota xy + 6 
         (x^{2}-y^{2}) - 12 \iota xy \big) \\
         & = \iota (1+ 4x -2 \iota y + 7x^{2}-7y^{2} - 10 \iota xy) \\
         & = (2y+10xy, 1+4x+7x^{2}-7y^{2}).
\end{split}
\end{equation}
Note that the metric $\textbf{g}_{\dd}$ at the origin does not change. Following 
the criteria for harmonicity of a vector field from \textbf{\cref{chapter3section2}} in  
\textbf{\cref{chapter3}}, 
we 
notice that the quadratic differential $q$ associated to $\textbf{K}$ is given as 
$(6\iota-24 \iota z) dz^{2}$. The function 
$f(z)= 6\iota-24 \iota z$ is holomorphic. Hence, $\textbf{K}$ is harmonic at 
the origin in $\dd$. Now, we claim that the vector field $\textbf{K}$ when transformed using elements 
$\gamma \in \mathrm{PSU}(1,1)$ which fix the element $1$ in the boundary circle 
$\ss^{1}$, 
changes 
only by multiplying it by a non-zero real constant. 
\paragraph{\textit{Proof of the claim:}} 
Recall the general form of elements $\gamma$ of 
the group 
$\mathrm{PSU}(1, 1)$ which fix the element $1$ in the boundary circle 
$\ss^{1}$ given by 
(\ref{gammafix1}) in Remark \ref{acharemark}.
Now, 
 $\gamma$ acts on 
$\textbf{K}$ in the usual way: 
\begin{equation}
\label{actiononkernelvector}
\begin{split}
 \gamma^{\ast}\textbf{K} & = \textbf{K}(\gamma(z))\gamma'(z)^{-1} \\
 & = \frac{\iota \big(1-|\gamma(z)|^{2}\big)^{3}}{\big|1-\overline{\gamma(z)}\big|^{2} 
 \cdot \big(1-\overline{\gamma(z)}\big)^{2}} \cdot \gamma'(z)^{-1}.
\end{split}
\end{equation}
Using (\ref{gammafix1}), the numerator and the denominator of the term $
\frac{\iota \big(1-|\gamma(z)|^{2}\big)^{3}}{\big|1-\overline{\gamma(z)}\big|^{2} 
 \cdot \big(1-\overline{\gamma(z)}\big)^{2}}$ in 
the R.H.S of (\ref{actiononkernelvector}) are explicitly written as: 
\begin{equation}
\label{numerator}
 \begin{split}
 \iota \big(1-|\gamma(z)|^{2}\big)^{3} & = 
 \iota \bigg(1- \gamma(z) \overline{\gamma(z)} \bigg)^{3} \\
 & = \iota \bigg(1 - \frac{az+b}{\bar{b}z+\bar{a}} \cdot 
 \frac{\bar{a}\bar{z}+\bar{b}}{b\bar{z}+a} \bigg)^{3}
 = 
 \iota \bigg(1 - \frac{|a|^{2}|z|^{2} + az \bar{b}+b\bar{a}\bar{z}+|b|^{2}}
 {|b|^{2}|z|^{2}+az \bar{b}+b\bar{a}\bar{z}+|a|^{2}}   \bigg)^{3} \\
 & = \iota \bigg( \frac{1-|z|^{2}}
 {|b|^{2}|z|^{2}+az \bar{b}+b\bar{a}\bar{z}+|a|^{2}}\bigg)^{3} 
  = \frac{\iota \big(1-|z|^{2}\big)^{3}}{\big((\bar{b}z+\bar{a})(b\bar{z}+a)\big)^{3}}
 \end{split}
\end{equation}
and 
\begin{equation}
\label{denominator}
 \begin{split}
  \big|1-\overline{\gamma(z)}\big|^{2} 
 \cdot \big(1-\overline{\gamma(z)}\big)^{2} & = 
 \big(1-\overline{\gamma(z)}\big) \overline{\big( 1-\overline{\gamma(z)}\big)}
 \cdot \big(1-\overline{\gamma(z)}\big)^{2} \\
 & = \big(1-\overline{\gamma(z)}\big) \big(1-\gamma(z)\big) \cdot \big(1-\overline{\gamma(z)}\big)^{2} \\
 & = \bigg(\frac{(b-\bar{a})\bar{z} - (\bar{b}-a)}{b \bar{z}+a} \bigg) \bigg(
 \frac{(\bar{b}-a)z - (b-\bar{a})}{\bar{b} z+\bar{a}} \bigg)
 \bigg(\frac{(b-\bar{a})\bar{z} - (\bar{b}-a)}{b \bar{z}+a} \bigg)^{2} \\
 & = \frac{\big(b-\bar{a}\big)^{2} \cdot \big|1-\bar{z}\big|^{2}}
 {(b\bar{z}+a) (\bar{b}z+\bar{a})} \cdot 
 \frac{\big(b-\bar{a}\big)^{2}\big(1-\bar{z}\big)^{2}}{\big(b \bar{z}+a\big)^{2}} 
  = \frac{\big(b-\bar{a}\big)^{4} \big|1-\bar{z}\big|^{2} \big(1-\bar{z}\big)^{2} }
 { (b\bar{z}+a)^{3} (\bar{b}z+\bar{a})},
 \end{split}
\end{equation}
where in the last two equalities in (\ref{denominator}) we have used the fact that 
$b-\bar{a}$ is real, i.e., $b-\bar{a}=\bar{b}-a$. 
Also, $\gamma'(z)^{-1}= \big(\bar{b} z+\bar{a}\big)^{2}$. 
Using (\ref{numerator}) and (\ref{denominator}), the explicit form of the R.H.S of 
(\ref{actiononkernelvector}) is 
\begin{equation*}
 \begin{split}
  \frac{\iota \big(1-|\gamma(z)|^{2}\big)^{3}}{\big|1-\overline{\gamma(z)}\big|^{2} 
 \cdot \big(1-\overline{\gamma(z)}\big)^{2}} \cdot \gamma'(z)^{-1} & = 
\frac{\iota \big(1-|z|^{2}\big)^{3} (b\bar{z}+a)^{3} (\bar{b}z+\bar{a})}
{\big((\bar{b}z+\bar{a})(b\bar{z}+a)\big)^{3} \big(b-\bar{a}\big)^{4} 
\big|1-\bar{z}\big|^{2} \big(1-\bar{z}\big)^{2}} \cdot \big(\bar{b} z+\bar{a}\big)^{2} \\
& = \frac{1}{\big(b-\bar{a}\big)^{4}} \cdot 
\frac{\iota (1-|z|^{2})^{3}}{|1-\bar{z}|^{2} \cdot (1-\bar{z})^{2}} \\
& = \frac{1}{\big(b-\bar{a}\big)^{4}} \textbf{K}(z) = 
(\gamma'(1))^{-2} \textbf{K}(z).
 \end{split}
\end{equation*}
As mentioned in Remark \ref{acharemark},  
every element $A \in \mathrm{PSU}(1, 1)$ has a unique expression $A= BC$ where 
$B \in \mathrm{SO}(2)$ and $C$ is in the two-dimensional subgroup 
$\mathrm{Stab}_{\mathrm{PSU}(1, 1)}(1)$ 
of $\mathrm{PSU}(1, 1)$ consisting of 
all elements which fix the element $1$ in the boundary circle $\ss^{1}$. 
Therefore, 
$\textbf{K}$ is $\mathrm{Stab}_{\mathrm{PSU}(1, 1)}(1)$-invariant up 
to multiplication by 
real scalars. 
Note
that the harmonicity of a vector field on the open unit disk 
$\dd$ is preserved by 
conformal automorphisms of $\dd$. 
Hence, $\textbf{K}$ is harmonic everywhere on 
the open unit disk $\dd$.
\hfill \qedsymbol
\end{proof}
\begin{remark}
 \normalfont
 $\mathrm{Stab}_{\mathrm{PSU}(1, 1)}(1)$-invariance of 
 $\textbf{K}$ up 
to multiplication by 
real scalars suffices to ensure that $\textbf{K}$ is harmonic on the open unit disk 
$\dd$ because $\mathrm{Stab}_{\mathrm{PSU}(1, 1)}(1)$ acts transitively on the 
open unit disk $\dd$.
\end{remark}
\begin{remark}
 \normalfont
 Since the Poisson Kernel vector field $\textbf{K}$ is harmonic, $\mathcal{F}(X)$
 given by (\ref{finals}) is also harmonic on $\dd$, 
 where $X$ is a tangential vector field 
 on $\ss^{1}$. 
\end{remark}
\bigskip
\paragraph{\textit{Proof of Proposition \ref{propcon}:}} 
The map $\mathcal{F}$, given by (\ref{finals}), 
is clearly $\mathrm{PSU}(1, 1)$-equivariant. It follows from 
$\mathrm{Stab}_{\mathrm{PSU}(1, 1)}(1)$-invariance of $\textbf{K}$ up to 
multiplication by real scalars (see Proof of 
Proposition \ref{kernelvectorfieldisharmonic}). Hence, it immediately 
follows that $\mathcal{F}$ satisfies all the conditions stated in Proposition 
\ref{prop412}. 
\hfill \qedsymbol
\begin{coro}
\label{coro4213}
 The map $\mathcal{F}$, given by (\ref{finals}), is same as the map 
 $\mathcal{P}$ in Proposition \ref{prop412}. 
 Hence, the map $\mathcal{P}$ in  
 Proposition \ref{prop412} lands in the vector space of harmonic 
  vector fields on the open unit disk $\dd$.
\end{coro}
\begin{lemma}\label{finaltheorem}
 For a continuous tangential vector field $X$ on $\ss^{1}$, 
 $\mathcal{F}(X)$ and $X$ together 
 make up a continuous vector field on the closed unit disk $\overline{\dd}$. 
\end{lemma}
\begin{proof}[Sketch]
For every $\epsilon > 0$, 
we get a continuous vector field $\textbf{K}_{1-\epsilon}$ on 
$\ss^{1}$ by composing 
$\textbf{K}$ with the map $z \longmapsto (1-\epsilon)z$. 
We first notice that 
\begin{equation}
\label{diractvector}
  \textbf{K}_{1-\epsilon}(z) = 
  \frac{\iota \big(1-|(1-\epsilon)z|^{2}\big)^{3}}
  {\big(1-(1-\epsilon)\bar{z}\big)^{3} \cdot \big(1-(1-\epsilon)z\big)},
\end{equation}
where $|z|=1$. Simplifying (\ref{diractvector}), we get 
\begin{equation}
 \label{diractvector1}
 \textbf{K}_{1-\epsilon}(z) \approx 
 \frac{\iota 8 \epsilon^{3} z^{3}}{\big(1-(1-\epsilon)z\big) \cdot 
 \big(z-(1-\epsilon) \big)^{3}},
\end{equation}
where we used the fact that $\bar{z}=z^{-1}$. We put $1-\epsilon =s$ in 
(\ref{diractvector1}) and get 
\begin{equation*}
\label{diractvector2}
\textbf{K}_{1-\epsilon}(z) \approx 
 \frac{\iota 8 \epsilon^{3} z^{3}}{(1-sz) \cdot 
 (z-s)^{3}}.
\end{equation*}
Let $\lambda_{z}=|z-(1-\epsilon)|$. 
Notice that 
\begin{equation}
 \label{estimatedelta}
 |\textbf{K}_{1-\epsilon}(z)| \leq \frac{8}{\epsilon}.
\end{equation}
The estimate in (\ref{estimatedelta}) is independent of $z$. 
And now, let us try to make an upper bound for $|\textbf{K}_{1-\epsilon}(z)|$ 
which is dependent on $z$. The following estimate for 
$|\textbf{K}_{1-\epsilon}(z)|$ ensures that $\textbf{K}_{1-\epsilon}$ is 'very small' 
outside the arc of length $\sqrt{2\epsilon}$ centered at $1$. 
\begin{equation}
 \label{estimatedelta1}
 |\textbf{K}_{1-\epsilon}(z)| \leq \frac{8 \epsilon^{3}}{\lambda_{z}^{4}}. 
\end{equation}
(\ref{estimatedelta}) and (\ref{estimatedelta1}) have following two consequences:
\begin{enumerate}
 \item if $X=fY$, where $f$ is a real-valued continuous function on 
 $\ss^{1}$ and $Y$ is the norm $1$ continous tangential vector field on $\ss^{1}$, 
 we will have 
 \begin{equation}
  \label{estimatedelta2}
  (f \ast K_{1-\epsilon})(z) \approx f(z) \cdot \bigg(\frac{1}{2 \pi} 
 \int_{\ss^{1}} K_{1-\epsilon} \bigg), \quad z \in \ss^{1}.
 \end{equation}
 Therefore, it is enough to show that 
 $$ \lim_{\epsilon \rightarrow 0}\bigg(\frac{1}{2 \pi} 
 \int_{\ss^{1}} K_{1-\epsilon} \bigg) = \iota. $$
 \item we may replace the ordinary Haar integral by the complex path integral 
 at the price of dividing by $\iota$. 
\end{enumerate}
Therefore,  
\begin{equation}
 \label{diractvector3}
 \begin{split}
  \lim_{\epsilon \rightarrow 0} \Bigg(\frac{1}{2\pi \iota} \int_{\ss^{1}} 
 \frac{\iota 8 \epsilon^{3} z^{3}}{\big(1-sz\big) \cdot 
 \big(z-s \big)^{3}} dz \Bigg) & = \lim_{\epsilon \rightarrow 0} 
 \Bigg( \frac{\iota 8 \epsilon^{3}}{2 \pi \iota} \int_{\ss^{1}}
 \frac{ z^{3}}{\big(1-sz\big) \cdot 
 \big(z-s \big)^{3}} dz \Bigg) \\
 & = \lim_{\epsilon \rightarrow 0} \bigg( \frac{\iota 8 \epsilon^{3}}{2 \pi \iota} \big(
 2 \pi \iota \cdot \mathrm{Res}(f, s) \big)\bigg),
  \end{split}
\end{equation}
where $f(z) = \frac{z^{3}}{(1-sz) \cdot (z-s)^{3}}$, and 
$$\mathrm{Res}(f, s) = \frac{6s-12s^{3}+8s^{5}-2s^{7}}{2(1-s^{2})^{4}}
=\frac{4 \epsilon + 2 \epsilon^{2}+2\epsilon^{3}-
30 \epsilon^{4}+34\epsilon^{5}- 14 \epsilon^{6}+2 \epsilon^{7}}
{2\big(16 \epsilon^{4}-32 \epsilon^{5}+20 \epsilon^{6}-8 
\epsilon^{7}+\epsilon^{8}\big)}.$$
Rewriting (\ref{diractvector3}), we get
\begin{equation*}
\begin{split}
 \lim_{\epsilon \rightarrow 0} \bigg( \frac{\iota 8 \epsilon^{3}}{2 \pi \iota} \Big(
 2 \pi \iota \cdot \mathrm{Res}(f, s) \big)\bigg) & = 
 \lim_{\epsilon \rightarrow 0} \Bigg(8 \iota \cdot 
 \frac{4 \epsilon + 2 \epsilon^{2}+2\epsilon^{3}-
30 \epsilon^{4}+34\epsilon^{5}- 14 \epsilon^{6}+2 \epsilon^{7}}
{2\big(16 \epsilon-32 \epsilon^{2}+20 \epsilon^{3}-8 
\epsilon^{4}+\epsilon^{5}\big)}
    \Bigg) \\
    & = 4 \iota \Bigg( \lim_{\epsilon \rightarrow 0} 
    \frac{4 \epsilon + 2 \epsilon^{2}+2\epsilon^{3}-
30 \epsilon^{4}+34\epsilon^{5}- 14 \epsilon^{6}+2 \epsilon^{7}}
{16 \epsilon-32 \epsilon^{2}+20 \epsilon^{3}-8 
\epsilon^{4}+\epsilon^{5}}
    \Bigg) \\
    & = \iota. 
    \end{split}
\end{equation*} \hfill \qedsymbol 
\end{proof}

\begin{coro}
\label{lastcoro}
 For an $L^{2}$-tangential vector field $X$ on $\ss^{1}$, $X$ is an $L^{2}$-boundary 
 extension of the smooth vector field $\mathcal{F}(X)$ on the open unit disk $\dd$. 
\end{coro} 
\begin{proof}
Notice that in the proof of Lemma \ref{finaltheorem}, we showed that 
$$\lim_{\epsilon \rightarrow 0} \textbf{K}_{1-\epsilon} 
= 2 \pi \boldsymbol \delta.$$
Hence, Corollary \ref{lastcoro} follows from Lemma \ref{finaltheorem} and 
\cite[Proposition 5.4]{shubin}. \hfill \qedsymbol
\end{proof}
\begin{remark}
 \normalfont
 We suspect that Corollary \ref{lastcoro} is an infinitesimal version of 
 the problem of finding harmonic extensions of quasiconformal maps (from $\ss^{1}$ to 
 itself) to the open unit disk $\dd$ or the upper half plane $\hh$. See \cite{hardt} 
 for more details. 
\end{remark}

\subsection{A detailed map from $H^{1}(\Gamma; \mathfrak{g})$ 
to $\mathrm{HQD}(\dd, \Gamma)$}
Here we summarize the main consequences of 
\textbf{\cref{cohomtoanalsection1}} and \textbf{\cref{cohomtoanalsection2}}. 
\begin{theorem}
\label{lastsectiontheo}
 Let $\Gamma$ be a discrete cocompact subgroup of $\mathrm{PSU}(1, 1)$. 
 For every cocycle $c$ representing a cohomology class $[c] \in  
 H^{1}(\Gamma; \mathfrak{g})$, there exists a smooth vector field 
 $\psi$ on the open unit disk $\dd$ 
 such that $c= \delta \psi$.
Moreover, any such $\psi$ admits an 
$L^{2}$-extension to $\overline{\dd}$ whose restriction $\psi^{\sharp}$ to the 
boundary circle $\ss^{1}$ is tangential. 
There exists a homomorphism
\begin{equation}
\label{mainmap2}
 \begin{split}
  \varPsi: H^{1}(\Gamma; \mathfrak{g}) & \longrightarrow 
  \mathrm{HQD}(\dd, \Gamma) \\
  [c] & \longmapsto 
  \big(\mathcal{L}_{\mathcal{F}(\psi^{\sharp})}\textbf{g}_{\dd} \big)^{(2, 0)},
 \end{split}
\end{equation}
  where the map $\mathcal{F}$ is introduced in (\ref{finals}) 
  and 
  $\mathcal{F}(\psi^{\sharp})$ is a harmonic vector field on the open disk $\dd$. 
 \end{theorem}

\begin{coro}
\label{lastsectioncorol}
$$ \varPhi \circ \varPsi  = \mathrm{Id},$$
where the map $\varPhi$ is constructed in (\ref{mainmap1}) (see 
Corollary \ref{onewaymap}) and the map $\varPsi$ in (\ref{mainmap2}) (see Theorem 
\ref{lastsectiontheo}). 
\end{coro}
\begin{proof}[Sketch]
Recall from Corollary \ref{shortcoro} (and \textbf{\cref{cohomtoanalsection1}}) 
that given a 
cocycle $c$ representing a cohomology class $[c] \in  
 H^{1}(\Gamma; \mathfrak{g})$, there exists a smooth vector field 
 $\psi$ on the open unit disk $\dd$ 
 such that $c= \delta \psi$ and $\psi$ admits a unique $L^{2}$-extension to 
 $\overline{\dd}$. We denote the restriction of that extension 
 to the boundary circle $\ss^{1}$ by $\psi^{\sharp}$, and $\psi^{\sharp}$ is tangential. 
 Note that 
 $\delta \psi^{\sharp} = c^{\sharp}$, where $c^{\sharp}$ is a $1$-cocycle (determined 
 by $c$) with 
 values in the vector space of Killing vector fields on $\ss^{1}$. 
 The map $\mathcal{F}$ maps 
Killing vector fields on $\ss^{1}$ to Killing vector fields on the open unit disk 
$\dd$.
 Therefore, it 
 is clear that $\delta \mathcal{F}(\psi^{\sharp}) =c$ and 
 $\mathcal{F}(\delta \psi^{\sharp}(\gamma)) = c^{\sharp}(\gamma)$, for every 
 $\gamma \in \Gamma$.  \hfill \qedsymbol
\end{proof} 
\subsection{Open Problems}
We state the following non-exhaustive list of
open problems based on this section: 
\medskip

From Corollary \ref{shortcoro} we know that  
 given a 
cocycle $c$ representing a cohomology class $[c] \in  
 H^{1}(\Gamma; \mathfrak{g})$, there exists a smooth vector field 
 $\psi$ on the open unit disk $\dd$ 
 such that $c= \delta \psi$ and $\psi$ admits a unique $L^{2}$-extension to the closed 
 unit disk $\overline{\dd}$
 whose restriction $\psi^{\sharp}$ to the 
boundary circle $\ss^{1}$ is tangential. 
 For the construction of $\psi$ we can either use the $\Gamma$-invariant partition 
 of unity method or the difficult theory of \textbf{\cref{chapter3}} and 
 \textbf{\cref{analtocohom}} which produces a harmonic solution. 
 Corollary \ref{shortcoro} 
 is valid for all of these but the construction of an $L^{2}$-extension of $\psi$ 
 to $\overline{\dd}$ 
 relies on the existence of harmonic vector fields. 
 \begin{problem}
\label{openprob1}
  Is there 
 a more direct way of proving Corollary \ref{shortcoro} which does not 
 take harmonicity into account? 
\end{problem}
\medskip

In \textbf{\cref{invariancepoisson}}, 
we have not shown that there exists a \textit{unique} harmonic
extension of a tangential $L^{2}$-vector field $X$ on $\ss^{1}$ to the closed unit 
disk $\overline{\dd}$. 
\begin{problem}
\label{openprob2}
Given a tangential $L^{2}$-vector field $X$ on the boundary circle $\ss^{1}$, 
 does there exist a unique harmonic extension to the closed unit disk
$\overline{\dd}$? 
\end{problem}
 
\section{Application: a connection on the universal Teichm\"uller curve}
\label{connectionuniversal}
\begin{defn}[Ehresmann's definition]
\label{connectionnew}
 \normalfont
 A connection 
on a smooth fiber bundle $f: E \longrightarrow M$ is a smooth vector subbundle 
$T_{h}E$ - the horizontal tangent bundle - 
of the tangent bundle $TE \longrightarrow E$ 
such that $T_{h}E \oplus T_{v}E = TE$.
\end{defn}
\begin{remark}
 \normalfont
 According to \cite[Note 1, Page 287]{kob}, 
Definition \ref{connectionnew} of a connection on a smooth fibre bundle 
$f: E \longrightarrow M$
is given for the first time in \cite{ehres}. 
\end{remark}

\begin{remark}
\label{connectionnew1}
 \normalfont
Equivalently, 
a connection on a smooth fiber bundle $f: E \longrightarrow M$ is a smooth 
vector bundle homomorphism $f^{\ast}TM \longrightarrow TE$ such that the 
composition
$$f^{\ast}TM \longrightarrow TE \longrightarrow TE/T_{v}E $$
is identity. 
\end{remark}

\begin{remark}
\label{remarkconnectionnew}
 \normalfont
 Consequently, the difference of two connections on a smooth fiber bundle 
 $f: E \longrightarrow M$ 
 is a vector bundle homomorphism $f^{\ast}TM \longrightarrow T_{v}E$. In particular, 
 if $E$ is a product, i.e., $E= M \times F$, and $f$ is the projection, then there 
 is a preferred choice of connection and any other connection on this 
 trivial bundle is described by a vector field on the fiber $F$ 
 for every tangent vector $X \in T_{p}M$, where $p \in M$. 
\end{remark}
\medskip
With the help of Definition 
\ref{connectionnew} and Remark \ref{remarkconnectionnew}, we will describe a 
connection on the \textit{universal Teichm\"uller curve}. 
Associated to the $\mathrm{PSL}(2, \rr)$-principal bundle 
$$ \mathrm{Hom}_{0}(\Gamma_{g}, \mathrm{PSL}(2, \rr)) \longrightarrow 
\mathrm{Hom}_{0}(\Gamma_{g}, \mathrm{PSL}(2, \rr)) / \mathrm{PSL}(2, \rr)$$
we have the following smooth fiber bundle \footnote{Actually, $\pi$ is a proper submersion and from the Ehresmann fibration theorem, it 
 follows that $\pi$ is a smooth fiber bundle.} known as the 
 universal Teichm\"uller curve
 \begin{equation}
  \label{universalbundle}
  \pi: \Gamma_{g} \big \backslash \mathrm{Hom}_{0}(\Gamma_{g}, \mathrm{PSL}(2, \rr)) 
\times_{\mathrm{PSL}(2, \rr)} \hh \longrightarrow 
 \mathrm{Hom}_{0}(\Gamma_{g}, \mathrm{PSL}(2, \rr)) / \mathrm{PSL}(2, \rr)
 \end{equation}
 with 
the fiberwise $\Gamma_{g}$-action (which is free and proper), given by the following map 
 $$\gamma \cdot [\rho, z] \longmapsto [\rho, \rho(\gamma)z], \quad \forall \gamma 
 \in \Gamma_{g}.$$ 
The kernel of the map $d\pi$ gives us a line bundle over the total space 
of the bundle $\pi$ given in (\ref{universalbundle}).  
To make the process of describing a connection on the universal Teichm\"uller curve more clear 
and digestible to the reader, we first restrict our attention to the trivial bundle 
$$ \mathrm{Hom}_{0}(\Gamma_{g}, \mathrm{PSL}(2, \rr)) 
\times \hh \longrightarrow 
 \mathrm{Hom}_{0}(\Gamma_{g}, \mathrm{PSL}(2, \rr)).$$
From Remark \ref{remarkconnectionnew}, we know that to describe a connection on 
 the trivial bundle 
 \begin{equation}
  \label{trivialbundle}
  \mathrm{Hom}_{0}(\Gamma_{g}, \mathrm{PSL}(2, \rr)) 
\times \hh \longrightarrow 
 \mathrm{Hom}_{0}(\Gamma_{g}, \mathrm{PSL}(2, \rr))
 \end{equation}
 we need to describe a vector field $\mathfrak{Y}$ on $\hh$ for every $1$-cocycle 
 $$c \in 
 T_{\rho}(\mathrm{Hom}_{0}(\Gamma_{g}, \mathrm{PSL}(2, \rr))) \cong 
 Z^{1}(\Gamma_{g}; \mathfrak{g}_{\mathrm{Ad}\rho}).$$ 
 But there is more to it than meets the eye.  
 We need to describe a connection that respects the $\Gamma_{g}$-action on each fiber 
 of the bundle given in (\ref{trivialbundle}). So, this condition translates 
 to the following condition on $\mathfrak{Y}$ 
 $$\delta \mathfrak{Y} =c, $$
 where $\delta$ is the coboundary operator. 
 \medskip
 
 Therefore, the description of a connection on the universal Teichm\"uller curve given in 
 (\ref{universalbundle}) is 
 equivalent to the description of a vector field  $\mathfrak{Y}$ on $\hh$ (or on $\dd$) 
 for every $c \in T_{\rho}(\mathrm{Hom}_{0}(\Gamma_{g}, \mathrm{PSL}(2, \rr)))$ 
 representing the 
 cohomology class $[c] \in T_{[\rho]}
 (\mathrm{Hom}_{0}(\Gamma_{g}, \mathrm{PSL}(2, \rr)) / \mathrm{PSL}(2, \rr))$ such that 
 $$\delta \mathfrak{Y} =c.$$ 
 We choose $\mathfrak{Y}$ so that it is the unique ``harmonic'' vector field on $\dd$ 
 satisfying $\delta \mathfrak{Y} =c$. See \textbf{\cref{chapter3section2}} 
 and \textbf{\cref{analtocohom}}. Note that the connection on the trivial bundle given in 
 (\ref{trivialbundle}) so constructed is not only invariant under the action of $
 \Gamma_{g}$, but also under the action of $\mathrm{PSL}(2, \rr)$. Also, 
 this connection on the universal Teichm\"uller curve is identical with the one 
 proposed by S. Wolpert in \cite[Section 5]{wolpert1}. The reasons for this agreement 
 are given in \cref{genesis}. 
\begin{problem}
 \normalfont
 The Chern form of the vertical bundle $\mathrm{ker}d\pi$ is calculated by 
 S. Wolpert using 
 the connection $1$-form and curvature $2$-form for a smooth metric on the 
 vertical bundle $\mathrm{ker}d\pi$ (\cite[Section 4 \& 5]{wolpert1}). 
  An obvious question would be whether any of the forms in Chern forms and Riemann 
  tensor would be \textit{harmonic} (in the sense of Hodge theory) w.r.t the Weil-Petersson metric (\cite{wolpert2}, 
  \cite{yamada}).  
  If not, what are obstructions for Chern class forms to be harmonic? 
\end{problem}
\begin{remark}
 \normalfont
 This open problem arises in an email correspondence between the author and S. Wolpert.  
\end{remark}

\newpage
\appendix
 \section{The genesis of the potential equation $F_{\bar{z}}=(z-\bar{z})^{2}\overline{\phi(z)}$ }
 \label{genesis}
\subsection{A swift introduction to Beltrami differentials}
Let 
 $(V, J_{V}), (W, J_{W})$ be complex vector spaces which we treat as real vector 
 spaces with linear operators $J_{V}$ and $J_{W}$ such that 
 $J_{V}^{2}=J_{W}^{2}=-\mathrm{Id}$.  
A $\rr$-linear map 
$$f: (V, J_{V}) \longrightarrow (W, J_{W})$$
can be written uniquely as a sum of $\cc$-linear map $f_{1}$ and $\cc$-antilinear map $f_{2}$, i.e., 
$$f_{1} \circ J_{V} = J_{W} \circ f_{1}, \quad  f_{2} \circ J_{V} = -J_{W} \circ f_{2}.$$

\begin{defn}
\label{beltramiform}
 \normalfont
 Given an invertible $\rr$-linear map which is orientation preserving 
 $$f: (V, J_{V}) \longrightarrow (W, J_{W})$$
 of complex 
 vector spaces, the \textit{Beltrami form} of $f$ is the map 
 \begin{equation}
 \label{beltramiequation1}
  \mu(f):= f_{1}^{-1} \circ f_{2} \in \mathrm{End}_{\rr}((V, J_{V})).
 \end{equation}
\end{defn}
\begin{remark}
 \normalfont
 $\mu(f)$ anticommutes with $J_{V}$. 
\end{remark} 
\smallskip
Now, we will restrict our 
discussion to one dimensional complex vector spaces.
Any $\rr$-linear map $f: (\cc, \iota) \longrightarrow (\cc, \iota)$ can be 
written as $f(z)=az+b\bar{z}$, $a, b, z \in \cc$. Here $f_{1}(z)=az$ and $f_{2}(z)=b\bar{z}$.
From 
\cite[Exercise 4.8.5, Chapter 4]{hub}, we have 
$$ \frac{\lVert f \rVert^{2}}{\mathrm{det}f} = \frac{|f_{1}|+|f_{2}|}{|f_{1}|-|f_{2}|},$$
where $\lVert \cdot \rVert$ denotes the operator norm on the vector 
space of $\rr$-linear maps 
$(\cc, \iota) \longrightarrow (\cc, \iota)$, and $| \cdot |$ denotes the 
operator norm on the vector space of $\cc$-linear maps 
$(\cc, \iota) \longrightarrow (\cc, \iota)$ and $\cc$-antilinear maps 
$(\cc, \iota) \longrightarrow (\cc, \iota)$. Moreover, if $|a| > |b|$, then the map 
$f$ is orientation-preserving. Hence, it immediately follows that
$\lVert \mu(f) \rVert < 1$. 
The space of all Beltrami forms on $(\cc, \iota)$ is defined as follows: 
$$\mathrm{Bel}(\cc):= 
\{\mu \in \mathrm{End}_{\rr}(\cc)| \exists c \in \cc, |c|<1,  \mu(z)=c \bar{z} \}.$$

Now, we ask the following question: given $\mu \in \mathrm{Bel}(\cc)$, how do we find 
an orientation preserving $f: \cc \longrightarrow \cc$ with $\mu(f)=\mu$? 
The equation $\mu(f)=\mu$ is famously known as the  \textit{Beltrami equation}. 
The most sophisticated answer to the above question is that $f$ solves $\mu(f)=\mu$ iff 
$f$ maps an ellipse in $\cc$ whose ratio of the major to the minor axis is 
$\frac{1+\lVert \mu \rVert}{1-\lVert \mu \rVert}$ to a circle in $\cc$. 
Let's discuss 
how the above discussion translates to the case of Riemann surfaces $X$ and $Y$
and an orientation 
preserving $C^{1}$ map $f: X \longrightarrow Y$ between them. Note that 
$df(x): T_{x}X \longrightarrow 
T_{f(x)}Y$ can be written as a sum of a $\cc$-linear map and a $\cc$-antilinear 
map. For example, when $f: U \subset \cc \longrightarrow \cc$, we have 
$df = df^{(1, 0)} + df^{(0, 1)}$ (see the discussion just before Example \ref{harmexamplenice}), 
where $df^{(1, 0)}= \frac{\partial f}{\partial z} dz$ and 
$df^{(0, 1)} = \frac{\partial f}{\partial \bar{z}} d\bar{z}$. For a function 
$f: X \longrightarrow \cc$, 
\begin{equation}
\label{beltramiequation2}
 \mu(f) = (df^{(1, 0)})^{-1} \circ df^{(0, 1)}.
\end{equation}
Compare (\ref{beltramiequation2}) it with (\ref{beltramiequation1}) 
given in Definition \ref{beltramiform}. 
$\mu(f)$ is an antilinear bundle map $TX \longrightarrow TX$. 
\begin{defn}
\label{belt}
\normalfont
 A smooth \textit{Beltrami differential} on $X$ is a smooth antilinear bundle map 
 $\mu: TX \longrightarrow TX$. 
\end{defn}
\begin{remark}
\label{belt1}
 \normalfont
 We can think of a Beltrami differential $\mu$ as 
 a smooth section of the bundle 
 $\overline{T^{\ast}X} \otimes_{\cc} TX$. 
\end{remark}
\subsection{Filling in the gap} 
\label{fillingingap}
Let $\Sigma_{g}$ be given as $\hh/\Gamma$ where $\Gamma$ is a discrete-cocompact subgroup of 
$\mathrm{PSL}(2, \rr)$. Given a $\Gamma$-invariant Beltrami differential $\mu$ on $\hh$ 
with $\lVert \mu \rVert < 1$, 
there exists a smooth map $f: \hh \longrightarrow \hh$ such that 
$f_{\bar{z}} = \mu f_{z}$ 
(see \cite{ahlbers}, \cite{ahlfors111}, \cite{ahlfors22}, \cite{hub}, \cite{imayoshi}).
For $t$ real and small, $\{f^{t\mu}\}$ denotes the family of 
smooth maps determined by the Beltrami differential $t\mu$. Then the 
\textit{deformation vector field} 
$$F := \frac{d}{dt}f^{t\mu}|_{t=0}$$ on $\hh$ satisfies 
the famous potential equation 
$$F_{\bar{z}}\frac{d\bar{z}}{dz}=\mu.$$ 

Let 
$T\Sigma_{g}$ denote the holomorphic tangent bundle of $\Sigma_{g}$. 
Recall Definition \ref{belt} and Remark \ref{belt1}. 
A Beltrami differential $\mu$ can be thought of as a smooth differential form 
on $\Sigma_{g}$
of type $(0,1)$ with values in the bundle $T\Sigma_{g}$. 
In classical Teichmueller theory, we have 
the following short exact sequence of sheaves:
 \begin{equation}
 \label{shortex}
  \xymatrix{
  0 \ar[r] & \mathcal{S}_{\mathrm{Hol}}\big(T\Sigma_{g}\big) \ar[r]^-{i} 
  & \mathcal{S}(T\Sigma_{g})
  \ar[r]^-{\frac{\partial}{\partial \bar{z}}} & \mathcal{BEL} \ar[r] & 0
  }
 \end{equation}
where
\newline
$$\mathcal{S}_{\mathrm{Hol}}\big(T\Sigma_{g}\big) \text{ is the 
sheaf of holomorphic sections of} \hspace{2pt} T\Sigma_{g} 
\hspace{1pt} \text{on} \hspace{2pt} \Sigma_{g},$$
$$\mathcal{S}(T\Sigma_{g}) \text{is the 
sheaf of smooth sections of} \hspace{1pt} T\Sigma_{g},$$ 
$$\mathcal{BEL}  \hspace{1pt} 
\text{is the sheaf of (smooth) Beltrami differentials on} \hspace{1pt} \Sigma_{g}.$$ 

Clearly, $i$ is the inclusion map.
Locally, a smooth section of $T\Sigma_{g}$ can be written as 
$f_{i} \frac{\partial}{\partial z_{i}}$, where $f_{i}$ is a smooth function.  
Applying $\frac{\partial}{\partial \bar{z}}$ on 
$f_{i}\frac{\partial}{\partial z_{i}}$ gives 
us a Beltrami differential 
$\frac{\partial f_{i}}{\partial \bar{z}} d\bar{z}_{i} \otimes 
\frac{\partial}{\partial z_{i}}$. This definition of $\frac{\partial}{\partial \bar{z}}$
is independent of the choice of coordinates.  
Note that (\ref{shortex}) is a special case of a 
more general construction \big(called the  
\textit{Dolbeault resolution} of the sheaf 
$\mathcal{S}_{\mathrm{Hol}}\big(T\Sigma_{g}\big)$\big). 
See \cite[Chapter 4]{voisin} for more details. In particular, 
the map $\frac{\partial}{\partial \bar{z}}$ in (\ref{shortex}) contributes to a 
long exact sequence in sheaf cohomology. Therefore, we have 
$$H^{1}\big(\Sigma_{g}; \mathcal{S}_{\mathrm{Hol}}\big(T\Sigma_{g}\big)\big) \cong 
\frac{H^{0}(\Sigma_{g}; \mathcal{BEL})}{
\overline{\partial}H^{0}\big(\Sigma_{g}; \mathcal{S}
  \big(T\Sigma_{g}\big)\big)}. $$ 
  It is a well known fact that for $\Sigma_{g}$ with a hyperbolic metric, a 
cohomology class in 
$$\frac{H^{0}(\Sigma_{g}; \mathcal{BEL})}{
\overline{\partial}H^{0}\big(\Sigma_{g}; \mathcal{S}
  \big(T\Sigma_{g}\big)\big)}$$ has a unique 
  representative known as a 
\textit{harmonic Beltrami differential}, i.e.,
a Beltrami differential which is annihilated by the appropriate Laplacian 
(see \cite[Chapter 5]{voisin}). 
In the famous potential equation 
$$F_{\bar{z}} = 
(z-\bar{z})^{2} \overline{f(z)},$$ 
\begin{equation}
\label{lastequation}
 \mu = (z-\bar{z})^{2} \overline{f(z)} \frac{d\bar{z}}{dz}
\end{equation}
is a harmonic Beltrami differential if $f$ is holomorphic. 
\smallskip

For $\hh$, we have the following sequence: 
\begin{equation}
 \label{shortex2}
 \xymatrix{
  0 \ar[r] & \mathcal{S}_{\mathrm{Hol}}\big(T \mathbb{H}^{2}\big) 
  \ar[r]^-{i} 
  & \mathcal{S} \big(T \mathbb{H}^{2}\big)
  \ar[r]^-{\frac{\partial}{\partial \bar{z}}} & \mathcal{BEL} \ar[r] & 0
  }
 \end{equation}
 In (\ref{shortex2}), $\mathcal{BEL}$ is the sheaf of smooth Beltrami differentials 
 on $\hh$. 
Note that the short exact sequence given by 
(\ref{shortexact})
in Thereom \ref{sess}, i.e., 
 \begin{equation}
 \label{shortex1}
  \xymatrix{
  0 \ar[r] & \mathcal{HOL} \ar[r]^-{\alpha} & \mathcal{HARM}
  \ar[r]^-{\beta} & \mathcal{HQD} \ar[r] & 0
  }
 \end{equation}
is similar to the one given by (\ref{shortex2}). In (\ref{shortex1}), 
$\mathcal{HOL}$ is the sheaf of holomorphic vector fields on $\hh$, 
$\mathcal{HARM}$ is 
the sheaf of harmonic vector fields on $\hh$ and $\mathcal{HQD}$ is 
the sheaf of holomorphic quadratic differentials on $\hh$. 
So, the gap is following: 
\begin{qus}
\label{gap}
\nonumber
\normalfont
 How do (\ref{shortex2}) and (\ref{shortex1}) relate? 
\end{qus}
\smallskip

The following diagram fills the gap: 
\begin{equation}
\label{shortex3}
\xymatrix{
 0 \ar[r] & \mathcal{HOL} \ar@{=}[d] \ar[r]^-{\alpha} & 
 \mathcal{HARM} \ar@{^{(}->}[d]^{\varrho_{1}} \ar[r]^-{\beta} & 
\mathcal{HQD} \ar[d]^{\varrho_{2}} \ar[r] & 0 \\
0 \ar[r] & \mathcal{S}_{\mathrm{Hol}}\big(T \mathbb{H}^{2}\big) 
\ar[r]^-{i}
& \mathcal{S}\big(T \mathbb{H}^{2}\big)
  \ar[r]^-{\frac{\partial}{\partial \bar{z}}} & \mathcal{BEL} \ar[r] & 0
}
\end{equation}
In (\ref{shortex3}), $\varrho_{1}$ is clearly the inclusion map because a harmonic 
vector field on $\hh$ is a smooth vector field. And, $\varrho_{2}$ is defined as: 
$$\varrho_{2}(q) = \frac{g_{\hh}^{-1}}{2} \bar{q},$$
where $q$ is a holomorphic quadratic differential on $\hh$. 
Note that $\varrho_{2}(q)$ is a harmonic Beltrami differential 
(see (\ref{lastequation})) on $\hh$. Moreover, $\varrho_{2}$ has a coordinate 
independent meaning. Here is an 
argument: 
recall that a Riemannian metric on an almost complex manifold $M$ 
is the real part of a unique
Hermitian metric on $M$. Therefore, a Riemannian metric on 
$M$ determines an isomorphism 
$$TM \longrightarrow \overline{T^{\ast}M}.$$ That in 
turn determines an isomorphism 
$$
\mathrm{Hom}(\overline{TM}, TM) \longrightarrow \mathrm{Hom}(\overline{TM}, 
\overline{T^{\ast}M}).$$ 
Note that $\mathrm{Hom}(\overline{TM}, 
\overline{T^{\ast}M}) = \mathrm{Hom}(TM, 
T^{\ast}M)$. 
Therefore, the bundle 
of Beltrami differetials on $\hh$ is isomorphic to the 
bundle of quadratic differentials on $\hh$.



\end{document}